\documentclass[11pt,reqno]{amsart}
\usepackage[utf8]{inputenc}
\usepackage{amsmath}
\usepackage{amsfonts}
\usepackage{amssymb}
\usepackage{mathrsfs}
\usepackage[dvipsnames]{xcolor}
\usepackage{amsthm}
\theoremstyle{theorem}
\newtheorem{theorem}{Theorem}[section]
\newtheorem{proposition}[theorem]{Proposition}
\newtheorem{lemma}[theorem]{Lemma}
\newtheorem{condition}[theorem]{Condition}
\newtheorem{remark}[theorem]{Remark}

\numberwithin{equation}{section}

\theoremstyle{plain}
\newtoks\thehProclaim
\newtheorem*{Proclaim}{\the\thehProclaim}

\usepackage[paperwidth=20.5cm,paperheight=29.7cm,width=16cm,height=25.5cm,top=24mm,left=19mm,headsep=11mm]{geometry}

\begin{document}

\title[Homogenization of hyperbolic systems]{On operator error estimates for homogenization of~hyperbolic systems with periodic coefficients}

\author{Yu.~M.~Meshkova}

\thanks{The study was supported by  project of Russian Science Foundation  no. 17-11-01069.}


\date{\today}

\address{Chebyshev Laboratory, St. Petersburg State University, 14th Line V.O., 29b, St.~Petersburg, 199178, Russia}
\email{y.meshkova@spbu.ru,\quad juliavmeshke@yandex.ru}


\begin{abstract}
In $L_2(\mathbb{R}^d;\mathbb{C}^n)$, we consider a selfadjoint matrix strongly elliptic second order differential operator $\mathcal{A}_\varepsilon$, 
$\varepsilon >0$. 
The coefficients of the operator $\mathcal{A}_\varepsilon$ are periodic and depend on $\mathbf{x}/\varepsilon$. 
We study the behavior of the operator $\mathcal{A}_\varepsilon ^{-1/2}\sin (\tau \mathcal{A}_\varepsilon ^{1/2})$, 
$\tau\in\mathbb{R}$, in the small period limit. The principal term of approximation in the $(H^1\rightarrow L_2)$-norm 
for this operator is found. Approximation in the $(H^2\rightarrow H^1)$-operator norm with the correction term taken 
into account is also established. The results are applied to homogenization for the solutions of the nonhomogeneous 
hyperbolic equation $\partial ^2_\tau \mathbf{u}_\varepsilon =-\mathcal{A}_\varepsilon \mathbf{u}_\varepsilon +\mathbf{F}$.
\end{abstract}
\maketitle

\tableofcontents

\section*{Introduction}
The paper is devoted to homogenization of periodic differential operators (DO's). A broad literature is devoted to homogenization theory, see, e.~g., the books \cite{BaPa,BeLP,Sa,ZhKO}. 
We use the spectral approach to homogenization problems based on the Floquet-Bloch theory and the analytic perturbation theory.

\subsection{The class of operators} 
In $L_2(\mathbb{R}^d;\mathbb{C}^n)$, we consider a matrix elliptic second order DO $\mathcal{A}_\varepsilon$ admitting a factorization 
$
\mathcal{A}_\varepsilon=b(\mathbf{D})^*g(\mathbf{x}/\varepsilon)b(\mathbf{D})$, $\varepsilon >0$.
Here $b(\mathbf{D})=\sum _{j=1}^d b_jD_j$ is an $(m\times n)$-matrix-valued first order DO with constant coefficients. Assume that $m\geqslant n$ and that the symbol $b(\boldsymbol{\xi})$ has maximal rank. A periodic $(m\times m)$-matrix-valued function $g(\mathbf{x})$ is such that 
$g(\mathbf{x})>0$; $g, g^{-1}\in L_\infty$. 
The coefficients of the operator $\mathcal{A}_\varepsilon$ oscillate rapidly as $\varepsilon \rightarrow 0$.

\subsection{Operator error estimates for elliptic and parabolic problems}

In a series of papers \cite{BSu,BSu05-1,BSu05,BSu06} by M.~Sh.~Birman and T.~A.~Suslina, an abstract operator-theoretic (spectral) approach to homogenization problems in $\mathbb{R}^d$ was developed. This approach is based on the scaling transformation, the Floquet-Bloch theory, and the analytic perturbation theory.  

A typical homogenization problem is to study the behavior of the solution 
$\mathbf{u}_\varepsilon$ of the equation $\mathcal{A}_\varepsilon \mathbf{u}_\varepsilon +\mathbf{u}_\varepsilon=\mathbf{F}$, where $\mathbf{F}\in L_2(\mathbb{R}^d;\mathbb{C}^n)$, as $\varepsilon\rightarrow 0$. 
It turns out that the solutions 
$\mathbf{u}_\varepsilon$ converge in some sense to the solution 
$\mathbf{u}_0$ of the homogenized equation $\mathcal{A}^0\mathbf{u}_0+\mathbf{u}_0=\mathbf{F}$. Here
\begin{equation*}
\mathcal{A}^0=b(\mathbf{D})^*g^0b(\mathbf{D})
\end{equation*}
is the \textit{effective operator} and $g^0$ is the constant \textit{effective matrix}. The way to construct $g^0$ is well known in homogenization theory.

In \cite{BSu}, it was shown that 
\begin{equation}
\label{eq/ 0}
\Vert \mathbf{u}_\varepsilon -\mathbf{u}_0\Vert _{L_2(\mathbb{R}^d)}\leqslant C\varepsilon \Vert\mathbf{F}\Vert _{L_2(\mathbb{R}^d)}.
\end{equation}
This estimate is order-sharp. The constant 
$C$ is controlled explicitly in terms of the problem data. Inequality \eqref{eq/ 0} means that  the resolvent  $(\mathcal{A}_\varepsilon +I)^{-1}$ converges to the resolvent of the effective operator in the  $L_2(\mathbb{R}^d;\mathbb{C}^n)$-operator norm, as $\varepsilon\rightarrow 0$. Moreover,
\begin{equation*}
\Vert (\mathcal{A}_\varepsilon +I)^{-1}-(\mathcal{A}^0 +I)^{-1}\Vert _{L_2(\mathbb{R}^d)\rightarrow L_2(\mathbb{R}^d)}\leqslant C\varepsilon .
\end{equation*}
Results of this type are called \textit{operator error estimates} in homogenization theory.

In \cite{BSu06}, approximation of the resolvent $(\mathcal{A}_\varepsilon +I)^{-1}$ in the $(L_2\rightarrow H^1)$-operator norm was found:
\begin{equation*}
\Vert (\mathcal{A}_\varepsilon +I)^{-1}-(\mathcal{A}^0 +I)^{-1}-\varepsilon K(\varepsilon)\Vert _{L_2(\mathbb{R}^d)\rightarrow H^1(\mathbb{R}^d)}\leqslant C\varepsilon .
\end{equation*}
Here the \textit{correction term} $K(\varepsilon)$ is taken into account. It contains a rapidly oscillating factor and so depends on  $\varepsilon$. Herewith, $\Vert \varepsilon K(\varepsilon)\Vert _{L_2\rightarrow H^1}=O(1)$. 
In contrast to the traditional corrector of homogenization theory, the operator 
$K(\varepsilon)$ contains an auxiliary smoothing operator $\Pi _\varepsilon $ (see \eqref{Pi eps} below).

To parabolic homogenization problems the spectral approach was applied in 
\cite{Su04,Su07,Su_MMNP}. The principal term of approximation was found in 
\cite{Su04,Su07}:
\begin{equation*}
\Vert e^{-\tau \mathcal{A}_\varepsilon}-e^{-\tau\mathcal{A}^0}\Vert _{L_2(\mathbb{R}^d)\rightarrow L_2(\mathbb{R}^d)}
\leqslant C\varepsilon \tau  ^{-1/2},\quad \tau >0.
\end{equation*}
Approximation with the corrector taken into account was obtained in 
\cite{Su_MMNP}:
\begin{equation*}
\Vert e^{-\tau \mathcal{A}_\varepsilon}-e^{-\tau\mathcal{A}^0}-\varepsilon\mathcal{K}(\varepsilon ,\tau)\Vert _{L_2(\mathbb{R}^d)\rightarrow H^1(\mathbb{R}^d)}\leqslant C\varepsilon (\tau ^{-1}+\tau ^{-1/2}),\quad 0<\varepsilon\leqslant\tau ^{1/2}.
\end{equation*}

Another approach to deriving operator error estimates (the so-called \textit{modified method of first order approximation} or the \textit{shift method})
was suggested by V.~V.~Zhikov \cite{Zh1,Zh2} and developed by V.~V.~Zhikov and S.~E.~Pastukhova \cite{ZhPas}. In these papers the elliptic problems for the operators of acoustics and elasticity theory were studied. To parabolic problems the shift method was applied in 
\cite{ZhPAs_parabol}. Further results of V. V. Zhikov and S. E. Pastukhova are discussed in the recent survey \cite{ZhPasUMN}.

\subsection{Operator error estimates for homogenization of hyperbolic equations and nonstationary Schr\"odinger-type equations}

For elliptic and parabolic problems operator error estimates are well studied. The situation with homogenization of nonstationary Schr\"odinger-type and hyperbolic equations is
different. 
In \cite{BSu08}, the operators $e^{-i\mathcal{A}_\varepsilon}$ and $\cos (\tau \mathcal{A}_\varepsilon ^{1/2})$ were studied. It turned out that for these operators it is impossible to find approximations in the $(L_2\rightarrow L_2)$-norm. Approximations in the $(H^s\rightarrow L_2)$-norms with suitable $s$ were found in 
\cite{BSu08}:
\begin{align}
\label{Schreg intr}
&\Vert e^{-i\tau \mathcal{A}_\varepsilon }-e^{-i\tau\mathcal{A}^0}\Vert _{H^3(\mathbb{R}^d)\rightarrow L_2(\mathbb{R}^d)}\leqslant C\varepsilon(1+\vert \tau\vert ),
\\
\label{cos intr}
&\Vert \cos (\tau\mathcal{A}_\varepsilon ^{1/2})-\cos (\tau (\mathcal{A}^0)^{1/2})\Vert _{H^2(\mathbb{R}^d)\rightarrow L_2(\mathbb{R}^d)}\leqslant C\varepsilon (1+\vert\tau\vert) .
\end{align}
Later T.~A.~Suslina \cite{Su17}, by using the analytic perturbation theory, 
proved that estimate 
\eqref{Schreg intr} cannot be refined with respect to the type of the operator norm. Developing the method of \cite{Su17}, M.~A.~Dorodnyi and T.~A.~Suslina \cite{DSu,DSu2} showed that estimate \eqref{cos intr} is sharp in the same sense. In \cite{DSu,DSu2,Su17}, under some additional assumptions on
the operator, the results \eqref{Schreg intr} and \eqref{cos intr} were  improved with respect to the type of the operator norm. In \cite{BSu08,DSu2}, by virtue of  the identity $\mathcal{A}_\varepsilon ^{-1/2}\sin(\tau\mathcal{A}_\varepsilon ^{1/2})=\int _0^\tau \cos(\widetilde{\tau} \mathcal{A}_\varepsilon ^{1/2})\,d\widetilde{\tau}$ and the similar identity  for the effective operator, 
the estimate
\begin{equation}
\label{Th BSu}
\begin{split}
\Vert  \mathcal{A}_\varepsilon^{-1/2}\sin(\tau \mathcal{A}^{1/2}_\varepsilon)-(\mathcal{A}^0)^{-1/2}\sin (\tau (\mathcal{A}^0)^{1/2})\Vert_{H^2(\mathbb{R}^d)\rightarrow L_2(\mathbb{R}^d)} \leqslant C\varepsilon (1+\vert \tau\vert )^2,
\quad \tau\in\mathbb{R},
\end{split}
\end{equation}
was deduced from \eqref{cos intr} as a (rough) consequence. 
The sharpness of estimate \eqref{Th BSu} with respect to the type of the operator norm was not discussed. Estimates \eqref{cos intr} and \eqref{Th BSu} were applied to homogenization for the solution of the Cauchy problem
\begin{equation}
\label{hyperbolic problem in introduction}
\begin{cases}
\partial ^2_\tau \mathbf{u}_\varepsilon (\mathbf{x},\tau)=-\mathcal{A}_\varepsilon\mathbf{u}_\varepsilon (\mathbf{x},\tau )+\mathbf{F}(\mathbf{x},\tau ),\\
\mathbf{u}_\varepsilon (\mathbf{x},0)=\boldsymbol{\varphi}(\mathbf{x}),\quad\partial _\tau \mathbf{u}_\varepsilon (\mathbf{x},0)=\boldsymbol{\psi}(\mathbf{x}).
\end{cases}
\end{equation}

\subsection{Approximation for the solutions of hyperbolic systems with the correction term taken into account}

Operator error estimates with the correction term for nonstationary equations of Schr\"odinger type and hyperbolic type previously have not been established. 
So, we discuss the known ``classical'' homogenization results that cannot be written in the uniform operator topology. These results concern the operators in a bounded domain  $\mathcal{O}\subset\mathbb{R}^d$. Approximation for the solution of the hyperbolic equation with the zero initial data and a non-zero right-hand side was obtained in  \cite[Chapter 2, Subsec. 3.6]{BeLP}. In \cite{BeLP}, it was shown that the difference of the solution and the first order approximation strongly converges to zero in  $L_2((0,T);H^1(\mathcal{O}))$. The error estimate was not established. The case of zero initial data and non-zero right-hand side was also considered in \cite[Chapter 4, Section 5]{BaPa}. In \cite{BaPa}, the complete asymptotic expansion of the solution was constructed and the estimate of order  $O(\varepsilon ^{1/2})$ for the difference of the solution and the first order approximation in the $H^1$-norm on the cylinder $\mathcal{O}\times(0,T)$ was obtained. Herewith, the right-hand side was assumed to be  $C^\infty$-smooth.

It is natural to be interested in the approximation with the correction term for the solutions of hyperbolic systems with non-zero initial data, i. e., in approximation of the operator cosine $\cos (\tau \mathcal{A}_\varepsilon ^{1/2})$ in some suitable sense. One could expect the correction term in this case to be of similar structure as for elliptic and parabolic problems.  However, in
 \cite{BrOFMu} it was observed that this is true only for very special class of initial data. 
In the general case, approximation with the corrector was found in 
\cite{BraLe,CaDCoCaMaMarG1}, but  the correction term was non-local because of the dispersion of waves in inhomogeneous media. Dispersion effects for homogenization of the wave equation were discussed in 
\cite{ABriV,ConOrV,ConSaMaBalV} via the Floquet-Bloch theory and the analytic perturbation theory. Operator error estimates have not been obtained.

\subsection{Main results }

\textit{Our goal} is to refine estimate \eqref{Th BSu} with respect to the type of the operator norm without any additional assumptions and to find an approximation for the operator $\mathcal{A}_\varepsilon ^{-1/2}\sin(\tau \mathcal{A}_\varepsilon ^{1/2})$ in the $(H^2\rightarrow H^1)$-norm. We wish to apply the results to problem  \eqref{hyperbolic problem in introduction} with $\boldsymbol{\varphi}=0$ and non-zero $\mathbf{F}$ and $\boldsymbol{\psi}$.

Our first main result is the estimate
\begin{align}
\label{main result 1}
\begin{split}
\Vert \mathcal{A}_\varepsilon^{-1/2}\sin(\tau \mathcal{A}^{1/2}_\varepsilon)-(\mathcal{A}^0)^{-1/2}\sin (\tau (\mathcal{A}^0)^{1/2})\Vert_{H^1(\mathbb{R}^d)\rightarrow L_2(\mathbb{R}^d)} \leqslant C\varepsilon (1+\vert \tau\vert ),
\quad\varepsilon >0,\quad \tau\in\mathbb{R}.
\end{split}
\end{align}
(Under additional assumptions on the operator, improvement of estimate \eqref{main result 1} with respect to the type of the norm was obtained by M.~A.~Dorodnyi and T.~A.~Suslina in the forthcoming paper \cite{DSu17} that is, actually, major revision of \cite{DSu2}.) 
Our second main result is the approximation
\begin{align}
\label{main result 2}
\begin{split}
\left\Vert \mathcal{A}_\varepsilon ^{-1/2}\sin (\tau \mathcal{A}_\varepsilon ^{1/2})-(\mathcal{A}^0)^{-1/2}\sin (\tau (\mathcal{A}^0)^{1/2})-\varepsilon\mathrm{K}(\varepsilon ,\tau)\right\Vert _{H^2(\mathbb{R}^d)\rightarrow H^1(\mathbb{R}^d)}
\leqslant C \varepsilon(1+\vert\tau\vert),
\end{split}
\end{align}
$\varepsilon >0$, $\tau\in\mathbb{R}$. 
In the general case, the corrector contains the smoothing operator. We distinguish the cases when the smoothing operator can be removed. Also we show that the smoothing operator naturally arising from our method can be replaced by the Steklov smoothing. The latter is more convenient for homogenization problems in a bounded domain. Using of the Steklov smoothing is borrowed from \cite{ZhPas}.

The results are applied to homogenization of the system \eqref{hyperbolic problem in introduction} with $\boldsymbol{\varphi}=0$. 
A more general equation $Q (\mathbf{x}/\varepsilon)\partial ^2_\tau \mathbf{u}_\varepsilon (\mathbf{x},\tau)=-\mathcal{A}_\varepsilon\mathbf{u}_\varepsilon (\mathbf{x},\tau )+Q (\mathbf{x}/\varepsilon) \mathbf{F}(\mathbf{x},\tau )$ is also considered. Here $Q(\mathbf{x})$ is a $\Gamma$-periodic $(n\times n)$-matrix-valued function such that $Q(\mathbf{x})>0$ and $Q,Q^{-1}\in L_\infty$. In Introduction, we discuss only the case $Q=\mathbf{1}_n$  for simplicity.

\subsection{Method} 

We apply the method of \cite{BSu08,DSu2} carrying out all the constructions for the operator $\mathcal{A}_\varepsilon ^{-1/2}\sin (\tau\mathcal{A}_\varepsilon ^{1/2})$.
To obtain the result with the correction term, we borrow some technical tools from 
\cite{Su_MMNP}.  By the \textit{scaling transformation}, inequality \eqref{main result 1} is equivalent to
\begin{equation}
\label{est no eps intr}
\begin{split}
\Bigl\Vert & \left(\mathcal{A}^{-1/2}\sin(\varepsilon ^{-1}\tau\mathcal{A}^{1/2})-(\mathcal{A}^0)^{-1/2}\sin(\varepsilon ^{-1}\tau (\mathcal{A}^0)^{1/2})\right)\varepsilon (-\Delta +\varepsilon ^2 I)^{-1/2}\Bigr\Vert _{L_2(\mathbb{R}^d)\rightarrow L_2(\mathbb{R}^d)}
\\
&\leqslant C(1+\vert\tau\vert),\quad \tau\in\mathbb{R},\quad \varepsilon >0.
\end{split}
\end{equation}
Here $\mathcal{A}=b(\mathbf{D})^*g(\mathbf{x})b(\mathbf{D})$. 
Because of the presence of differentiation in the definition of $H^1$-norm, by the scaling transformation, inequality \eqref{main result 2} reduces to the estimate of order $O(\varepsilon)$:
\begin{equation}
\label{0.8a introduction}
\begin{split}
\Bigl\Vert &\mathbf{D}\left(\mathcal{A}^{-1/2}\sin (\varepsilon ^{-1}\tau \mathcal{A}^{1/2})
-(\mathcal{A}^0)^{-1/2}\sin (\varepsilon ^{-1}\tau (\mathcal{A}^0)^{1/2})
-\mathrm{K}(1,\varepsilon ^{-1}\tau)
\right)
\\
&\times
\varepsilon ^2 (-\Delta +\varepsilon ^2 I)^{-1}\Bigr\Vert _{L_2(\mathbb{R}^d)\rightarrow L_2(\mathbb{R}^d)}
\leqslant C\varepsilon (1+\vert \tau\vert ),\quad \tau \in\mathbb{R},\quad\varepsilon >0.
\end{split}
\end{equation}
For this reason, in estimate \eqref{0.8a introduction}, we use the ,,smoothing operator'' $\varepsilon ^2(-\Delta +\varepsilon ^2 I)^{-1}$ instead of the operator $\varepsilon (-\Delta +\varepsilon ^2 I)^{-1/2}$ which was used in estimate \eqref{est no eps intr} of order $O(1)$. Thus, the principal term of approximation of the operator $\mathcal{A}_\varepsilon ^{-1/2}\sin (\tau\mathcal{A}_\varepsilon ^{1/2})$ is obtained in the $(H^1\rightarrow L_2)$-norm, but approximation in the energy class is given in the $(H^2\rightarrow H^1)$-norm. 

To obtain estimates \eqref{est no eps intr} and \eqref{0.8a introduction}, using  the unitary \textit{Gelfand transformation} (see Section~\ref{Subsec Gelfand} below), we decompose the operator $\mathcal{A}$ into the direct integral of operators $\mathcal{A}(\mathbf{k})$ acting in the space $L_2$ on the cell of periodicity and depending on the parameter $\mathbf{k}\in\mathbb{R}^d$ called the \textit{quasimomentum}. 
We study the family $\mathcal{A}(\mathbf{k})$ by means of the analytic perturbation theory with respect to the onedimensional parameter $\vert\mathbf{k}\vert$. 
Then we should make our
constructions and estimates uniform in the additional parameter $\boldsymbol{\theta}:=\mathbf{k}/\vert \mathbf{k}\vert$. Herewith, a good deal of considerations can be done in the framework of an abstract 
operator-theoretic scheme.

\subsection{Plan of the paper} The paper consists of three chapters. Chapter~{I} (Sec.~\ref{Section Preliminaries}--\ref{Section sandwiched abstract}) contains necessary operator-theoretic material. Chapter~{II} (Sec. \ref{Section Factorized families}--\ref{Section 10}) 
is devoted to periodic DO's. In Sec.~\ref{Section Factorized families}--\ref{Sec eff op}, the class of operators under consideration is introduced, the direct integral decomposition is described, and the effective characteristics are found.  In Sec.~\ref{Section 7 appr A(k)} and \ref{Section 10}, the approximations for the operator-valued function $\mathcal{A} ^{-1/2}\sin(\varepsilon ^{-1}\tau\mathcal{A}^{1/2})$ are obtained and estimates \eqref{est no eps intr} and \eqref{0.8a introduction} are proven. In Chapter~{III} (Sec.~\ref{Section main results in general case} and \ref{Section hyperbolic general case}), homogenization for hyperbolic systems is considered. In Sec.~\ref{Section main results in general case}, the main results of the paper in operator terms (estimates \eqref{main result 1} and \eqref{main result 2}) are obtained. Afterwards, in Sec.~\ref{Section hyperbolic general case}, these results are applied to homogenization for solutions of the nonhomogeneous hyperbolic systems.  Section~\ref{Section Applications} is devoted to applications of the general results to the acoustics equation, the operator of elasticity theory and the model equation of electrodynamics.

\subsection{Acknowledgement} The author is grateful to T.~A.~Suslina for 
 attention to work and numerous comments that helped to improve the quality of presentation.

\subsection{Notation} 
Let $\mathfrak{H}$ and $\mathfrak{H}_*$ be separable Hilbert spaces. The symbols $(\cdot ,\cdot )_\mathfrak{H}$ and $\Vert \cdot \Vert _\mathfrak{H}$ mean the inner product and the norm in 
 $\mathfrak{H}$, respectively; the symbol $\Vert \cdot\Vert _{\mathfrak{H}\rightarrow\mathfrak{H}_*}$ denotes the norm of a~bounded linear operator acting from $\mathfrak{H}$ to $\mathfrak{H}_*$. Sometimes we omit the indices if this does not lead to confusion. By $I=I_\mathfrak{H}$ we denote the identity operator in $\mathfrak{H}$. If $A:\mathfrak{H}\rightarrow\mathfrak{H}_*$ is a linear operator, then $\mathrm{Dom}\,A$ denotes the domain of $A$. If $\mathfrak{N}$ is a subspace of $\mathfrak{H}$, then $\mathfrak{N}^\perp:=\mathfrak{H}\ominus\mathfrak{N}$.
  
The symbol $\langle\cdot,\cdot\rangle$ denotes the inner product in $\mathbb{C}^n$, $\vert \cdot\vert$ means the norm of a vector in $\mathbb{C}^n$; $\mathbf{1}_n$ is the unit matrix of size $n\times n$. If $a$ is an $(m\times n)$-matrix, then $\vert a\vert$ denotes its norm as a linear operator from  $\mathbb{C}^n$ to $\mathbb{C}^m$; $a^*$ means the Hermitian conjugate $(n\times m)$-matrix.

The classes $L_p$ of $\mathbb{C}^n$-valued functions on a domain $\mathcal{O}\subset\mathbb{R}^d$ are denoted by $L_p(\mathcal{O};\mathbb{C}^n)$, $1\leqslant p\leqslant \infty$. The Sobolev spaces of order $s$ of $\mathbb{C}^n$-valued functions on a domain $\mathcal{O}\subset\mathbb{R}^d$  are denoted by $H^s(\mathcal{O};\mathbb{C}^n)$. By $\mathcal{S}(\mathbb{R}^d;\mathbb{C}^n)$ we denote the Schwartz class of $\mathbb{C}^n$-valued functions in $\mathbb{R}^d$.  If $n=1$, then we simply write $L_p(\mathcal{O})$, $H^s(\mathcal{O})$ and so on, but sometimes we use such simplified notation also for the spaces of vector-valued or matrix-valued functions. The symbol $L_p((0,T);\mathfrak{H})$, $1\leqslant p\leqslant\infty$, stands for $L_p$-space of $\mathfrak{H}$-valued functions on the interval $(0,T)$.

Next, $\mathbf{x}=(x_1,\dots,x_d)\in\mathbb{R}^d$, $iD_j=\partial _j=\partial /\partial x_j$, $j=1,\dots,d$, $\mathbf{D}=-i\nabla=(D_1,\dots,D_d)$. The Laplace operator is denoted by $\Delta =\partial ^2/\partial x_1^2+\dots +\partial ^2/\partial x_d^2$.

By $C$, $\mathcal{C}$, $\mathfrak{C}$, $c$, $\mathfrak{c}$ (probably, with indices and marks) we denote various constants in estimates. The absolute constants are denoted by $\beta$ with various indices. 

\section*{Chapter I. Abstract scheme}
\label{Section Abstract sheme}

\section{Preliminaries}
\label{Section Preliminaries}

\subsection{Quadratic operator pencils} 
\label{Subsubsection operator pencils}
Let $\mathfrak{H}$ and $\mathfrak{H}_*$ be separable complex Hilbert spaces. Suppose that $X_0 :\mathfrak{H}\rightarrow \mathfrak{H}_*$ is a densely
defined and closed operator, and that $X_1 :\mathfrak{H}\rightarrow\mathfrak{H}_*$ is a bounded operator. On the domain $\mathrm{Dom}\,X(t)=\mathrm{Dom}\,X_0$, consider the operator $X(t):=X_0+tX_1$, $t\in\mathbb{R}$. \textit{Our main object }is a family of operators
\begin{equation}
\label{A(t)=}
A(t):=X(t)^*X(t),\quad t\in\mathbb{R},
\end{equation}
that are selfadjoint in $\mathfrak{H}$ and non-negative. 
The operator $A(t)$ acting in $\mathfrak{H}$ is generated by the closed quadratic form $\Vert X(t)u\Vert ^2_{\mathfrak{H}_*}$, $u\in\mathrm{Dom}\,X_0$. Denote $A(0)=X_0^*X_0=:A_0$. 
Put $
\mathfrak{N}:=\mathrm{Ker}\,A_0=\mathrm{Ker}\,X_0$, $ \mathfrak{N}_*:=\mathrm{Ker}\,X_0^*$. 
We assume that 
the point $\lambda _0=0$ is  isolated in the spectrum of $A_0$ and 
$
0<n:=\mathrm{dim}\,\mathfrak{N}<\infty$, $n\leqslant n_*:=\mathrm{dim}\,\mathfrak{N}_*\leqslant\infty $. 
By $d_0$ we denote the distance from the point zero to the rest of the spectrum of $A_0$ and by $F(t,s)$ we denote the spectral projection of the operator $A(t)$ for the interval $[0,s]$. Fix $\delta >0$ such that $8\delta <d_0$. Next, we choose a number $t_0>0$ such that
\begin{equation}
\label{t_0(delta) abstact scheme}
t_0\leqslant \delta ^{1/2}\Vert X_1\Vert ^{-1} _{\mathfrak{H}\rightarrow \mathfrak{H}_*}.
\end{equation}
Then (see \cite[Chapter 1, (1.3)]{BSu}) $F(t,\delta )=F(t,3\delta)$ and $\mathrm{rank}\,F(t,\delta)=n$ for $\vert t\vert\leqslant t_0$. We often write $F(t)$ instead of $F(t,\delta )$. Let $P$ and $P_*$ be the orthogonal projections of  $\mathfrak{H}$ onto $\mathfrak{N}$ and of $\mathfrak{H}_*$ onto $\mathfrak{N}_*$, respectively. 

\subsection{Operators $Z$ and $R$}
Let $\mathcal{D}:=\mathrm{Dom}\,X_0\cap \mathfrak{N}^\perp$, and let $u\in\mathfrak{H}_*$. Consider the following equation for the element $\psi\in\mathcal{D}$ (cf. \cite[Chapter 1, (1.7)]{BSu}):
\begin{equation}
\label{first eq. for Z}
X_0^*(X_0\psi - u)=0.
\end{equation}
The equation is understood in the weak sense. In other words, $\psi\in\mathcal{D}$ satisfies the identity
\begin{equation*}
(X_0\psi ,X_0\zeta )_{\mathfrak{H}_*}=(u,X_0\zeta)_{\mathfrak{H}_*},\quad \forall \zeta\in\mathcal{D}.
\end{equation*}
Equation \eqref{first eq. for Z} has a unique solution $\psi$, and $\Vert X_0\psi\Vert_{\mathfrak{H}_*}\leqslant\Vert u\Vert _{\mathfrak{H}_*}$. Now, let $\omega\in\mathfrak{N}$ and $u=-X_1\omega$. The corresponding solution of equation \eqref{first eq. for Z} is denoted by $\psi(\omega)$. We define the bounded operator $Z:\mathfrak{H}\rightarrow\mathfrak{H}$ by the identities
\begin{equation*}
Z\omega =\psi (\omega),\quad\omega\in\mathfrak{N};\quad Zx=0,\quad x\in\mathfrak{N}^\perp .
\end{equation*}
Note that
\begin{equation}
\label{ZP=Z, PZ=0}
ZP=Z,\quad PZ=0.
\end{equation}

Now, we introduce an operator $
R :\mathfrak{N}\rightarrow \mathfrak{N}_*$ (see \cite[Chapter 1, Subsec. 1.2]{BSu}) as follows: 
 $R\omega =X_0\psi(\omega)+X_1\omega\in\mathfrak{N}_*$. 
Another description of $R$ is given by the formula 
$
R=P_*X_1\vert _\mathfrak{N}$. 

\subsection{The spectral germ}

The selfadjoint operator 
$
S:=R^*R :\mathfrak{N}\rightarrow \mathfrak{N}
$ 
is called the 
\textit{spectral germ} of the operator family \eqref{A(t)=} at $t=0$ (see \cite[Chapter 1, Subsec. 1.3]{BSu}). This operator also can be written as $S=PX_1^*P_*X_1\vert _\mathfrak{N}$. So,
\begin{equation}
\label{vert S vert =}
\Vert S\Vert \leqslant \Vert X_1\Vert ^2.
\end{equation}
The spectral germ $S$ is called \textit{nondegenerate,} if $\mathrm{Ker}\,S=\lbrace 0\rbrace$ or, equivalently, $\mathrm{rank}\,R=n$.

In accordance with the analytic perturbation theory 
(see \cite{K}), for $\vert t\vert\leqslant t_0$ there exist
real-analytic functions $\lambda _l(t)$ and real-analytic $\mathfrak{H}$-valued functions $\phi _l(t)$ such that
\begin{equation*}
A(t)\phi _l(t)=\lambda _l(t)\phi _l(t),\quad l=1,\dots,n,\quad\vert t\vert \leqslant t_0,
\end{equation*}
and $\phi _l(t)$, $l=1,\dots,n$, form an orthonormal basis in the eigenspace $F(t)\mathfrak{H}$. For sufficiently small $t_*$ ($\leqslant t_0$) and $\vert t\vert \leqslant t_*$, we have the following convergent power series expansions:
\begin{align}
\label{lambda_l(t)= series}
&\lambda _l(t)=\gamma _l t^2 +\mu _l t^3+\dots ,\quad \gamma _l\geqslant 0,\quad\mu _l\in\mathbb{R},\quad l=1,\dots,n;\\
&\phi _l(t)=\omega _l +t\phi _l^{(1)}+t^2\phi _l^{(2)}+\dots,\quad l=1,\dots,n.
\nonumber
\end{align}
The elements $\omega _l=\phi_l(0)$, $l=1,\dots ,n$, form an orthonormal basis in $\mathfrak{N}$.

In \cite[Chapter 1, Subsec. 1.6]{BSu} it was shown that the numbers $\gamma _l$ and the elements $\omega _l$, $l=1,\dots,n$, are
eigenvalues and eigenvectors of the operator $S$:
\begin{equation}
\label{S omega _l=}
S\omega_l=\gamma_l\omega_l,\quad l=1,\dots, n.
\end{equation}
The numbers $\gamma _l$ and the vectors $\omega _l$, $l=1,\dots ,n$, are called \textit{threshold characteristics at the bottom of the spectrum} of the operator family $A(t)$.

\subsection{Threshold approximations}
We assume that 
\begin{equation}
\label{A(t)>=}
A(t)\geqslant c_*t^2 I,\quad \vert t\vert\leqslant t_0,
\end{equation}
for some $c_*>0$. 
This is equivalent to the following estimates for the eigenvalues $\lambda _l(t)$ of the operator $A(t)$: 
$
\lambda_l(t)\geqslant c_* t^2$, $\vert t\vert\leqslant t_0$, $l=1,\dots,n$.  
Taking \eqref{lambda_l(t)= series} into account, we see that $\gamma _l\geqslant c_*$, $l=1,\dots,n$.  So, by \eqref{S omega _l=}, the germ $S$ is nondegenerate:
\begin{equation}
\label{S>=}
S\geqslant c_* I_\mathfrak{N} .
\end{equation}

As was shown in \cite[Chapter 1, Theorem 4.1]{BSu},
\begin{equation}
\label{F-P}
\Vert F(t)-P\Vert \leqslant C_1 \vert t\vert,\quad \vert t\vert\leqslant t_0;\quad C_1:=\beta_1 \delta ^{-1/2}\Vert X_1\Vert .
\end{equation}
Besides \eqref{F-P}, we need more accurate approximation of the spectral projection obtained in \cite[(2.10) and (2.15)]{BSu05-1}:
\begin{equation}
\label{F(t)=P+tF_1+F_2(t)}
F(t)=P+tF_1+F_2(t),\quad \Vert F_2(t)\Vert\leqslant C_2 t^2,\quad \vert t\vert\leqslant t_0;\quad C_2:= \beta _2 \delta ^{-1}\Vert X_1\Vert ^2;
\end{equation}
where
\begin{equation}
\label{F_1=}
F_1=ZP+PZ^*.
\end{equation}
From \eqref{ZP=Z, PZ=0} and \eqref{F_1=} it follows that
\begin{equation}
\label{F1P=ZP}
F_1P=ZP.
\end{equation}

In \cite[Chapter 1, Theorem 5.2]{BSu}, it was proven that
\begin{align}
\label{Th res BSu-108}
&\Vert (A(t)+\zeta I)^{-1}F(t)-(t^2SP+\zeta I)^{-1}P\Vert \leqslant C_3\vert t\vert (c_* t^2 +\zeta )^{-1},\quad \zeta >0,\quad\vert t\vert \leqslant t_0;
\\
\label{C_3 abstract}
&C_3:=\beta _3\delta ^{-1/2}\Vert X_1\Vert(1 +c_*^{-1} \Vert X_1\Vert ^2).
\end{align}

According to \cite[Theorem 2.4]{BSu08}, we have
\begin{align}
\label{Th A(t)+1/2}
&\Vert A(t)^{1/2}F(t)-(t^2 S)^{1/2}P\Vert \leqslant C_4t^2,\quad \vert t\vert \leqslant t_0;
\\
\label{C_4 abstract}
&C_4:=\beta _4 \delta ^{-1/2}\Vert X_1\Vert ^2(1+c_*^{-1/2}\Vert X_1\Vert ) .
\end{align}
Combining this with \eqref{vert S vert =}, we see that
\begin{equation}
\label{A(t)^1/2F<=}
\Vert A(t)^{1/2}F(t)\Vert \leqslant \vert t\vert \Vert S\Vert ^{1/2}+C_4 t^2\leqslant (\Vert X_1\Vert +C_4 t_0)\vert t\vert ,\quad \vert t\vert \leqslant t_0.
\end{equation}

We also need the following estimate for the operator $A(t)^{1/2}F_2(t)$ obtained in \cite[(2.23)]{BSu06}:
\begin{equation}
\label{A^1/2F_2}
\Vert A(t)^{1/2}F_2(t)\Vert _{\mathfrak{H}\rightarrow \mathfrak{H}}\leqslant C_5t^2,\quad\vert t\vert \leqslant t_0;\quad C_5:=\beta _5\delta ^{-1/2}\Vert X_1\Vert ^2.
\end{equation}

\subsection{Approximation of the operator $A(t)^{-1/2}F(t)$ for $t\neq 0$}

\begin{lemma}
For $\vert t\vert \leqslant t_0$ and $t\neq 0$ we have
\begin{equation}
\label{Lm A-1/2}
\Vert A(t)^{-1/2}F(t)-(t^2 S)^{-1/2} P\Vert \leqslant C_6.
\end{equation}
The constant $C_6$ is defined below in \eqref{C_6 abstract} and depends only on $\delta$, $\Vert X_1\Vert $, and $c_*$.
\end{lemma}

\begin{proof}
We have
\begin{equation}
\label{A-1/2 tozd}
A(t)^{-1/2}F(t)=\frac{1}{\pi}\int _0^\infty \zeta ^{-1/2}(A(t)+\zeta I)^{-1}F(t)\,d\zeta ,\quad t\neq 0.
\end{equation}
(See, e.~g., \cite[Chapter III, Section 3, Subsection 4]{ViKr}). Similarly,
\begin{equation}
\label{germ -1/2 tozd}
(t^2 S)^{-1/2}P=\frac{1}{\pi}\int _0^\infty \zeta ^{-1/2}(t^2 S+\zeta I_\mathfrak{N} )^{-1}P\,d\zeta
=\frac{1}{\pi}\int _0^\infty \zeta ^{-1/2}(t^2 SP+\zeta I )^{-1}P\,d\zeta .
\end{equation}
Subtracting \eqref{germ -1/2 tozd} from \eqref{A-1/2 tozd}, using \eqref{Th res BSu-108}, and changing the variable $\widetilde{\zeta}:=(c_* t^2)^{-1}\zeta$, we obtain
\begin{equation*}
\begin{split}
\Vert A(t)^{-1/2}F(t)-(t^2 S)^{-1/2}P\Vert 
&\leqslant \frac{C_3}{\pi}\int _0^\infty \zeta ^{-1/2}\vert t\vert (c_* t^2 +\zeta )^{-1}\,d\zeta
=\frac{C_3}{\pi} c_*^{-1/2}\int _0^\infty \widetilde{\zeta}^{-1/2}(1+\widetilde{\zeta})^{-1}\,d\widetilde{\zeta}
\\
&\leqslant \frac{C_3}{\pi} c_*^{-1/2}\left(\int _0 ^1 \widetilde{\zeta}^{-1/2}\,d\widetilde{\zeta}
+\int _1^\infty \widetilde{\zeta}^{-3/2}\,d\widetilde{\zeta}\right)
= 4 \pi ^{-1}c_*^{-1/2}C_3.
\end{split}
\end{equation*}
We arrive at estimate 
\eqref{Lm A-1/2} with the constant
\begin{equation}
\label{C_6 abstract}
C_6:=4 \pi ^{-1}c_*^{-1/2}C_3.
\end{equation}
\end{proof}

\section{Approximation of the operator $A(t)^{-1/2}\sin (\tau A(t)^{1/2})$}
\label{Subsection abstract approximations for sin}

\subsection{The principal term of approximation}

\begin{proposition}
\label{Proposition abstract sin tau principal}
For $\vert t\vert \leqslant t_0$ and $\tau\in\mathbb{R}$ we have
\begin{equation}
\label{A-1/2sin F}
\begin{split}
&\left\Vert \left( A(t)^{-1/2}\sin (\tau A(t)^{1/2})-(t^2 S)^{-1/2}\sin (\tau (t^2S)^{1/2}P)\right)P\right \Vert 
\leqslant C_7(1+\vert\tau\vert\vert t\vert )
.
\end{split}
\end{equation}
The constant $C_7$ depends only on $\delta$, $\Vert X_1\Vert $, and $c_*$.
\end{proposition}

\begin{proof}
For $t=0$ the operator under the norm sign in \eqref{A-1/2sin F} is understood as a limit for $t\rightarrow 0$. Using the Taylor series expansion for the sine function, we see that this limit is equal to zero.

Now, let $t\neq 0$. We put
\begin{align}
\label{E(tau)}
&E(\tau):=e^{-i\tau A(t)^{1/2}}A(t)^{-1/2}F(t)-e^{-i\tau (t^2 S)^{1/2}P}(t^2 S)^{-1/2}P;
\\
\label{Sigma(tau)}
&\Sigma (\tau):=e^{i\tau (t^2 S)^{1/2}P}E(\tau)=e^{i\tau (t^2 S)^{1/2}P}e^{-i\tau A(t)^{1/2}}A(t)^{-1/2}F(t)-(t^2S)^{-1/2}P.
\end{align}
Then
\begin{equation}
\label{2.4a}
\Sigma (0)=A(t)^{-1/2}F(t)-(t^2 S)^{-1/2}P
\end{equation}
and
\begin{equation}
\label{d Sigma /d tau}
\frac{d\Sigma (\tau)}{d\tau}=ie^{i\tau (t^2 S)^{1/2}P}\left( (t^2S)^{1/2}P-A(t)^{1/2}F(t)\right)e^{-i\tau A(t)^{1/2}}A(t)^{-1/2}F(t).
\end{equation}
By \eqref{A(t)>=} and \eqref{Th A(t)+1/2}, the operator-valued function \eqref{d Sigma /d tau} satisfies the following estimate:
\begin{equation}
\label{2.5a very new}
\left\Vert \frac{d\Sigma (\tau)}{d\tau}\right\Vert \leqslant C_4 t^2\Vert A(t)^{-1/2}\Vert \leqslant C_4 c_*^{-1/2}\vert t\vert,\quad \vert t\vert \leqslant t_0,\quad t\neq 0.
\end{equation}
Then, taking \eqref{Lm A-1/2},  \eqref{Sigma(tau)}, \eqref{2.4a}, and \eqref{2.5a very new} into account, we see that
\begin{align}
\label{E(tau) estimate}
&\Vert E(\tau)\Vert =\Vert \Sigma(\tau)\Vert \leqslant C_4 c_*^{-1/2}\vert t\vert  \vert\tau\vert  +\Vert \Sigma(0)\Vert 
\leqslant C_8(1+\vert\tau\vert\vert t\vert),\quad \vert t\vert \leqslant t_0,\quad t\neq 0;
\\
\label{C_7 abstract}
&C_8:=\max\lbrace C_4 c_*^{-1/2};C_6\rbrace .
\end{align} 
(Cf. the proof of Theorem 2.5 from \cite{BSu08}.) So,
\begin{equation}
\label{A-1/2sin F 8}
\Vert A(t)^{-1/2}\sin (\tau A(t)^{1/2})F(t)-(t^2 S)^{-1/2}\sin (\tau (t^2S)^{1/2}P)P\Vert \leqslant C_8(1+\vert\tau\vert \vert t\vert ).
\end{equation}
By virtue of \eqref{A(t)>=} and \eqref{F-P}, from \eqref{A-1/2sin F 8} we derive the inequality
\begin{equation}
\label{A-1/2sin F 2}
\begin{split}
\Bigl\Vert &\left( A(t)^{-1/2}\sin (\tau A(t)^{1/2})-(t^2 S)^{-1/2}\sin (\tau (t^2S)^{1/2}P)\right)P\Bigr \Vert 
\\
& \leqslant C_8(1+\vert\tau\vert\vert t\vert )+\Vert A(t)^{-1/2}\sin (\tau A(t)^{1/2})(F(t)-P)\Vert 
\\
&\leqslant C_7(1+\vert\tau\vert\vert t\vert ),\quad \vert t\vert \leqslant t_0;\quad C_7:=C_8+c_*^{-1/2}C_1
.
\end{split}
\end{equation}
\end{proof}

\subsection{Approximation in the ``energy'' norm}

Now, we obtain another approximation for the operator 
$A(t)^{-1/2}\sin(\tau A(t)^{1/2})$ (in the ``energy'' norm).

\begin{proposition}
\label{Proposition corr sin tau abstr no hat}
For $\tau\in\mathbb{R}$ and $\vert t\vert \leqslant t_0$, we have
\begin{equation}
\label{Th corr abstr}
\left\Vert A(t)^{1/2}\left(
A(t)^{-1/2}\sin (\tau A(t)^{1/2})-(I+tZ)(t^2 S)^{-1/2}\sin (\tau (t^2S)^{1/2}P)\right)P\right\Vert
\leqslant C_9 (\vert t\vert + \vert \tau\vert t^2 ).
\end{equation}
The constant $C_9$ depends only on $\delta$, $\Vert X_1\Vert $, and $c_*$.
\end{proposition}

\begin{proof}
Note that
\begin{equation}
\label{corr abstr start}
\begin{split}
A(t)^{1/2}e^{-i\tau A(t)^{1/2}}A(t)^{-1/2}P
=A(t)^{1/2}e^{-i\tau A(t)^{1/2}}A(t)^{-1/2}F(t)P
+e^{-i\tau A(t)^{1/2}}(P-F(t))P.
\end{split}
\end{equation}
By \eqref{F-P},
\begin{equation}
\Vert e^{-i\tau A(t)^{1/2}}(P-F(t))P\Vert \leqslant C_1\vert t\vert ,\quad\tau\in\mathbb{R},
\quad \vert t\vert \leqslant t_0.
\end{equation}
Next,
\begin{equation}
\begin{split}
A(t)^{1/2}e^{-i\tau A(t)^{1/2}}A(t)^{-1/2}F(t)P=A(t)^{1/2}F(t)e^{-i\tau (t^2S)^{1/2}P}(t^2 S)^{-1/2}P
+A(t)^{1/2}F(t)
E(\tau)
P,
\end{split}
\end{equation}
where $E(\tau)$ is given by \eqref{E(tau)}. 
By \eqref{A(t)^1/2F<=} and \eqref{E(tau) estimate}, for $t\neq 0$ we have
\begin{equation}
\label{corr using principal term t neq 0}
\Bigl\Vert A(t)^{1/2}F(t)
E(\tau)
P\Bigr\Vert
\leqslant C_8(\Vert X_1\Vert +C_4t_0)(\vert t\vert +\vert \tau\vert t^2),
\quad
\tau\in\mathbb{R}, \quad \vert t\vert \leqslant t_0,\quad
 t\neq 0.
\end{equation}
For $t=0$ the operator under the norm sign in \eqref{corr using principal term t neq 0} is understood as a limit for $t\rightarrow 0$. We have $e^{-i\tau A(t)^{1/2}}F(t)\rightarrow P$, as $t\rightarrow 0$. Next, by \eqref{S>=} and \eqref{Th A(t)+1/2},
\begin{equation*}
\begin{split}
\Vert &A(t)^{1/2}F(t)e^{-i\tau (t^2 S)^{1/2}P}(t^2 S)^{-1/2}P -e^{-i\tau (t^2 S)^{1/2}P}P\Vert 
\\
&=\Vert  A(t)^{1/2}F(t)(t^2 S)^{-1/2}P -P\Vert \leqslant c_*^{-1/2}C_4\vert t\vert ,\quad\tau\in\mathbb{R},\quad
 \vert t\vert \leqslant t_0.
\end{split}
\end{equation*}
Using these arguments, we see that the limit of the left-hand side of 
\eqref{corr using principal term t neq 0} as $t\rightarrow 0$ is equal to zero. 

According to \eqref{F(t)=P+tF_1+F_2(t)} and \eqref{F1P=ZP},
\begin{equation}
\begin{split}
A&(t)^{1/2}F(t)e^{-i\tau (t^2 S)^{1/2}P}(t^2 S)^{-1/2}P
-A(t)^{1/2}(I+tZ)e^{-i\tau (t^2 S)^{1/2}P}(t^2 S)^{-1/2}P
\\
&=A(t)^{1/2}F_2(t)e^{-i\tau (t^2S)^{1/2}P}(t^2 S)^{-1/2}P.
\end{split}
\end{equation}
By \eqref{S>=} and \eqref{A^1/2F_2},
\begin{equation}
\label{est corr finish}
\Vert A(t)^{1/2}F_2(t)e^{-i\tau (t^2S)^{1/2}P}(t^2 S)^{-1/2}P\Vert \leqslant c_*^{-1/2}C_5\vert t\vert,\quad \tau\in\mathbb{R},\quad \vert t\vert\leqslant t_0.
\end{equation}
Combining \eqref{corr abstr start}--\eqref{est corr finish}, we arrive at
\begin{equation}
\label{Th corr exp}
\begin{split}
\left\Vert A(t)^{1/2}\left(e^{-i\tau A(t)^{1/2}}A(t)^{-1/2}-(I+tZ)e^{-i\tau (t^2 S)^{1/2}P}(t^2S)^{-1/2}P\right)P\right\Vert
\leqslant C_9(\vert t\vert +\vert \tau\vert t^2),\\
\tau\in\mathbb{R},\quad\vert t\vert \leqslant t_0;\quad C_9:=C_1+c_*^{-1/2}C_5+C_8(\Vert X_1\Vert +C_4t_0).
\end{split}
\end{equation}

(Cf. the proof of Theorem 3.1 from \cite{Su_MMNP}.)
\end{proof}

\subsection{Approximation of the operator $A(t)^{-1/2}\sin (\varepsilon ^{-1}\tau A(t)^{1/2})P$} 
Now, we introduce a parameter $\varepsilon>0$. We need to study the behavior of the operator 
$A(t)^{-1/2}\sin (\varepsilon ^{-1
}\tau A(t)^{1/2})P$ for small $\varepsilon$. Replace $\tau$ by $\varepsilon ^{-1}\tau$ in \eqref{A-1/2sin F}:
\begin{equation*}
\begin{split}
\left\Vert \left( A(t)^{-1/2}\sin (\varepsilon ^{-1}\tau A(t)^{1/2})-(t^2 S)^{-1/2}\sin (\varepsilon ^{-1}\tau (t^2S)^{1/2}P)\right)P\right\Vert 
\leqslant  C_7(1+\varepsilon ^{-1}\vert\tau\vert\vert t\vert ),\\
 \vert t\vert \leqslant t_0,\quad \varepsilon >0,\quad \tau\in\mathbb{R}.
\end{split}
\end{equation*}
Multiplying this inequality by the ``smoothing'' factor 
$\varepsilon (t^2+\varepsilon ^2)^{-1/2}$  and taking into account the inequalities 
$
\varepsilon (t^2+\varepsilon ^2)^{-1/2}\leqslant 1
$ 
and 
$
\vert\tau\vert \vert t\vert (t^2+\varepsilon ^2)^{-1/2}\leqslant \vert\tau\vert $, 
we obtain the following result.

\begin{theorem}
\label{Theorem sin A(t) smoothed}
For $\tau\in\mathbb{R}$, $\varepsilon >0$, and $\vert t\vert\leqslant t_0$ we have
\begin{equation*}
\left\Vert \left( A(t)^{-1/2}\sin (\varepsilon ^{-1}\tau A(t)^{1/2})-(t^2 S)^{-1/2}\sin (\varepsilon ^{-1}\tau (t^2S)^{1/2}P)\right) \varepsilon (t^2+\varepsilon ^2)^{-1/2}P\right\Vert \leqslant C_7(1+\vert\tau\vert ).
\end{equation*}
\end{theorem}
Replacing  $\tau$ by $\varepsilon ^{-1}\tau$ in \eqref{Th corr abstr}  and multiplying the operator by  $\varepsilon ^2 (t^2+\varepsilon ^2)^{-1}$,  we arrive at the following statement.
\begin{theorem}
\label{Theorem sin A(t) smoothed corrector}
For $\tau\in\mathbb{R}$, $\varepsilon >0$, and $\vert t\vert\leqslant t_0$ we have
\begin{equation*}
\begin{split}
\Bigl\Vert & A(t)^{1/2}\left(
A(t)^{-1/2}\sin (\varepsilon ^{-1}\tau A(t)^{1/2})-(I+tZ)(t^2 S)^{-1/2}\sin (\varepsilon ^{-1}\tau (t^2S)^{1/2}P)\right)\varepsilon ^2(t^2+\varepsilon ^2)^{-1} P\Bigr\Vert\\
&\leqslant C_9\varepsilon (1+\vert\tau\vert ).
\end{split}
\end{equation*}
\end{theorem}

\section{Approximation of the sandwiched operator sine}

\label{Section sandwiched abstract}

\subsection{The operator family $A(t)=M^*\widehat{A}(t)M$} 
\label{Subsec hat famimy abstract}
Now, we consider an operator family of the form $ A(t) = M^* \widehat{A}(t)M$ (see \cite[Chapter 1, Subsections
1.5 and 5.3]{BSu}).

Let $\widehat{\mathfrak{H}}$ be yet another separable Hilbert space. Let $\widehat{X}(t)=\widehat{X}_0+t\widehat{X}_1 : \widehat{\mathfrak{H}}\rightarrow \mathfrak{H}_*$ 
be a family of operators of the same form as $X(t)$, and suppose that $\widehat{X
}(t)$ satisfies the assumptions of Subsection~\ref{Subsubsection operator pencils}. 

Let $M : \mathfrak{H}\rightarrow \widehat{\mathfrak{H}} $ be an isomorphism. Suppose that 
$
M\mathrm{Dom}\,X_0 = \mathrm{Dom}\,\widehat{X}_0$; $X_0 = \widehat{X}_0M$; $ X_1 = \widehat{X}_1M$. 
Then $X(t) = \widehat{X}(t)M$. Consider the family of operators
\begin{equation}
\label{hat A(t)}
\widehat{A}(t) = \widehat{X}(t)^*\widehat{X}(t)
: \widehat{\mathfrak{H}}\rightarrow\widehat{\mathfrak{H}}.
\end{equation}
Obviously, 
\begin{equation}
\label{A(t) and hat A(t)}
A(t) = M^*\widehat{A} (t)M. 
\end{equation}
In what follows, all the objects corresponding to the family \eqref{hat A(t)} are supplied with the upper
mark "$\widehat{\phantom{a}} $". Note that $\widehat{\mathfrak{N}} = M\mathfrak{N}$, $\widehat{n} = n$, $\widehat{\mathfrak{N}}_* = \mathfrak{N}_*$, $\widehat{n}_* = n_*$, and $\widehat{P}_* = P_*$.

We denote
\begin{equation}
\label{Q= abstract}
Q:=(MM^*)^{-1}=(M^*)^{-1}M^{-1}:\widehat{\mathfrak{H}}\rightarrow\widehat{\mathfrak{H}}.
\end{equation}
Let $Q_{\widehat{\mathfrak{N}}}$ be the block of $Q$ in the subspace $\widehat{\mathfrak{N}}$: 
$
Q_{\widehat{\mathfrak{N}}}=\widehat{P}Q\vert _{\widehat{\mathfrak{N}}}:\widehat{\mathfrak{N}}\rightarrow\widehat{\mathfrak{N}}$. 
Obviously, $Q_{\widehat{\mathfrak{N}}}$ is an isomorphism in $\widehat{\mathfrak{N}}$. Let $M_0:=\left(Q_{\widehat{\mathfrak{N}}}\right)^{-1/2}: \widehat{\mathfrak{N}}\rightarrow \widehat{\mathfrak{N}}$. As was shown in \cite[Proposition 1.2]{Su07}, the orthogonal projection $P$ of the space $\mathfrak{H}$ onto $\mathfrak{N}$ and the orthogonal projection $\widehat{P}$ of the space $\widehat{\mathfrak{H}}$ onto $\widehat{\mathfrak{N}}$ satisfy the following relation: 
$
P=M^{-1}(Q_{\widehat{\mathfrak{N}}})^{-1}\widehat{P}(M^*)^{-1}$. 
Hence,
\begin{equation}
\label{PM*=}
PM^*=M^{-1}(Q_{\widehat{\mathfrak{N}}})^{-1}\widehat{P}=M^{-1}M_0^2 \widehat{P}.
\end{equation}
According to \cite[Chapter 1, Subsec. 1.5]{BSu}, the spectral germs $S$ and $\widehat{S}$ satisfy
\begin{equation*}
S=PM^*\widehat{S}M\vert _\mathfrak{N}.
\end{equation*}

For the operator family \eqref{hat A(t)} we introduce the operator $\widehat{Z}_Q$ acting in $\widehat{\mathfrak{H}}$ and taking an element $\widehat{u}\in\widehat{\mathfrak{H}}$ to the solution $\widehat{\varphi}_Q$ of the problem
\begin{equation}
\label{hat ZQ equation}
\widehat{X}_0^*(\widehat{X}_0\widehat{\varphi}_Q+\widehat{X}_1\widehat{\omega})=0,\quad Q\widehat{\varphi}_Q \perp \widehat{\mathfrak{N}},
\end{equation}
where $\widehat{\omega}:=\widehat{P}\widehat{u}$. Equation \eqref{hat ZQ equation} is understood in the weak sense. As was shown in \cite[Lemma~6.1]{BSu05-1}, the operator $Z$ for $A(t)$ and the operator $\widehat{Z}_Q$ satisfy 
\begin{equation}
\label{hat Z_Q=}
\widehat{Z}_Q=MZM^{-1}\widehat{P}.
\end{equation}

\subsection{The principal term of approximation for the sandwiched operator $A(t)^{-1/2}\sin(\tau A(t)^{1/2})$}
In this subsection, we find an approximation for the operator $A(t)^{-1/2}\sin(\tau A(t)^{1/2})$, where $A(t)$
is given by \eqref{A(t) and hat A(t)}, in terms of the germ $\widehat{S}$ of $\widehat{A}(t)$ and the isomorphism $M$. It
is convenient to border the operator $A(t)^{-1/2}\sin(\tau A(t)^{1/2})$ by appropriate factors.

\begin{proposition} Suppose that the assumptions of Subsec.~\textnormal{\ref{Subsec hat famimy abstract}} are satisfied. 
Then 
for $\tau\in\mathbb{R}$ and $\vert t\vert \leqslant t_0$ we have
\begin{align}
\label{prop hat sin principal abstr scheme}
\begin{split}
\Vert & MA(t)^{-1/2}\sin (\tau A(t)^{1/2})M^{-1}\widehat{P}
-M_0(t^2M_0\widehat{S}M_0)^{-1/2}\sin (\tau(t^2M_0\widehat{S}M_0)^{1/2})M_0^{-1}\widehat{P}\Vert _{\widehat{\mathfrak{H}}\rightarrow\widehat{\mathfrak{H}}}
\\
&\leqslant C_7\Vert M\Vert \Vert M^{-1}\Vert  (1+\vert\tau\vert \vert t\vert ).
\end{split}
\end{align}
Here $t_0$ is defined according to \eqref{t_0(delta) abstact scheme}, and $C_7$ is the constant from \eqref{A-1/2sin F 2} depending only on $\delta$, $\Vert X_1\Vert $, and $c_*$.
\end{proposition}

\begin{proof}
Estimate \eqref{prop hat sin principal abstr scheme} follows from Proposition~\ref{Proposition abstract sin tau principal} by recalculation. 
In \cite[Proposition 3.3]{BSu08}, it was shown that
\begin{equation}
\label{cos and hat cos identity}
M\cos (\tau (t^2S)^{1/2}P)PM^*
=M_0\cos (\tau (t^2 M_0\widehat{S}M_0)^{1/2})M_0\widehat{P}.
\end{equation}
Obviously,
\begin{equation}
\label{sin not hat = int}
(t^2S)^{-1/2}\sin (\tau (t^2S)^{1/2}P)P=
\int _0^\tau \cos (\widetilde{\tau}(t^2S)^{1/2}P)P\,d\widetilde{\tau}.
\end{equation}
Similarly,
\begin{equation}
\label{sin hat = int}
(t^2 M_0 \widehat{S}M_0)^{-1/2}\sin (\tau (t^2 M_0\widehat{S}M_0)^{1/2})M_0\widehat{P}
=\int _0^\tau \cos(\widetilde{\tau}(t^2 M_0\widehat{S}M_0)^{1/2})M_0\widehat{P}\,d\widetilde{\tau}.
\end{equation}
Integrating \eqref{cos and hat cos identity} over $\tau$ and taking \eqref{sin not hat = int}, \eqref{sin hat = int} into account, we conclude that
\begin{equation}
\label{sine and hat sine}
M (t^2S)^{-1/2}\sin (\tau (t^2S)^{1/2}P)P M^*=M_0(t^2 M_0 \widehat{S}M_0)^{-1/2}\sin (\tau (t^2 M_0\widehat{S}M_0)^{1/2})M_0\widehat{P}.
\end{equation}
Next, since $M_0=(Q_{\widehat{\mathfrak{N}}})^{-1/2}$, using \eqref{PM*=}, we obtain 
$PM^*M_0^{-2}\widehat{P}=M^{-1}\widehat{P}$. 
So, by \eqref{sine and hat sine},
\begin{equation}
\label{3.13a}
M(t^2S)^{-1/2}\sin (\tau (t^2S)^{1/2}P))M^{-1}\widehat{P}
=M_0(t^2M_0\widehat{S}M_0)^{-1/2}\sin (\tau (t^2M_0\widehat{S}M_0)^{1/2})M_0^{-1}\widehat{P}.
\end{equation}
Thus,
\begin{equation}
\label{end of proof prop 3.1}
\begin{split}
& MA(t)^{-1/2}\sin (\tau A(t)^{1/2})M^{-1}\widehat{P}
-M_0(t^2M_0\widehat{S}M_0)^{-1/2}\sin (\tau(t^2M_0\widehat{S}M_0)^{1/2})M_0^{-1}\widehat{P}\\
&=M\Bigl(
A(t)^{-1/2}\sin (\tau A(t)^{1/2})P
-(t^2S)^{-1/2}\sin (\tau (t^2S)^{1/2}P)P
\Bigr)M^{-1}\widehat{P}.
\end{split}
\end{equation}
Using Proposition~\ref{Proposition abstract sin tau principal} and \eqref{end of proof prop 3.1}, we arrive at inequality~\eqref{prop hat sin principal abstr scheme}.
\end{proof}

\subsection{Approximation with the corrector}

\begin{proposition}
Under the assumptions of Subsec.~\textnormal{\ref{Subsec hat famimy abstract}}, for $\tau\in\mathbb{R}$ and $\vert t\vert \leqslant t_0$ we have
\begin{equation}
\label{prop sin tu corr abstract}
\begin{split}
\Bigl\Vert & \widehat{A}(t)^{1/2}\Bigl(
MA(t)^{-1/2}\sin(\tau A(t)^{1/2})M^{-1}\widehat{P}
\\
&-(I+t\widehat{Z}_Q)M_0(t^2 M_0\widehat{S}M_0)^{-1/2}\sin (\tau (t^2 M_0\widehat{S}M_0)^{1/2})M_0^{-1}\widehat{P}\Bigr)\Bigr\Vert _{\widehat{\mathfrak{H}}\rightarrow\widehat{\mathfrak{H}}}
\\
&\leqslant C_9 \Vert M^{-1}\Vert  (\vert t\vert +\vert \tau\vert t^2).
\end{split}
\end{equation}
The constant $C_9$ is defined in \eqref{Th corr exp} and depends only on $\delta$, $\Vert X_1\Vert$, and $c_*$.
\end{proposition}

\begin{proof}
Estimate \eqref{prop sin tu corr abstract} follows from Proposition~\ref{Proposition corr sin tau abstr no hat} by recalculation.  According to \eqref{hat Z_Q=} and \eqref{3.13a},
\begin{align}
\label{3.14a}
\begin{split}
&t \widehat{Z}_Q M_0 (t^2 M_0\widehat{S}M_0)^{-1/2}\sin \bigl(\tau (t^2 M_0\widehat{S}M_0)^{1/2}\bigr)M_0^{-1}\widehat{P}\\
&=
tMZM^{-1}M_0 (t^2 M_0\widehat{S}M_0)^{-1/2}\sin \bigl(\tau (t^2 M_0\widehat{S}M_0)^{1/2}\bigr)M_0^{-1}\widehat{P}
\\
&=
tMZ(t^2S)^{-1/2}\sin (\tau (t^2S)^{1/2})PM^{-1}\widehat{P}.
\end{split}
\end{align}
Combining \eqref{3.14a} with \eqref{A(t) and hat A(t)} and \eqref{end of proof prop 3.1}, we obtain
\begin{equation*}
\begin{split}
&\Bigl\Vert
\widehat{A}(t)^{1/2}\Bigl(
MA(t)^{-1/2}\sin (\tau A(t)^{1/2})M^{-1}\widehat{P}\\
&-(I+t\widehat{Z}_Q)M_0(t^2 M_0\widehat{S}M_0)^{-1/2}\sin \bigl(\tau (t^2 M_0\widehat{S}M_0)^{1/2}\bigr)M_0^{-1}\widehat{P}
\Bigr)\Bigr\Vert _{\widehat{\mathfrak{H}}\rightarrow\widehat{\mathfrak{H}}}
\\
&=\Bigl\Vert A(t)^{1/2}\Bigl(
A(t)^{-1/2}\sin (\tau A(t)^{1/2})P
-(I+tZ)(t^2S)^{-1/2}\sin (\tau (t^2S)^{1/2})P\Bigr)M^{-1}\widehat{P}
\Bigr\Vert _{\widehat{\mathfrak{H}}\rightarrow\mathfrak{H}}.
\end{split}
\end{equation*}
Together with Proposition~\ref{Proposition corr sin tau abstr no hat}, this implies~\eqref{prop sin tu corr abstract}.
\end{proof}

\subsection{Approximation of the sandwiched operator $A(t)^{-1/2}\sin (\varepsilon ^{-1}\tau A(t)^{1/2})$}

Writing down  \eqref{prop hat sin principal abstr scheme} and \eqref{prop sin tu corr abstract} with $\tau$ replaced by $\varepsilon ^{-1}\tau$ and multiplying the corresponding inequalities by the ,,smoothing factors,'' we arrive at the following result.

\begin{theorem}
\label{Theorem sin sandwiched abstract}
Under the assumptions of Subsec.~\textnormal{\ref{Subsec hat famimy abstract}}, for $\tau\in\mathbb{R}$, $\varepsilon >0$, and $\vert t\vert \leqslant t_0$ we have
\begin{align}
\label{Th 3.3_1}
\begin{split}
\Bigl\Vert &\Bigl( MA(t)^{-1/2}\sin (\varepsilon ^{-1}\tau A(t)^{1/2})M^{-1}\widehat{P}
\\
&-M_0(t^2M_0\widehat{S}M_0)^{-1/2}\sin (\varepsilon ^{-1}\tau(t^2M_0\widehat{S}M_0)^{1/2})M_0^{-1}\widehat{P}\Bigr)
\varepsilon (t^2+\varepsilon ^2)^{-1/2}\widehat{P}\Bigr\Vert _{\widehat{\mathfrak{H}}\rightarrow\widehat{\mathfrak{H}}}
\\
&\leqslant C_7\Vert M\Vert \Vert M^{-1}\Vert  (1+\vert\tau\vert ),
\end{split}
\\
\label{Th 3.3_2}
\begin{split}
\Bigl\Vert & \widehat{A}(t)^{1/2}\Bigl(
MA(t)^{-1/2}\sin(\varepsilon ^{-1}\tau A(t)^{1/2})M^{-1}\widehat{P}
\\
&-(I+t\widehat{Z}_Q)M_0(t^2 M_0\widehat{S}M_0)^{-1/2}\sin (\varepsilon ^{-1}\tau (t^2 M_0\widehat{S}M_0)^{1/2})M_0^{-1}\widehat{P}\Bigr)
\varepsilon ^2(t^2+\varepsilon ^2)^{-1}\widehat{P}\Bigr\Vert _{\widehat{\mathfrak{H}}\rightarrow\widehat{\mathfrak{H}}}
\\
&\leqslant C_9 \Vert M^{-1}\Vert  \varepsilon(1 +\vert \tau\vert ).
\end{split}
\end{align}
The number $t_0$ is subject to \eqref{t_0(delta) abstact scheme}, the constants $C_7$ and $C_9$ are defined by \eqref{A-1/2sin F 2} and \eqref{Th corr exp}.
\end{theorem}

\section*{Chapter II. Periodic differential operators in $L_2(\mathbb{R}^d;\mathbb{C}^n)$}

\label{Section Chapter 2}

In the present chapter, we describe the class of matrix second order differential operators admitting a factorization of the form 
 $\mathcal{A}=\mathcal{X}^*\mathcal{X}$, where $\mathcal{X}$ is a homogeneous first order DO. This class was introduced and studied in \cite[Chapter~2]{BSu}. 

\section{Factorized second order operators}
\label{Section Factorized families}

\subsection{Lattices $\Gamma$ and $\widetilde{\Gamma}$} Let $\Gamma$ be a lattice in $\mathbb{R}^d$ generated by the basis $\mathbf{a}_1,\dots, \mathbf{a}_d$:
\begin{equation*}
\Gamma:=\left\lbrace\mathbf{a}\in\mathbb{R}^d : \mathbf{a}=\sum _{j=1}^d \nu _j\mathbf{a}_j,\quad \nu _j\in\mathbb{Z}\right\rbrace ,
\end{equation*}
and let $\Omega$ be the elementary cell of the lattice $\Gamma$: 
$$
\Omega :=\left\lbrace \mathbf{x}\in\mathbb{R}^d :\mathbf{x}=\sum_{j=1}^d \zeta _j\mathbf{a}_j,\quad -\frac{1}{2}<\zeta _j <\frac{1}{2}\right\rbrace .$$

The basis $\mathbf{b}_1,\dots ,\mathbf{b}_d$ dual to $\mathbf{a}_1,\dots ,\mathbf{a}_d$ is defined by the relations $\langle \mathbf{b}_l,\mathbf{a}_j\rangle =2\pi\delta _{lj}$. This basis generates the lattice $\widetilde{\Gamma}$ dual to $\Gamma$: 
$
\widetilde{\Gamma}:=\left\lbrace \mathbf{b}\in\mathbb{R}^d : \mathbf{b}=\sum _{j=1}^d\mu_j\mathbf{b}_j,\quad \mu_j\in\mathbb{Z}\right\rbrace $. 
Let $\widetilde{\Omega}$ be the first Brillouin zone of the lattice  $\widetilde{\Gamma}$:
\begin{equation}
\label{Omega-tilda}
\widetilde{\Omega}:=\left\lbrace\mathbf{k}\in\mathbb{R}^d : \vert \mathbf{k}\vert <\vert \mathbf{k}-\mathbf{b}\vert,\quad 0\neq\mathbf{b}\in\widetilde{\Gamma}\right\rbrace .
\end{equation}
Let $\vert \Omega\vert$ be the Lebesgue measure of the cell $\Omega$: $\vert \Omega\vert =\mathrm{meas}\,\Omega$, and let $\vert \widetilde{\Omega}\vert =\mathrm{meas}\,\widetilde{\Omega}$.  We put $2r_1:=\mathrm{diam}\,\Omega$. 
The maximal radius of the ball containing in $\mathrm{clos}\,\widetilde{\Omega}$ is denoted by $r_0$.  Note that
\begin{equation}
\label{2r0=}
2r_0=\min _{0\neq\mathbf{b}\in\widetilde{\Gamma}} \vert \mathbf{b}\vert.
\end{equation}

With the lattice $\Gamma$, we associate the discrete Fourier transformation
\begin{equation}
\label{discr Fourier}
v(\mathbf{x})=\vert \Omega\vert ^{-1/2}\sum _{\mathbf{b}\in \widetilde{\Gamma}}\widehat{v}_\mathbf{b} e^{i\langle \mathbf{b},\mathbf{x}\rangle },\quad\mathbf{x}\in\Omega ,
\end{equation}
which is a unitary mapping of $l_2(\widetilde{\Gamma})$ onto $L_2(\Omega)$:
\begin{equation}
\label{Fourier unit}
\int _\Omega \vert v(\mathbf{x})\vert ^2\,d\mathbf{x}=\sum _{\mathbf{b}\in\widetilde{\Gamma}}\vert \widehat{v}_\mathbf{b}\vert ^2 .
\end{equation}

\textit{Below by $\widetilde{H}^1(\Omega;\mathbb{C}^n)$ we denote the subspace of functions from $H^1(\Omega;\mathbb{C}^n)$ whose $\Gamma$-periodic extension to $\mathbb{R}^d$ belongs to $H^1_{\mathrm{loc}}(\mathbb{R}^d;\mathbb{C}^n)$.} We have
\begin{equation}
\label{6.5 from DSu}
\Vert (\mathbf{D}+\mathbf{k})\mathbf{u}\Vert ^2_{L_2(\Omega)}  =\sum _{\mathbf{b}\in\widetilde{\Gamma}}\vert \mathbf{b}+\mathbf{k}\vert ^2\vert \widehat{\mathbf{u}}_\mathbf{b}\vert ^2,\quad \mathbf{u}\in \widetilde{H}^1(\Omega ;\mathbb{C}^n),\quad\mathbf{k}\in\mathbb{R}^d,
\end{equation}
and convergence of the series in the right-hand side of \eqref{6.5 from DSu} is equivalent to
the relation $\mathbf{u}\in\widetilde{H}^1(\Omega ;\mathbb{C}^n)$. From \eqref{Omega-tilda}, \eqref{Fourier unit}, and \eqref{6.5 from DSu} it follows that
\begin{equation}
\label{D+k >=}
 \Vert (\mathbf{D}+\mathbf{k})\mathbf{u}\Vert ^2 _{L_2(\Omega)}
\geqslant \sum _{\mathbf{b}\in\widetilde{\Gamma}}\vert \mathbf{k}\vert ^2\vert \widehat{\mathbf{u}}_\mathbf{b}\vert ^2
=\vert \mathbf{k}\vert ^2 \Vert \mathbf{u}\Vert ^2_{L_2(\Omega)} 
 ,\quad \mathbf{u}\in\widetilde{H}^1(\Omega ;\mathbb{C}^n),\quad\mathbf{k}\in\widetilde{\Omega}.
\end{equation}

If $\psi (\mathbf{x})$ is a $\Gamma$-periodic measurable matrix-valued function in $\mathbb{R}^d$, we put $\overline{\psi}:=\vert \Omega\vert ^{-1}\int _\Omega \psi (\mathbf{x})\,d\mathbf{x}$ and $\underline{\psi}:=\left(\vert \Omega\vert ^{-1}\int _\Omega \psi (\mathbf{x})^{-1}\,d\mathbf{x}\right)^{-1}$. Here, in the definition of $\overline{\psi}$ it is assumed that $\psi\in L_{1,\mathrm{loc}}(\mathbb{R}^d)$, and in the definition of $\underline{\psi}$ it is assumed that the matrix $\psi (\mathbf{x})$ is square and nondegenerate, and $\psi ^{-1}\in L_{1,\mathrm{loc}}(\mathbb{R}^d)$.

\subsection{The Gelfand transformation} 
\label{Subsec Gelfand}
Initially, the Gelfand transformation 
$\mathcal{U}$ is defined on the functions of the Schwartz class by the formula
\begin{equation*}
\begin{split}
\widetilde{\mathbf{v}}(\mathbf{k},\mathbf{x})=(\mathcal{U}\mathbf{v})(\mathbf{k},\mathbf{x})=\vert \widetilde{\Omega}\vert ^{-1/2}\sum _{\mathbf{a}\in\Gamma}e^{-i\langle\mathbf{k},\mathbf{x}+\mathbf{a}\rangle }\mathbf{v}(\mathbf{x}+\mathbf{a}),\quad
\mathbf{v}\in\mathcal{S}(\mathbb{R}^d;\mathbb{C}^n),\quad \mathbf{x}\in\Omega ,\quad \mathbf{k}\in\widetilde{\Omega}.
\end{split}
\end{equation*}
Since 
$
\int _{\widetilde{\Omega}}\int _\Omega \vert \widetilde{\mathbf{v}}(\mathbf{k},\mathbf{x})\vert ^2\,d\mathbf{x}\,d\mathbf{k}=\int _{\mathbb{R}^d}\vert\mathbf{v}(\mathbf{x})\vert ^2\,d\mathbf{x}$, 
the transformation $\mathcal{U}$  extends by continuity up to a unitary mapping 
$
\mathcal{U}: L_2(\mathbb{R}^d;\mathbb{C}^n)\rightarrow \int _{\widetilde{\Omega}}\oplus L_2(\Omega ;\mathbb{C}^n)\,d\mathbf{k}$. 
Relation $\mathbf{v}\in H^1(\mathbb{R}^d;\mathbb{C}^n)$ is equivalent to $\widetilde{\mathbf{v}}(\mathbf{k},\cdot)\in \widetilde{H}^1(\Omega ;\mathbb{C}^n)$ for a.~e. $\mathbf{k}\in\widetilde{\Omega}$ and 
$
\int _{\widetilde{\Omega}}\int _\Omega \left(\vert (\mathbf{D}+\mathbf{k})\widetilde{\mathbf{v}}(\mathbf{k},\mathbf{x})\vert ^2 +\vert \widetilde{\mathbf{v}}(\mathbf{k},\mathbf{x})\vert ^2\right)\,d\mathbf{x}\,d\mathbf{k}<\infty $. 
Under the Gelfand transformation, the operator of multiplication by a
bounded periodic function in $L_2(\mathbb{R}^d;\mathbb{C}^n)$ turns into multiplication by the
same function on the fibers of the direct integral. The operator $\mathbf{D}$ applied to $\mathbf{v}\in H^1(\mathbb{R}^d;\mathbb{C}^n)$ turns into the operator $\mathbf{D}+\mathbf{k}$ applied to $\widetilde{\mathbf{v}}(\mathbf{k},\cdot )\in \widetilde{H}^1(\Omega ;\mathbb{C}^n)$.

\subsection{Factorized second order operators}

\label{Subsubsection factorized famalies}

Let $b(\mathbf{D})$ be a matrix first order DO of the form $\sum _{j=1}^d b_j D_j$, where $b_j$, $j=1,\dots ,d$, are constant matrices of size $m\times n$ (in general, with complex entries). \textit{We always assume that} $m\geqslant n$. Suppose that the symbol $b(\boldsymbol{\xi})=\sum _{j=1}^d b_j \xi _j$, $\boldsymbol{\xi}\in\mathbb{R}^d$, of the operator $b(\mathbf{D})$ has maximal rank: $\mathrm{rank}\,b(\boldsymbol{\xi})=n$ for $0\neq\boldsymbol{\xi}\in\mathbb{R}^d$. This condition is equivalent to the existence of constants  $\alpha _0$, $\alpha _1>0$ such that 
\begin{equation}
\label{<b^*b<}
\alpha _0\mathbf{1}_n\leqslant b(\boldsymbol{\theta})^*b(\boldsymbol{\theta})\leqslant \alpha _1\mathbf{1}_n,\quad\boldsymbol{\theta}\in\mathbb{S}^{d-1},\quad 0<\alpha _0\leqslant \alpha _1 <\infty .
\end{equation}
From \eqref{<b^*b<} it follows that
\begin{equation}
\label{b_j <=}
\vert b_j\vert \leqslant \alpha _1^{1/2},\quad j=1,\dots ,d.
\end{equation}

Let $\Gamma$-periodic Hermitian $(m\times m)$-matrix-valued function $g(\mathbf{x})$ be positive definite and bounded together with the inverse matrix:
\begin{equation}
\label{g in}
g(\mathbf{x})>0;\quad g, g^{-1}\in L_\infty (\mathbb{R}^d).
\end{equation}
Suppose that $f(\mathbf{x})$ is a $\Gamma$-periodic $(n\times n)$-matrix-valued function such that $f,f^{-1}\in L_\infty (\mathbb{R}^d)$. 
In $L_2(\mathbb{R}^d;\mathbb{C}^n)$, consider DO $\mathcal{A}$ formally given by the differential expression
\begin{equation}
\label{A=}
\mathcal{A}=f(\mathbf{x})^*b(\mathbf{D})^*g(\mathbf{x})b(\mathbf{D})f(\mathbf{x}).
\end{equation} 
The precise definition of the operator $\mathcal{A}$ is given in terms of the quadratic form 
\begin{equation*}\mathfrak{a}[\mathbf{u},\mathbf{u}]:=\left(gb(\mathbf{D})(f\mathbf{u}),b(\mathbf{D})(f\mathbf{u})\right)_{L_2(\mathbb{R}^d)}, \quad\mathbf{u}\in \mathrm{Dom}\,\mathfrak{a}:=\lbrace \mathbf{u}\in L_2(\mathbb{R}^d;\mathbb{C}^n) : f\mathbf{u}\in H^1(\mathbb{R}^d;\mathbb{C}^n)\rbrace.
\end{equation*} 
Using the Fourier transformation and assumptions \eqref{<b^*b<}, \eqref{g in}, it is easily seen that
\begin{equation}
\label{<a<}
\alpha _0\Vert g^{-1}\Vert _{L_\infty}^{-1}\Vert \mathbf{D}(f\mathbf{u})\Vert ^2_{L_2(\mathbb{R}^d)}\leqslant\mathfrak{a}[\mathbf{u},\mathbf{u}] 
\leqslant \alpha _1\Vert g\Vert _{L_\infty} \Vert \mathbf{D}(f\mathbf{u})\Vert ^2_{L_2(\mathbb{R}^d)},\quad \mathbf{u}\in \mathrm{Dom}\,\mathfrak{a}.
\end{equation}
Thus, the form $\mathfrak{a}[\cdot,\cdot]$ is closed and non-negative.

The operator $\mathcal{A}$ admits a factorization of the form $\mathcal{A}=\mathcal{X}^*\mathcal{X}$, where 
\begin{equation*}\mathcal{X}:=g(\mathbf{x})^{1/2}b(\mathbf{D})f(\mathbf{x}) : L_2(\mathbb{R}^d;\mathbb{C}^n)\rightarrow L_2(\mathbb{R}^d;\mathbb{C}^m),\quad\mathrm{Dom}\,\mathcal{X}=\mathrm{Dom}\,\mathfrak{a}.
\end{equation*}

\section{Direct integral decomposition for the operator $\mathcal{A}$}

\subsection{The forms $\mathfrak{a}(\mathbf{k})$ and the operators $\mathcal{A}(\mathbf{k})$}

We put 
\begin{equation}
\label{frak H for L2}
\mathfrak{H}:=L_2(\Omega ;\mathbb{C}^n),\quad \mathfrak{H}_*:=L_2(\Omega;\mathbb{C}^m), 
\end{equation}
and consider the closed operator $\mathcal{X}(\mathbf{k}): \mathfrak{H}\rightarrow \mathfrak{H}_*$, $\mathbf{k}\in\mathbb{R}^d$, defined on the domain 
\begin{equation*}
\mathrm{Dom}\,\mathcal{X}(\mathbf{k})=\lbrace \mathbf{u}\in\mathfrak{H} : f\mathbf{u}\in \widetilde{H}^1(\Omega;\mathbb{C}^n)\rbrace =:\mathfrak{d}
\end{equation*}
by the expression $\mathcal{X}(\mathbf{k})=g(\mathbf{x})^{1/2}b(\mathbf{D}+\mathbf{k})f(\mathbf{x})$. 
The selfadjoint operator $\mathcal{A}(\mathbf{k}):=\mathcal{X}(\mathbf{k})^*\mathcal{X}(\mathbf{k})$ in $L_2(\Omega;\mathbb{C}^n)$ is formally given by the differential expression 
\begin{equation}
\label{A(k)=}
\mathcal{A}(\mathbf{k})=f(\mathbf{x})^*b(\mathbf{D}+\mathbf{k})^*g(\mathbf{x})b(\mathbf{D}+\mathbf{k})f(\mathbf{x})
\end{equation}
with the periodic boundary conditions. The precise definition of the operator $\mathcal{A}(\mathbf{k})$ is given in terms of the closed quadratic form   
$\mathfrak{a}(\mathbf{k})[\mathbf{u},\mathbf{u}]:=\Vert \mathcal{X}(\mathbf{k})\mathbf{u}\Vert ^2_{\mathfrak{H}_*}$, $\mathbf{u}\in\mathfrak{d}$. Using the discrete Fourier transformation \eqref{discr Fourier} and assumptions \eqref{<b^*b<}, \eqref{g in}, it is easily seen that
\begin{equation}
\label{a(k) estimates}
\begin{split}
\alpha _0\Vert g^{-1}\Vert ^{-1}_{L_\infty}\Vert (\mathbf{D}+\mathbf{k})(f\mathbf{u})\Vert ^2 _{L_2(\Omega)}\leqslant
\mathfrak{a}(\mathbf{k})[\mathbf{u},\mathbf{u}]\leqslant \alpha _1\Vert g\Vert _{L_\infty}\Vert (\mathbf{D}+\mathbf{k})(f\mathbf{u})\Vert ^2_{L_2(\Omega)},\quad\mathbf{u}\in\mathfrak{d}.
\end{split}
\end{equation}
So, by the compactness of the embedding $\widetilde{H}^1(\Omega ;\mathbb{C}^n)\hookrightarrow L_2(\Omega ;\mathbb{C}^n)$,  the spectrum of $\mathcal{A}(\mathbf{k})$ is discrete and the resolvent is compact. 

By \eqref{D+k >=} and the lower estimate \eqref{a(k) estimates},
\begin{equation}
\label{A(k)>=}
\mathcal{A}(\mathbf{k})\geqslant c_*\vert \mathbf{k}\vert  ^2 I,\quad \mathbf{k}\in\mathrm{clos}\,\widetilde{\Omega};\quad 
c_*:=\alpha _0\Vert g^{-1}\Vert ^{-1}_{L_\infty}\Vert f^{-1}\Vert ^{-2}_{L_\infty}
.
\end{equation}

We  put
\begin{equation}
\label{N in L2}
\mathfrak{N}:=\mathrm{Ker}\,\mathcal{A}(0)=\mathrm{Ker}\,\mathcal{X}(0).
\end{equation} 
Then 
\begin{equation}
\label{N= in L2}
\mathfrak{N}=\lbrace \mathbf{u}\in L_2(\Omega;\mathbb{C}^n) : f\mathbf{u}=\mathbf{c}\in\mathbb{C}^n\rbrace.
\end{equation}  

From \eqref{2r0=} and \eqref{6.5 from DSu} with $\mathbf{k}=0$ it follows that
\begin{equation*}
\Vert \mathbf{D}\mathbf{v}\Vert ^2 _{L_2(\Omega)}\geqslant 4 r_0^2\Vert \mathbf{v}\Vert ^2 _{L_2(\Omega)},\quad\mathbf{v}=f\mathbf{u}\in\widetilde{H}^1(\Omega ;\mathbb{C}^n),\quad \int _\Omega \mathbf{v} (\mathbf{x})\,d\mathbf{x}=0.
\end{equation*}
Combining this with the lower estimate 
\eqref{a(k) estimates} for $\mathbf{k}=0$, we see that the distance $d_0$ from the point zero to the rest of the spectrum of $\mathcal{A}(0)$ satisfies
\begin{equation}
\label{d^0<= in L2}
d_0\geqslant 4c_*r_0^2.
\end{equation}

\subsection{Direct integral decomposition for $\mathcal{A}$}
Using the Gelfand transformation, we decompose 
the operator $\mathcal{A}$ into the direct integral of the operators $\mathcal{A}(\mathbf{k})$:
\begin{equation}
\label{5.7a}
\mathcal{U}\mathcal{A}\mathcal{U}^{-1}=\int _{\widetilde{\Omega}} \oplus \mathcal{A}(\mathbf{k})\,d\mathbf{k}.
\end{equation}
This means the following. If $\mathbf{v}\in \mathrm{Dom}\,\mathfrak{a}$, then
\begin{align}
\label{v in dom a e k}
&\widetilde{\mathbf{v}}(\mathbf{k},\cdot)=(\mathcal{U}\mathbf{v})(\mathbf{k},\cdot)\in \mathfrak{d}\quad \mbox{for a. e.}\;\mathbf{k}\in\widetilde{\Omega},\\
\label{a(k)=oplus int}
&\mathfrak{a}[\mathbf{v},\mathbf{v}]=\int _{\widetilde{\Omega}}\mathfrak{a}(\mathbf{k})[\widetilde{\mathbf{v}}(\mathbf{k},\cdot),\widetilde{\mathbf{v}}(\mathbf{k},\cdot)]\,d\mathbf{k}.
\end{align}
Conversely, if $\widetilde{\mathbf{v}}\in \int _{\widetilde{\Omega}}\oplus L_2(\Omega;\mathbb{C}^n)\,d\mathbf{k}$ satisfies \eqref{v in dom a e k} and the integral in \eqref{a(k)=oplus int} is finite, then $\mathbf{v}\in \mathrm{Dom}\,\mathfrak{a}$ and \eqref{a(k)=oplus int} holds.

\subsection{Incorporation of the operators $\mathcal{A}(\mathbf{k})$ into the abstract scheme}

\label{Subsection Incorporation of the operators A(k) into the abstract scheme}

For $d>1$ the operators $\mathcal{A}(\mathbf{k})$ depend on the multidimensional parameter $\mathbf{k}$. According to \cite[Chapter~2]{BSu}, we consider the onedimensional parameter $t:=\vert\mathbf{k}\vert$. We will apply the scheme of Chapter~{I}. Herewith, all our considerations will depend on the additional parameter  $\boldsymbol{\theta}=\mathbf{k}/\vert\mathbf{k}\vert\in\mathbb{S}^{d-1}$, and we need to make our estimates uniform with respect to $\boldsymbol{\theta}$. 

The spaces $\mathfrak{H}$ and $\mathfrak{H}_*$ are defined by \eqref{frak H for L2}. Let $X(t)=X(t,\boldsymbol{\theta}):=\mathcal{X}(t\boldsymbol{\theta})$. Then $X(t,\boldsymbol{\theta})=X_0+tX_1(\boldsymbol{\theta})$, where $X_0=g(\mathbf{x})^{1/2}b(\mathbf{D})f(\mathbf{x})$, $\mathrm{Dom}\,X_0=\mathfrak{d}$, and $X_1(\boldsymbol{\theta})$ is a bounded operator of multiplication by the matrix-valued function $g(\mathbf{x})^{1/2}b(\boldsymbol{\theta})f(\mathbf{x})$. We put $A(t)=A(t,\boldsymbol{\theta}):=\mathcal{A}(t\boldsymbol{\theta})$. Then $A(t,\boldsymbol{\theta})=X(t,\boldsymbol{\theta})^*X(t,\boldsymbol{\theta})$. According to \eqref{N in L2} and \eqref{N= in L2}, $\mathfrak{N}=\mathrm{Ker}\,X_0=\mathrm{Ker}\,\mathcal{A}(0)$, $\mathrm{dim}\,\mathfrak{N}=n$. The distance $d_0$ 
from the point zero to the rest of the spectrum of $\mathcal{A}(0)$ satisfied 
estimate \eqref{d^0<= in L2}. As was shown in \cite[Chapter 2, Sec.~3]{BSu}, the condition $n\leqslant n_*=\mathrm{dim}\,\mathrm{Ker}\,X_0^*$ is also fulfilled. 
Thus, all the assumptions of Section~\ref{Section Preliminaries} are valid.

In Subsection \ref{Subsubsection operator pencils}, it was required to choose the number $\delta <d_0/8$. Taking \eqref{A(k)>=} and \eqref{d^0<= in L2} into account, we put
\begin{equation}
\label{delta L2}
\delta :=c_*r_0^2/4=(r_0/2)^2\alpha _0\Vert g^{-1}\Vert ^{-1}_{L_\infty}\Vert f^{-1}\Vert ^{-2}_{L_\infty}.
\end{equation}
Next, by \eqref{<b^*b<},  the operator  $X_1(\boldsymbol{\theta})=g(\mathbf{x})^{1/2}b(\boldsymbol{\theta})f(\mathbf{x})$ satisfies 
\begin{equation}
\label{X_1(theta)<=}
\Vert X_1(\boldsymbol{\theta})\Vert \leqslant \alpha _1^{1/2}\Vert g\Vert ^{1/2}_{L_\infty}\Vert f\Vert _{L_\infty}.
\end{equation}
This allows us to take 
the following number 
\begin{equation}
\label{t0 L2}
t_0:=\delta^{1/2}\alpha _1^{-1/2}\Vert g\Vert ^{-1/2}_{L_\infty}\Vert f\Vert ^{-1}_{L_\infty}
=(r_0/2)\alpha _0^{1/2}\alpha _1^{-1/2}\Vert g\Vert ^{-1/2}_{L_\infty}\Vert g^{-1}\Vert ^{-1/2}_{L_\infty}\Vert f\Vert ^{-1}_{L_\infty}\Vert f^{-1}\Vert ^{-1}_{L_\infty}
\end{equation}
in the role of 
$t_0$ (see \eqref{t_0(delta) abstact scheme}). 
Obviously, $t_0\leqslant r_0/2$, and the ball $\vert\mathbf{k}\vert\leqslant t_0$ lies in $\widetilde{\Omega}$. It is important
that $c_*$, $\delta$, and $t_0$ (see \eqref{A(k)>=}, \eqref{delta L2}, \eqref{t0 L2}) do not depend on the parameter $\boldsymbol{\theta}$.

From \eqref{A(k)>=} it follows that the spectral germ $S(\boldsymbol{\theta})$ (which now depends on $\boldsymbol{\theta}$) 
is nondegenerate:
\begin{equation}
\label{S(theta)>=}
S(\boldsymbol{\theta})\geqslant c_* I_\mathfrak{N}.
\end{equation}
It is important that the spectral germ is nondegenerate uniformly in $\boldsymbol{\theta}$.

\section{The operator $\widehat{\mathcal{A}}$. The effective matrix. The effective operator}
\label{Sec eff op}

\subsection{The operator $\widehat{\mathcal{A}}$}

In the case where $f=\mathbf{1}_n$, we agree to mark all the objects by the upper hat ,,$\widehat{\phantom{a}}$''. 
We have $\widehat{\mathfrak{H}}=\mathfrak{H}=L_2(\Omega ;\mathbb{C}^n)$. 
For the operator  
\begin{equation}
\label{hat A}
\widehat{\mathcal{A}}=b(\mathbf{D})^*g(\mathbf{x})b(\mathbf{D}),
\end{equation} 
 the family
\begin{equation}
\label{hat A(k)=}
\widehat{\mathcal{A}}(\mathbf{k})=b(\mathbf{D}+\mathbf{k})^*g(\mathbf{x})b(\mathbf{D}+\mathbf{k})
\end{equation}
is denoted by $\widehat{A}(t;\boldsymbol{\theta})$. If $f=\mathbf{1}_n$, the kernel \eqref{N= in L2} takes the form
\begin{equation}
\label{hat N= in L2}
\widehat{\mathfrak{N}}=\lbrace \mathbf{u}\in L_2(\Omega ;\mathbb{C}^n) :\mathbf{u}=\mathbf{c}\in\mathbb{C}^n\rbrace .
\end{equation}
Let $\widehat{P}$ be the orthogonal projection of $\mathfrak{H}$ onto the subspace $\widehat{\mathfrak{N}}$. Then $\widehat{P}$ is the operator of averaging over the cell:
\begin{equation}
\label{P L2}
\widehat{P}\mathbf{u}=\vert \Omega\vert ^{-1}\int _\Omega \mathbf{u}(\mathbf{x})\,d\mathbf{x},\quad\mathbf{u}\in L_2(\Omega ;\mathbb{C}^n).
\end{equation}

From \eqref{A(k)>=} with $f=\mathbf{1}_n$ it follows that
\begin{equation}
\label{hat A(k)>=}
\widehat{\mathcal{A}}(\mathbf{k})=\widehat{A}(t,\boldsymbol{\theta})
\geqslant \widehat{c}_* t^2 I,\quad \mathbf{k}=t\boldsymbol{\theta}\in\mathrm{clos}\, \widetilde{\Omega};\quad \widehat{c}_*:=\alpha _0\Vert g^{-1}\Vert ^{-1}_{L_\infty}.
\end{equation}

\subsection{The effective matrix}

In accordance with \cite[Chapter 3, Sec.~1]{BSu}, the spectral germ $\widehat{S}(\boldsymbol{\theta})$ of the operator family $\widehat{A}(t,\boldsymbol{\theta})$ acting in $\widehat{\mathfrak{N}}$ can be represented as
\begin{equation}
\label{S(theta)=}
\widehat{S}(\boldsymbol{\theta})=b(\boldsymbol{\theta})^*g^0b(\boldsymbol{\theta}),\quad\boldsymbol{\theta}\in\mathbb{S}^{d-1},
\end{equation}
where $b(\boldsymbol{\theta})$ is the symbol of the operator $b(\mathbf{D})$ and $g^0$ is the so-called \textit{effective matrix}. The constant positive $(m\times m)$-matrix $g^0$ is defined as follows. Assume that a $\Gamma$-periodic $(n\times m)$-matrix-valued function $\Lambda\in\widetilde{H}^1(\Omega)$ is the weak solution of the problem
\begin{equation}
\label{Lambda problem}
b(\mathbf{D})^*g(\mathbf{x})(b(\mathbf{D})\Lambda (\mathbf{x})+\mathbf{1}_m)=0,\quad \int _\Omega \Lambda (\mathbf{x})\,d\mathbf{x}=0.
\end{equation}
Denote
\begin{equation}
\label{tilde g}
\widetilde{g}(\mathbf{x}):=g(\mathbf{x})(b(\mathbf{D})\Lambda (\mathbf{x})+\mathbf{1}_m).
\end{equation}
Then the effective matrix $g^0$ is given by
\begin{equation}
\label{g0}
g^0=\vert \Omega\vert ^{-1}\int _\Omega \widetilde{g}(\mathbf{x})\,d\mathbf{x}.
\end{equation}
It turns out that the matrix $g^0$ is positive definite. In the case where $f=\mathbf{1}_n$, estimate \eqref{S(theta)>=} takes the form
\begin{equation}
\label{hat S(theta)>=}
\widehat{S}(\boldsymbol{\theta})\geqslant \widehat{c}_* I_{\widehat{\mathfrak{N}}}.
\end{equation}

From \eqref{Lambda problem} it is easy to derive that
\begin{equation}
\label{b(D)Lambda <=}
\Vert b(\mathbf{D})\Lambda \Vert _{L_2(\Omega)}\leqslant \vert \Omega\vert ^{1/2}m^{1/2}\Vert g\Vert ^{1/2}_{L_\infty}\Vert g^{-1}\Vert _{L_\infty}^{1/2}.
\end{equation}
We also need the following inequalities obtained in 
\cite[(6.28) and Subsec. 7.3]{BSu05}:
\begin{align}
\label{Lambda<=}
&\Vert \Lambda\Vert _{L_2(\Omega)}
\leqslant \vert \Omega\vert ^{1/2}M_1;\quad M_1:=m^{1/2}(2r_0)^{-1}\alpha_0^{-1/2}\Vert g\Vert ^{1/2}_{L_\infty}\Vert g^{-1}\Vert ^{1/2}_{L_\infty};
\\
\label{D Lambda <=}
&\Vert \mathbf{D}\Lambda\Vert _{L_2(\Omega)}
\leqslant \vert\Omega\vert ^{1/2}M_2;\quad M_2:=m^{1/2}\alpha_0^{-1/2}\Vert g\Vert ^{1/2}_{L_\infty}\Vert g^{-1}\Vert ^{1/2}_{L_\infty}.
\end{align}

\subsection{The effective operator $\widehat{\mathcal{A}}^0$}

By \eqref{S(theta)=} and the homogeneity of the symbol $b(\mathbf{k})$, we have
\begin{equation}
\label{S(k)=}
\widehat{S}(\mathbf{k}):=t^2 \widehat{S}(\boldsymbol{\theta})=b(\mathbf{k})^*g^0b(\mathbf{k}),\quad\mathbf{k}\in\mathbb{R}^d,\quad t=\vert\mathbf{k}\vert,\quad\boldsymbol{\theta}=\mathbf{k}/\vert\mathbf{k}\vert.
\end{equation}
The matrix $\widehat{S}(\mathbf{k})$ is the symbol of the differential operator 
\begin{equation}
\label{A^0 hat}
\widehat{\mathcal{A}}^0=b(\mathbf{D})^*g^0b(\mathbf{D})
\end{equation}
acting in $L_2(\mathbb{R}^d;\mathbb{C}^n)$ on the domain $H^2(\mathbb{R}^d;\mathbb{C}^n)$ and 
called the \textit{effective operator} for the operator~$\widehat{\mathcal{A}}$.

Let $\widehat{\mathcal{A}}^0(\mathbf{k})$ be the operator family in $L_2(\Omega;\mathbb{C}^n)$ corresponding to the effective operator  $\widehat{\mathcal{A}}^0$. Then 
$
\widehat{\mathcal{A}}^0(\mathbf{k})=b(\mathbf{D}+\mathbf{k})^*g^0b(\mathbf{D}+\mathbf{k})
$ 
with periodic boundary conditions: $\mathrm{Dom}\,\widehat{\mathcal{A}}^0(\mathbf{k})=\widetilde{H}^2(\Omega;\mathbb{C}^n)$. So, by \eqref{P L2} and \eqref{S(k)=},
\begin{equation}
\label{SP=A0P}
\widehat{S}(\mathbf{k})\widehat{P}=\widehat{\mathcal{A}}^0(\mathbf{k})\widehat{P}.
\end{equation}
From estimate \eqref{hat S(theta)>=} for the symbol of the operator $\widehat{\mathcal{A}}^0(\mathbf{k})$ it follows that
\begin{equation}
\label{A^0(k)>=}
\widehat{\mathcal{A}}^0(\mathbf{k})\geqslant \widehat{c}_* \vert \mathbf{k}\vert ^2 I,\quad\mathbf{k}\in\widetilde{\Omega}.
\end{equation}

\subsection{Properties of the effective matrix}
The effective matrix $g^0$ satisfies the estimates known in homogenization theory as
the Voigt-Reuss bracketing  (see, e. g., \cite[Chapter~3, Theorem~1.5]{BSu}).

\begin{proposition}
Let $g^0$ be the effective matrix \eqref{g0}. Then
\begin{equation}
\label{Voigt-Reuss}
\underline{g}
\leqslant g^0\leqslant\overline{g}
.
\end{equation}
If $m=n$, then $g^0=\underline{g}$.
\end{proposition}

From inequalities \eqref{Voigt-Reuss} it follows that
\begin{equation}
\label{g^0<=}
\vert g^0\vert \leqslant\Vert g\Vert _{L_\infty},\quad \vert (g^0)^{-1}\vert\leqslant\Vert g^{-1}\Vert _{L_\infty}.
\end{equation}

Now, we distinguish the cases where one of the
inequalities in \eqref{Voigt-Reuss} becomes an identity. 
See \cite[Chapter 3, Propositions 1.6 and 1.7]{BSu}. 

\begin{proposition}
The equality $g^0=\overline{g}$ is equivalent to the relations
\begin{equation}
\label{overline-g}
b(\mathbf{D})^* {\mathbf g}_k(\mathbf{x}) =0,\ \ k=1,\dots,m,
\end{equation}
where ${\mathbf g}_k(\mathbf{x})$, $k=1,\dots,m,$ are the columns of the matrix-valued function $g(\mathbf{x})$.
\end{proposition}

\begin{proposition} The identity $g^0 =\underline{g}$ is equivalent to the relations
\begin{equation}
\label{underline-g}
{\mathbf l}_k(\mathbf{x}) = {\mathbf l}_k^0 + b(\mathbf{D}) {\mathbf w}_k,\ \ {\mathbf l}_k^0\in \mathbb{C}^m,\ \
{\mathbf w}_k \in \widetilde{H}^1(\Omega;\mathbb{C}^m),\ \ k=1,\dots,m,
\end{equation}
where ${\mathbf l}_k(\mathbf{x})$, $k=1,\dots,m,$ are the columns of the matrix-valued function $g(\mathbf{x})^{-1}$.
\end{proposition}

\section{Approximation of the sandwiched operator $\mathcal{A}(\mathbf{k})^{-1/2}\sin (\varepsilon ^{-1}\tau \mathcal{A}(\mathbf{k})^{1/2})$}

\label{Section 7 appr A(k)}

Now, we consider the operator $\mathcal{A}(\mathbf{k})^{-1/2}\sin (\varepsilon ^{-1}\tau \mathcal{A}(\mathbf{k})^{1/2})$ in the general case where $f\neq\mathbf{1}_n$. Recall that $\mathcal{A}(\mathbf{k})$ is the operator \eqref{A(k)=}. Then
\begin{equation}
\label{A(k) and hat A(k)}
\mathcal{A}(\mathbf{k})=f (\mathbf{x})^*\widehat{\mathcal{A}}(\mathbf{k})f(\mathbf{x}).
\end{equation}

\subsection{Incorporation of $\mathcal{A}(\mathbf{k})$ in the framework of Section~\ref{Section sandwiched abstract}}

As was shown in Subsec.~\ref{Subsection Incorporation of the operators A(k) into the abstract scheme}, the operator $\mathcal{A}(\mathbf{k})$ satisfies the assumptions of Section~\ref{Section Preliminaries}. Now the assumptions of Subsec.~\ref{Subsec hat famimy abstract} are valid with $\mathfrak{H}=\widehat{\mathfrak{H}}=L_2(\Omega;\mathbb{C}^n)$ and $\mathfrak{H}_*=L_2(\Omega ;\mathbb{C}^m)$. The role of $\widehat{A}(t)$ is played by $\widehat{A}(t,\boldsymbol{\theta})=\widehat{\mathcal{A}}(t\boldsymbol{\theta})$, and the role of $A(t)$ is played by $A(t,\boldsymbol{\theta})=\mathcal{A}(t\boldsymbol{\theta})$. An isomorphism $M$ is the operator of multiplication by the function $f(\mathbf{x})$. Relation \eqref{A(t) and hat A(t)} corresponds to the identity \eqref{A(k) and hat A(k)}.

Next, the operator $Q$ (see \eqref{Q= abstract}) is the operator of multiplication by the matrix-valued function 
\begin{equation}
\label{8.1a}
Q(\mathbf{x}):=\left(f(\mathbf{x})f(\mathbf{x})^*\right)^{-1}. 
\end{equation}
The block $Q_{\widehat{\mathfrak{N}}}$ of $Q$ in the subspace $\widehat{\mathfrak{N}}$ (see \eqref{hat N= in L2}) is the operator of multiplication by the constant matrix 
$
\overline{Q}=\left(\underline{ff^*}\right)^{-1}=\vert \Omega\vert ^{-1}\int _\Omega \left( f(\mathbf{x})f(\mathbf{x})^*\right)^{-1}\,d\mathbf{x}$. 
The operator $M_0:=\left(Q_{\widehat{\mathfrak{N}}}\right)^{-1/2}$ acts in $\widehat{\mathfrak{N}}$ as multiplication by the matrix 
$
f_0:=\left(\overline{Q}\right)^{-1/2}=\left( \underline{ff^*}\right)^{1/2}$. 
Obviously,
\begin{equation}
\label{f0<=}
\vert f_0\vert \leqslant \Vert f\Vert _{L_\infty},\quad \vert f_0^{-1}\vert \leqslant \Vert f^{-1}\Vert _{L_\infty}.
\end{equation}

Now, we specify the operators from \eqref{Th 3.3_1} and \eqref{Th 3.3_2}. By \eqref{S(k)=},
\begin{equation}
\label{t2 M0 hat S M0=}
t^2M_0\widehat{S}(\boldsymbol{\theta})M_0=f_0b(\mathbf{k})^*g^0b(\mathbf{k})f_0,\quad t=\vert\mathbf{k}\vert,\quad \boldsymbol{\theta}=\mathbf{k}/\vert\mathbf{k}\vert.
\end{equation}

Let $\mathcal{A}^0$ be the following operator in $L_2(\mathbb{R}^d;\mathbb{C}^n)$:
\begin{equation}
\label{A0 no hat}
\mathcal{A}^0=f_0b(\mathbf{D})^*g^0b(\mathbf{D})f_0,\quad\mathrm{Dom}\,\mathcal{A}^0=H^2(\mathbb{R}^d;\mathbb{C}^n).
\end{equation}
Let $\mathcal{A}^0(\mathbf{k})$ be the corresponding operator family in $L_2(\Omega;\mathbb{C}^n)$ given by the expression
\begin{equation}
\label{A0(k)= no hat}
\mathcal{A}^0(\mathbf{k})=f_0b(\mathbf{D}+\mathbf{k})^*g^0b(\mathbf{D}+\mathbf{k})f_0
\end{equation}
with the periodic boundary conditions. By \eqref{SP=A0P}, \eqref{A^0(k)>=}, \eqref{f0<=}, and the idenity $c_*=\widehat{c}_*\Vert f^{-1}\Vert ^{-2}_{L_\infty}$, the symbol of the operator $\mathcal{A}^0$ satisfies the estimate
\begin{equation}
\label{f_0 dots >=}
f_0b(\mathbf{k})^*g^0b(\mathbf{k})f_0\geqslant c_*\vert \mathbf{k}\vert ^2 \mathbf{1}_n,\quad\mathbf{k}\in\mathbb{R}^d.
\end{equation}
Hence, using the Fourier series representation for the operator $\mathcal{A}^0(\mathbf{k})$ and \eqref{6.5 from DSu}, we deduce that
\begin{equation}
\label{8.8a}
\mathcal{A}^0(\mathbf{k})\geqslant c_*\vert \mathbf{k}\vert ^2 I,\quad \mathbf{k}\in\mathrm{clos}\,\widetilde{\Omega}.
\end{equation}
By \eqref{P L2}, \eqref{t2 M0 hat S M0=}, and \eqref{A0(k)= no hat}, we obtain $t^2M_0\widehat{S}(\boldsymbol{\theta})M_0\widehat{P}=\mathcal{A}^0(\mathbf{k})\widehat{P}$, whence
\begin{equation}
\label{M0 sin}
\begin{split}
M_0&(t^2 M_0 \widehat{S}(\boldsymbol{\theta})M_0)^{-1/2}\sin \left(\varepsilon^{-1}\tau (t^2 M_0\widehat{S}(\boldsymbol{\theta})M_0)^{1/2}\right)M_0^{-1}\widehat{P}
\\
&=f_0\mathcal{A}^0 (\mathbf{k})^{-1/2}\sin (\varepsilon ^{-1}\tau \mathcal{A}^0(\mathbf{k})^{1/2})f_0^{-1}\widehat{P}.
\end{split}
\end{equation}

In accordance with \cite[Sec.~5]{BSu05}, the role of $\widehat{Z}_Q$ is played by the operator 
\begin{equation}
\label{Z_Q(theta)=}
\widehat{Z}_Q(\boldsymbol{\theta})=\Lambda _Q b(\boldsymbol{\theta})\widehat{P}.
\end{equation}
Here $\Lambda _Q$ is the operator of multiplication by the $\Gamma$-periodic $(n\times m)$-matrix-valued solution $\Lambda _Q(\mathbf{x})$ of the problem 
\begin{equation*}
b(\mathbf{D})^*g(\mathbf{x})\left(b(\mathbf{D})\Lambda _Q(\mathbf{x})+\mathbf{1}_m\right)=0,\quad\int _\Omega Q(\mathbf{x})\Lambda_Q(\mathbf{x})\,d\mathbf{x}=0.
\end{equation*}
Note that
\begin{equation}
\label{Lambda_Q=Lambda +LambdaQ0}
\Lambda _Q(\mathbf{x})=\Lambda (\mathbf{x})+\Lambda _Q^0,\quad \Lambda _Q^0:=-\left(\overline{Q}\right)^{-1}\left(\overline{Q\Lambda}\right),
\end{equation}
where $\Lambda$ is the $\Gamma$-periodic solution of problem \eqref{Lambda problem}. From \eqref{Z_Q(theta)=} it follows that
\begin{equation*}
t\widehat{Z}_Q(\boldsymbol{\theta})\widehat{P}=\Lambda _Qb(\mathbf{k})\widehat{P}=\Lambda _Qb(\mathbf{D}+\mathbf{k})\widehat{P}.
\end{equation*}

\subsection{Estimates in the case where $\vert\mathbf{k}\vert \leqslant t_0$} 

Consider the operator $\mathcal{H}_0=-\Delta $ acting in $L_2(\mathbb{R}^d;\mathbb{C}^n)$. Under the Gelfand transformation, this operator is decomposed into the direct integral of the operators $\mathcal{H}_0(\mathbf{k})$ acting in $L_2(\Omega;\mathbb{C}^n)$ and given by the differential expression $\vert \mathbf{D}+\mathbf{k}\vert ^2$ with the periodic boundary conditions. Denote
\begin{equation}
\label{R(k,eps)}
\mathcal{R}(\mathbf{k},\varepsilon):=\varepsilon ^2 (\mathcal{H}_0(\mathbf{k})+\varepsilon ^2 I)^{-1}.
\end{equation}
Obviously,
\begin{equation}
\label{R(k)P}
\mathcal{R}(\mathbf{k},\varepsilon)\widehat{P}=\varepsilon ^2 (t^2+\varepsilon ^2 )^{-1}\widehat{P},\quad\vert \mathbf{k}\vert =t .
\end{equation}

In order to approximate the operator  $f\mathcal{A}(\mathbf{k})^{-1/2}\sin (\varepsilon ^{-1}\tau \mathcal{A}(\mathbf{k})^{1/2})f^{-1}$, we apply Theorem~\ref{Theorem sin sandwiched abstract}. We only need to specify the constants in estimates. The constants $c_*$, $\delta$, and $t_0$ are defined by \eqref{A(k)>=}, \eqref{delta L2}, and \eqref{t0 L2}. Using estimate \eqref{X_1(theta)<=}, 
we choose the following values of constants from \eqref{F-P}, \eqref{F(t)=P+tF_1+F_2(t)}, and \eqref{C_3 abstract}:
\begin{align*}
{C}_1:&=\beta _1\delta ^{-1/2}\alpha _1^{1/2}\Vert g\Vert ^{1/2}_{L_\infty}\Vert f\Vert _{L_\infty},
\quad
{C}_2:=\beta _2\delta ^{-1}\alpha_1\Vert g\Vert _{L_\infty}\Vert f\Vert ^2_{L_\infty},
\\
{C}_3:&=\beta _3\delta ^{-1/2}\alpha _1^{1/2}\Vert g\Vert ^{1/2}_{L_\infty}\Vert f\Vert _{L_\infty}(1+c_*^{-1}\alpha_1\Vert g\Vert _{L_\infty}\Vert f\Vert ^2_{L_\infty}).
\end{align*}
Similarly, in accordance with \eqref{C_4 abstract} and \eqref{A^1/2F_2} we define
\begin{align*}
{C}_4:&=\beta _4\delta ^{-1/2}\alpha_1\Vert g\Vert _{L_\infty}\Vert f\Vert ^2_{L_\infty}(1+c_*^{-1/2}\alpha_1^{1/2}\Vert g\Vert ^{1/2}_{L_\infty}\Vert f\Vert _{L_\infty}),
\\
{C}_5:&=\beta _5\delta ^{-1/2}\alpha _1\Vert g\Vert _{L_\infty}\Vert f\Vert ^2_{L_\infty}.
\end{align*}
Using these ${C}_1$, ${C}_3$, ${C}_4$, and ${C}_5$, according to \eqref{C_6 abstract}, \eqref{C_7 abstract}, \eqref{A-1/2sin F 2}, and \eqref{Th corr exp}, we put
\begin{align*}
{C}_6:&=4 \pi ^{-1}c_*^{-1/2}{C}_3,\quad
{C}_8:=\max\lbrace {C}_4c_*^{-1/2};{C}_6\rbrace,\\
{C}_7:&={C}_8+c_*^{-1/2}{C}_1,\quad
{C}_9:={C}_1+c_*^{-1/2}{C}_5+{C}_8(\alpha_1^{1/2}\Vert g\Vert ^{1/2}_{L_\infty}\Vert f\Vert _{L_\infty}+{C}_4t_0).
\end{align*}
By Theorem~\ref{Theorem sin sandwiched abstract}, taking \eqref{M0 sin}, \eqref{Z_Q(theta)=}, and \eqref{R(k)P}  into account, we have
\begin{align}
\label{I, k<t0-1}
\begin{split}
&\Bigl\Vert  \Bigl( f\mathcal{A}(\mathbf{k})^{-1/2}\sin (\varepsilon ^{-1}\tau \mathcal{A}(\mathbf{k})^{1/2})f^{-1}
-f_0 \mathcal{A}^0(\mathbf{k})^{-1/2}\sin (\varepsilon^{-1}\tau \mathcal{A}^0(\mathbf{k})^{1/2})f_0^{-1}\Bigr)
\\
&\times
\mathcal{R}(\mathbf{k},\varepsilon)^{1/2}\widehat{P}\Bigr\Vert _{L_2(\Omega )\rightarrow L_2(\Omega )}
\leqslant {C}_7\Vert f\Vert  _{L_\infty}\Vert f^{-1}\Vert _{L_\infty}(1+\vert\tau\vert ),\quad \tau \in\mathbb{R},\quad \varepsilon >0,\quad\vert \mathbf{k}\vert \leqslant t_0,
\end{split}
\\
\label{II, k<t0}
\begin{split}
&\Bigl\Vert \widehat{\mathcal{A}}(\mathbf{k})^{1/2}\Bigl(
f\mathcal{A}(\mathbf{k})^{-1/2}\sin (\varepsilon ^{-1}\tau \mathcal{A}(\mathbf{k})^{1/2})f^{-1}
\\
&-(I+\Lambda _Q b(\mathbf{D}+\mathbf{k}))f_0\mathcal{A}^0(\mathbf{k})^{-1/2}\sin (\varepsilon ^{-1}\tau \mathcal{A}^0(\mathbf{k})^{1/2})f_0^{-1}
\Bigr)\mathcal{R}(\mathbf{k},\varepsilon)\widehat{P}\Bigr\Vert _{L_2(\Omega)\rightarrow L_2(\Omega)}
\\
&\leqslant {C}_9 \Vert f^{-1}\Vert _{L_\infty}\varepsilon (1+\vert \tau\vert ),\quad \tau \in\mathbb{R},\quad \varepsilon >0,\quad\vert \mathbf{k}\vert \leqslant t_0.
\end{split}
\end{align} 

Using \eqref{Lambda_Q=Lambda +LambdaQ0}, we show that $\Lambda _Q$ can be replaced by $\Lambda$ in \eqref{II, k<t0}. Only the constant in the estimate will change under such replacement. Indeed, due to the presence of the projection $\widehat{P}$, taking \eqref{<b^*b<}, \eqref{hat A(k)=}, \eqref{f0<=}, \eqref{R(k)P},  and the inequality $\vert \sin x\vert/\vert x\vert \leqslant 1$ into account, we have
\begin{equation}
\label{A(k)1/2Lambda Q0<=}
\begin{split}
&\Bigl\Vert \widehat{\mathcal{A}}(\mathbf{k})^{1/2}\Lambda _Q^0 b(\mathbf{D}+\mathbf{k})f_0\mathcal{A}^0(\mathbf{k})^{-1/2}\sin (\varepsilon ^{-1}\tau \mathcal{A}^0(\mathbf{k})^{1/2})f_0^{-1}\mathcal{R}(\mathbf{k},\varepsilon)\widehat{P}\Bigr\Vert _{L_2(\Omega)\rightarrow L_2(\Omega)}
\\
&\leqslant \Vert g\Vert ^{1/2}_{L_\infty}\Vert b(\mathbf{k})\Lambda _Q^0b(\mathbf{k})f_0\mathcal{A}^0(\mathbf{k})^{-1/2}\sin (\varepsilon ^{-1}\tau \mathcal{A}^0(\mathbf{k})^{1/2})f_0^{-1}\mathcal{R}(\mathbf{k},\varepsilon)\widehat{P}\Vert _{L_2(\Omega)\rightarrow L_2(\Omega)}
\\
&\leqslant \alpha_1\Vert g\Vert ^{1/2}_{L_\infty}\vert \Lambda _Q^0\vert \vert \mathbf{k}\vert ^2\Vert f\Vert _{L_\infty}\Vert f^{-1}\Vert _{L_\infty}\vert \tau\vert \varepsilon  (\vert \mathbf{k}\vert ^2+\varepsilon ^2)^{-1},
\quad\varepsilon>0,\quad \tau\in\mathbb{R},\quad\mathbf{k}\in\mathrm{clos}\,\widetilde{\Omega}.
\end{split}
\end{equation}
Next, according to \cite[Sec.~7]{BSu05},
\begin{equation}
\label{Lambda Q0<=}
\vert \Lambda _Q^0\vert \leqslant m^{1/2}(2r_0)^{-1}\alpha _0^{-1/2}\Vert g\Vert ^{1/2}_{L_\infty}\Vert g^{-1}\Vert ^{1/2}_{L_\infty}\Vert f\Vert ^2_{L_\infty}\Vert f^{-1}\Vert ^2_{L_\infty}.
\end{equation}
Combining \eqref{Lambda_Q=Lambda +LambdaQ0} and \eqref{II, k<t0} --\eqref{Lambda Q0<=}, we arrive at the estimate
\begin{equation}
\label{II-2}
\begin{split}
&\Bigl\Vert
\widehat{\mathcal{A}}(\mathbf{k})^{1/2}\Bigl(
f\mathcal{A}(\mathbf{k})^{-1/2}\sin (\varepsilon ^{-1}\tau \mathcal{A}(\mathbf{k})^{1/2})f^{-1}\\
&-\left(I+\Lambda b(\mathbf{D}+\mathbf{k})\right)f_0\mathcal{A}^0(\mathbf{k})^{-1/2}\sin (\varepsilon ^{-1}\tau \mathcal{A}^0(\mathbf{k})^{1/2})f_0^{-1}\Bigr)\mathcal{R}(\mathbf{k},\varepsilon)\widehat{P}\Bigr\Vert _{L_2(\Omega)\rightarrow L_2(\Omega)}
\\
&\leqslant
{C}_{10}\varepsilon (1+\vert \tau\vert),\quad \varepsilon >0,\quad \tau\in\mathbb{R},\quad \mathbf{k}\in \mathrm{clos}\,\widetilde{\Omega},\quad \vert \mathbf{k}\vert \leqslant t_0,
\end{split}
\end{equation}
where ${C}_{10}:={C}_9\Vert f^{-1}\Vert _{L_\infty}+m^{1/2}(2r_0)^{-1}\alpha _0^{-1/2}\alpha _1\Vert g\Vert _{L_\infty}\Vert g^{-1}\Vert ^{1/2}_{L_\infty}\Vert f\Vert ^3_{L_\infty}\Vert f^{-1}\Vert ^3_{L_\infty}$. 

\subsection{Approximations for $\vert \mathbf{k}\vert >t_0$}

By \eqref{A(k)>=} and \eqref{8.8a},
\begin{equation}
\label{8.22a}
\begin{split}
\Vert \mathcal{A}(\mathbf{k})^{-1/2}\Vert _{L_2(\Omega)\rightarrow L_2(\Omega)}\leqslant c_*^{-1/2}t_0^{-1},\quad\Vert \mathcal{A}^0(\mathbf{k})^{-1/2}\Vert _{L_2(\Omega)\rightarrow L_2(\Omega)}\leqslant c_*^{-1/2}t_0^{-1},
\\ 
\mathbf{k}
\in\mathrm{clos}\,\widetilde{\Omega},\quad \vert \mathbf{k}\vert > t_0.
\end{split}
\end{equation}
By \eqref{R(k)P},
\begin{equation}
\label{8.22b}
\Vert \mathcal{R}(\mathbf{k},\varepsilon)\widehat{P}\Vert _{L_2(\Omega)\rightarrow L_2(\Omega)}
\leqslant 1,\quad\mathbf{k}\in\mathrm{clos}\,\widetilde{\Omega}.
\end{equation}
Combining  \eqref{f0<=} and \eqref{8.22a}, \eqref{8.22b},  we obtain
\begin{equation}
\label{I, k>t0}
\begin{split}
\Bigl\Vert& \Bigl( f \mathcal{A}(\mathbf{k})^{-1/2}\sin (\varepsilon ^{-1}\tau \mathcal{A}(\mathbf{k})^{1/2})f^{-1}
-f_0\mathcal{A}^0(\mathbf{k})^{-1/2}\sin (\varepsilon ^{-1}\tau \mathcal{A}^0(\mathbf{k})^{1/2})f_0^{-1}
\Bigr)
\\
&\times \mathcal{R}(\mathbf{k},\varepsilon)^{1/2}\widehat{P}\Bigr\Vert _{L_2(\Omega )\rightarrow L_2(\Omega)}
\leqslant 2c_*^{-1/2}t_0^{-1}\Vert f\Vert _{L_\infty}\Vert f^{-1}\Vert _{L_\infty},
\end{split}
\end{equation}
$\varepsilon >0$, $\tau\in\mathbb{R}$, $\mathbf{k}\in\mathrm{clos}\,\widetilde{\Omega}$, $\vert \mathbf{k}\vert > t_0$.
Bringing together \eqref{I, k<t0-1} and \eqref{I, k>t0}, we conclude that
\begin{equation}
\label{8.3_I}
\begin{split}
\Bigl\Vert &\Bigl(
f\mathcal{A}(\mathbf{k})^{-1/2}\sin (\varepsilon ^{-1}\tau \mathcal{A}(\mathbf{k})^{1/2})f^{-1}
-f_0\mathcal{A}^0(\mathbf{k})^{-1/2}\sin (\varepsilon ^{-1}\tau \mathcal{A}^0(\mathbf{k})^{1/2})f_0^{-1}
\Bigr)\\
&\times
 \mathcal{R}(\mathbf{k},\varepsilon)^{1/2}\widehat{P}\Bigr\Vert _{L_2(\Omega )\rightarrow L_2(\Omega)}
\leqslant \max\lbrace C_7;
2c_*^{-1/2}t_0^{-1}\rbrace\Vert f\Vert _{L_\infty}\Vert f^{-1}\Vert _{L_\infty}(1+\vert \tau\vert),
\end{split}
\end{equation}
$\varepsilon >0$, $\tau\in\mathbb{R}$, $\mathbf{k}\in\mathrm{clos}\,\widetilde{\Omega}$.

Now, we proceed to estimation of the operator under the norm sign in \eqref{II-2} for $\vert\mathbf{k}\vert >t_0$. By \eqref{R(k)P} and the elementary inequality $t^2+\varepsilon ^2 \geqslant 2\varepsilon  {t}_0$, $t >{t}_0$, we have
\begin{equation}
\label{analog of 7.9}
\Vert \mathcal{R}(\mathbf{k},\varepsilon)\widehat{P}\Vert _{L_2(\Omega)\rightarrow L_2(\Omega)}
\leqslant(2t_0)^{-1}\varepsilon ,\quad\varepsilon >0,\quad \mathbf{k}\in\mathrm{clos}\,\widetilde{\Omega},\quad\vert \mathbf{k}\vert >t_0.
\end{equation}
By \eqref{A(k) and hat A(k)} and \eqref{analog of 7.9},
\begin{equation}
\label{8.3 A}
\begin{split}
\Vert &\widehat{\mathcal{A}}(\mathbf{k})^{1/2}f\mathcal{A}(\mathbf{k})^{-1/2}\sin (\varepsilon ^{-1}\tau \mathcal{A}(\mathbf{k})^{1/2})f^{-1}\mathcal{R}(\mathbf{k},\varepsilon)\widehat{P}\Vert _{L_2(\Omega)\rightarrow L_2(\Omega)}
\\
&=\Vert \sin (\varepsilon ^{-1}\tau \mathcal{A}(\mathbf{k})^{1/2})f^{-1}\mathcal{R}(\mathbf{k},\varepsilon)\widehat{P}\Vert _{L_2(\Omega)\rightarrow L_2(\Omega)}\\
&\leqslant \varepsilon (2t_0)^{-1}\Vert f ^{-1}\Vert _{L_\infty},
\quad
\varepsilon >0,\quad\tau\in\mathbb{R},\quad\mathbf{k}\in\mathrm{clos}\,\widetilde{\Omega},\quad \vert \mathbf{k}\vert > t_0.
\end{split}
\end{equation}

From \eqref{g^0<=}, \eqref{f0<=}, \eqref{A0(k)= no hat}, and \eqref{analog of 7.9} it follows that
\begin{equation}
\label{8.3 B}
\begin{split}
\Vert &\widehat{\mathcal{A}}(\mathbf{k})^{1/2}f_0\mathcal{A}^0(\mathbf{k})^{-1/2}\sin (\varepsilon ^{-1}\tau \mathcal{A}^0(\mathbf{k})^{1/2})f_0 ^{-1}\mathcal{R}(\mathbf{k},\varepsilon)\widehat{P}\Vert _{L_2(\Omega)\rightarrow L_2(\Omega)}
\\
&\leqslant\varepsilon (2t_0)^{-1}\Vert g^{1/2}b(\mathbf{D}+\mathbf{k})f_0\sin (\varepsilon ^{-1}\tau \mathcal{A}^0(\mathbf{k})^{1/2})\mathcal{A}^0(\mathbf{k})^{-1/2}f_0 ^{-1}\widehat{P}\Vert _{L_2(\Omega)\rightarrow L_2(\Omega)}
\\
&\leqslant \varepsilon (2 t_0)^{-1}\Vert g\Vert ^{1/2}_{L_\infty}\Vert g^{-1}\Vert ^{1/2}_{L_\infty}\Vert \mathcal{A}^0(\mathbf{k})^{1/2}\sin (\varepsilon ^{-1}\tau \mathcal{A}^0(\mathbf{k})^{1/2})\mathcal{A}^0(\mathbf{k})^{-1/2}f_0^{-1} \widehat{P}\Vert _{L_2(\Omega)\rightarrow L_2(\Omega)}
\\
&\leqslant
\varepsilon (2 t_0)^{-1}\Vert g\Vert ^{1/2}_{L_\infty}\Vert g^{-1}\Vert ^{1/2}_{L_\infty}\Vert f ^{-1}\Vert _{L_\infty},
\quad\varepsilon >0,\quad\tau\in\mathbb{R},\quad \mathbf{k}\in\mathrm{clos}\,\widetilde{\Omega},\quad\vert \mathbf{k}\vert >t_0.
\end{split}
\end{equation}

Next, we have
\begin{equation*}
\begin{split}
&\widehat{\mathcal{A}}(\mathbf{k})^{1/2}\Lambda b(\mathbf{D}+\mathbf{k})f_0\mathcal{A}^0(\mathbf{k})^{-1/2}\sin (\varepsilon ^{-1}\tau \mathcal{A}^0(\mathbf{k})^{1/2})f_0 ^{-1}\mathcal{R}(\mathbf{k},\varepsilon)\widehat{P}
\\
&=\left(\widehat{\mathcal{A}}(\mathbf{k})^{1/2}\Lambda \widehat{P}_m\right)b(\mathbf{D}+\mathbf{k})f_0\mathcal{A}^0(\mathbf{k})^{-1/2}\sin (\varepsilon ^{-1}\tau \mathcal{A}^0(\mathbf{k})^{1/2})f_0 ^{-1}\mathcal{R}(\mathbf{k},\varepsilon)\widehat{P},
\end{split}
\end{equation*}
where $\widehat{P}_m$ is the orthogonal projection of the space $\mathfrak{H}_*=L_2(\Omega;\mathbb{C}^m)$ onto the subspace of constants. According to \cite[(6.22)]{BSu06}, 
\begin{equation}
\label{7.11 c}
\Vert \widehat{\mathcal{A}}(\mathbf{k})^{1/2}\Lambda \widehat{P}_m\Vert _{L_2(\Omega)\rightarrow L_2(\Omega)}\leqslant C_\Lambda ,\quad\mathbf{k}\in\widetilde{\Omega},
\end{equation}
where the constant $C_\Lambda$ depends only on $m$, $\alpha _0$, $\alpha _1$, $\Vert g\Vert _{L_\infty}$, $\Vert g^{-1}\Vert _{L_\infty}$, and the parameters of the lattice $\Gamma$.

By \eqref{g^0<=},  \eqref{f0<=}, \eqref{A0(k)= no hat},  \eqref{analog of 7.9}, and \eqref{7.11 c},
\begin{equation}
\label{8.3 C}
\begin{split}
\Vert &\widehat{\mathcal{A}}(\mathbf{k})^{1/2}\Lambda b(\mathbf{D}+\mathbf{k})f_0\mathcal{A}^0(\mathbf{k})^{-1/2}\sin (\varepsilon ^{-1}\tau \mathcal{A}^0(\mathbf{k})^{1/2})f_0 ^{-1}\mathcal{R}(\mathbf{k},\varepsilon)\widehat{P}\Vert _{L_2(\Omega)\rightarrow L_2(\Omega)}
\\
&\leqslant C_\Lambda \Vert g^{-1}\Vert ^{1/2}_{L_\infty}\Vert f  ^{-1}\Vert _{L_\infty} (2t_0)^{-1}\varepsilon,
\quad\varepsilon >0,\quad\tau\in\mathbb{R},\quad\mathbf{k}\in\mathrm{clos}\,\widetilde{\Omega},\quad \vert \mathbf{k}\vert > t_0.
\end{split}
\end{equation}

Combining \eqref{II-2}, \eqref{8.3 A}, \eqref{8.3 B}, and \eqref{8.3 C}, we conclude that
\begin{equation}
\label{8.3 with C11}
\begin{split}
\Bigl\Vert &\widehat{\mathcal{A}}(\mathbf{k})^{1/2}\Bigl(
f\mathcal{A}(\mathbf{k})^{-1/2}\sin (\varepsilon ^{-1}\tau \mathcal{A}(\mathbf{k})^{1/2})f^{-1}
\\
&-
(I+\Lambda b(\mathbf{D}+\mathbf{k}))f_0\mathcal{A}^0(\mathbf{k})^{-1/2}\sin (\varepsilon ^{-1}\tau \mathcal{A}^0(\mathbf{k})^{1/2})f_0^{-1}
\Bigr)\mathcal{R}(\mathbf{k},\varepsilon)\widehat{P}\Bigr\Vert _{L_2(\Omega)\rightarrow L_2(\Omega)}
\\
&\leqslant
{C}_{11}\varepsilon (1+\vert \tau\vert ),\quad \varepsilon >0,\quad\tau\in\mathbb{R},\quad\mathbf{k}\in\mathrm{clos}\,\widetilde{\Omega}.
\end{split}
\end{equation}
Here ${C}_{11}:=\max\left\lbrace {C}_{10};
(2t_0)^{-1}\Vert f ^{-1}\Vert _{L_\infty}\left(1+\Vert g\Vert ^{1/2}_{L_\infty}\Vert g^{-1}\Vert ^{1/2}_{L_\infty}+C_\Lambda \Vert g^{-1}\Vert ^{1/2}_{L_\infty}\right)\right\rbrace$. 

\subsection{Removal of the operator $\widehat{P}$}
Now, we show that, in the operator under the norm sign in 
\eqref{8.3_I} the projection $\widehat{P}$ can be replaced by the identity operator. 
After such replacement, only the constant in the estimate will be different. 
To show this, we estimate the norm of the operator $\mathcal{R}(\mathbf{k},\varepsilon)^{1/2}(I-\widehat{P})$ by using the discrete Fourier transform:
\begin{equation}
\label{R(I-P)}
\Vert \mathcal{R}(\mathbf{k},\varepsilon)^{1/2}(I-\widehat{P})\Vert _{L_2(\Omega)\rightarrow L_2(\Omega )}
=\max _{0\neq \mathbf{b}\in\widetilde{\Gamma}}\varepsilon (\vert\mathbf{b}+\mathbf{k}\vert ^2+\varepsilon ^2)^{-1/2}\leqslant \varepsilon r_0^{-1},\quad \varepsilon >0,\quad\mathbf{k}\in\mathrm{clos}\,\widetilde{\Omega}.
\end{equation}
Next, applying the spectral theorem and the elementary inequality 
$\vert \sin x\vert/\vert x\vert \leqslant 1$, $x\in\mathbb{R}$, we conclude that
\begin{equation}
\label{A^-1-2sin <= grubo}
\Vert   {\mathcal{A}}(\mathbf{k})^{-1/2}\sin(\varepsilon ^{-1}\tau  {\mathcal{A}}(\mathbf{k})^{1/2})\Vert _{L_2(\Omega)\rightarrow L_2(\Omega)}
\leqslant \varepsilon ^{-1}\vert\tau\vert .
\end{equation}
Similarly,
\begin{equation}
\label{sin eff grubo}
\Vert   {\mathcal{A}}^0(\mathbf{k})^{-1/2}\sin(\varepsilon ^{-1}\tau  {\mathcal{A}}^0(\mathbf{k})^{1/2})\Vert _{L_2(\Omega)\rightarrow L_2(\Omega)}
\leqslant \varepsilon ^{-1}\vert\tau\vert .
\end{equation}
Bringing together \eqref{f0<=}, \eqref{R(I-P)}--\eqref{sin eff grubo}, we arrive at the estimate
\begin{equation*}
\begin{split}
\Bigl\Vert&\Bigl(
f\mathcal{A}(\mathbf{k})^{-1/2}\sin (\varepsilon ^{-1}\tau \mathcal{A}(\mathbf{k})^{1/2})f^{-1}
-f_0\mathcal{A}^0(\mathbf{k})^{-1/2}\sin (\varepsilon ^{-1}\tau \mathcal{A}^0(\mathbf{k})^{1/2})f_0^{-1} 
\Bigr)
\\
&\times
\mathcal{R}(\mathbf{k},\varepsilon)^{1/2}(I-\widehat{P})\Bigr\Vert _{L_2(\Omega)\rightarrow L_2(\Omega)}
\leqslant 2 r_0^{-1}\Vert f\Vert_{L_\infty}\Vert f^{-1}\Vert _{L_\infty}\vert\tau\vert .
\end{split}
\end{equation*}
Combining this with \eqref{8.3_I}, we see that
\begin{equation}
\label{8.4_0}
\begin{split}
\Bigl\Vert&\Bigl(
f\mathcal{A}(\mathbf{k})^{-1/2}\sin (\varepsilon ^{-1}\tau \mathcal{A}(\mathbf{k})^{1/2})f^{-1}
-f_0\mathcal{A}^0(\mathbf{k})^{-1/2}\sin (\varepsilon ^{-1}\tau \mathcal{A}^0(\mathbf{k})^{1/2})f_0^{-1} 
\Bigr)
\\
&\times
\mathcal{R}(\mathbf{k},\varepsilon)^{1/2}\Bigr\Vert _{L_2(\Omega)\rightarrow L_2(\Omega)}
\leqslant  {C}_{12}(1+\vert \tau\vert ),\quad \varepsilon >0,\quad \tau\in\mathbb{R},\quad \mathbf{k}\in\mathrm{clos}\,\widetilde{\Omega},
\end{split}
\end{equation}
where 
$
 {C}_{12}:=\left(2r_0^{-1}+\max\lbrace C_7;2c_*^{-1/2}t_0^{-1}\rbrace\right)\Vert f\Vert _{L_\infty}\Vert f^{-1}\Vert _{L_\infty}$. 
 
Now, we show that the operator $\widehat{P}$ in the principal terms of approximation \eqref{8.3 with C11} can be removed. 
Let us estimate the operator $\mathcal{R}(\mathbf{k},\varepsilon)(I-\widehat{P})$ using the discrete Fourier transform:
\begin{equation}
\label{R(I-P) for corr}
\Vert \mathcal{R}(\mathbf{k},\varepsilon)(I-\widehat{P})\Vert _{L_2(\Omega)\rightarrow L_2(\Omega)}
=\max _{0\neq\mathbf{b}\in\widetilde{\Gamma}}\varepsilon ^2 (\vert \mathbf{b}+\mathbf{k}\vert ^2+\varepsilon ^2)^{-1}\leqslant \varepsilon r_0^{-1},\quad \varepsilon >0,\quad\mathbf{k}\in\mathrm{clos}\,\widetilde{\Omega}.
\end{equation} 
By 
 \eqref{A(k) and hat A(k)} and \eqref{R(I-P) for corr},
\begin{equation}
\label{8.4 I}
\begin{split}
\Vert &\widehat{\mathcal{A}}(\mathbf{k})^{1/2}f \mathcal{A}(\mathbf{k})^{-1/2}\sin (\varepsilon^{-1}\tau \mathcal{A}(\mathbf{k})^{1/2})f^{-1}\mathcal{R}(\mathbf{k},\varepsilon)(I-\widehat{P})\Vert _{L_2(\Omega)\rightarrow L_2(\Omega)}
\\
&=\Vert \sin (\varepsilon^{-1}\tau \mathcal{A}(\mathbf{k})^{1/2})f^{-1}\mathcal{R}(\mathbf{k},\varepsilon)(I-\widehat{P})\Vert _{L_2(\Omega)\rightarrow L_2(\Omega)}
\\
&\leqslant \Vert f^{-1}\Vert _{L_\infty}\varepsilon r_0^{-1},\quad \varepsilon >0,\quad \tau\in\mathbb{R},\quad \mathbf{k}\in\mathrm{clos}\,\widetilde{\Omega}.
\end{split}
\end{equation}
Next, by \eqref{hat A(k)=}, \eqref{g^0<=},  \eqref{f0<=},  \eqref{A0(k)= no hat}, and \eqref{R(I-P) for corr},
\begin{equation}
\label{8.4 II}
\begin{split}
\Vert &\widehat{\mathcal{A}}(\mathbf{k})^{1/2}f_0\mathcal{A}^0(\mathbf{k})^{-1/2}\sin (\varepsilon ^{-1}\tau \mathcal{A}^0(\mathbf{k})^{1/2})f_0^{-1}\mathcal{R}(\mathbf{k},\varepsilon)(I-\widehat{P})\Vert_{L_2(\Omega)\rightarrow L_2(\Omega)}
\\
&\leqslant \Vert g\Vert ^{1/2}_{L_\infty}\Vert g^{-1}\Vert^{1/2}_{L_\infty}\Vert f ^{-1}\Vert_{L_\infty}\varepsilon r_0^{-1},
\quad\varepsilon >0,\quad\tau\in\mathbb{R},\quad\mathbf{k}\in\mathrm{clos}\,\widetilde{\Omega}.
\end{split}
\end{equation}
Combining \eqref{8.3 with C11}, \eqref{8.4 I}, and \eqref{8.4 II}, we have
\begin{equation}
\label{8.4 III}
\begin{split}
\Bigl\Vert &\widehat{\mathcal{A}}(\mathbf{k})^{1/2}\Bigl(
f\mathcal{A}(\mathbf{k})^{-1/2}\sin (\varepsilon ^{-1}\tau \mathcal{A}(\mathbf{k})^{1/2})f^{-1}
\\
&-
(I+\Lambda b(\mathbf{D}+\mathbf{k})\widehat{P})
f_0 \mathcal{A}^0(\mathbf{k})^{-1/2}\sin (\varepsilon ^{-1}\tau \mathcal{A}^0(\mathbf{k})^{1/2})f_0^{-1}
\Bigr)
\mathcal{R}(\mathbf{k},\varepsilon)\Bigr\Vert _{L_2(\Omega)\rightarrow L_2(\Omega)}
\\
&\leqslant
 {C}_{13}\varepsilon (1+\vert \tau\vert ),
\quad\varepsilon>0,\quad \tau\in\mathbb{R},\quad \mathbf{k}\in\mathrm{clos}\,\widetilde{\Omega},
\end{split}
\end{equation}
where $ {C}_{13}:= {C}_{11}+r_0^{-1}\Vert f ^{-1}\Vert _{L_\infty}(1+\Vert g\Vert^{1/2}_{L_\infty}\Vert g^{-1}\Vert ^{1/2}_{L_\infty})$. 

\section{Approximation of the sandwiched operator $\mathcal{A}^{-1/2}\sin (\varepsilon ^{-1}\tau\mathcal{A}^{1/2})$}

\label{Section 10}

\subsection{}
\label{Subsection 8/1 NEW}

Let $\mathcal{A}$ and $\mathcal{A}^0$ be the operators \eqref{A=} and \eqref{A0 no hat}, respectively, acting in $L_2(\mathbb{R}^d;\mathbb{C}^n)$. Recall the notation $\mathcal{H}_0=-\Delta$ and put  
$
\mathcal{R}(\varepsilon):=\varepsilon ^2(\mathcal{H}_0+\varepsilon ^2 I)^{-1}$. 
Using the Gelfand transformation, we decompose this operator into the direct integral of the operators 
 \eqref{R(k,eps)}:
\begin{equation}
\label{9.0}
\mathcal{R}(\varepsilon)=\mathcal{U}^{-1}\left(\int _{\widetilde{\Omega}}\oplus \mathcal{R}(\mathbf{k},\varepsilon)\,d\mathbf{k}\right)\mathcal{U}.
\end{equation}

In $L_2(\mathbb{R}^d;\mathbb{C}^n)$, we introduce the operator 
$\Pi :=\mathcal{U}^{-1}[\widehat{P}]\mathcal{U}$. Here $[\widehat{P}]$ is the projection in $\int _{\widetilde{\Omega}}\oplus L_2(\Omega;\mathbb{C}^n)\,d\mathbf{k}$ acting on fibers as the operator $\widehat{P}$ (see \eqref{P L2}). As was shown in \cite[(6.8)]{BSu05}, $\Pi$ is the pseudodifferential operator in $L_2(\mathbb{R}^d;\mathbb{C}^n)$ with the symbol $\chi _{\widetilde{\Omega}}(\boldsymbol{\xi})$, where $\chi _{\widetilde{\Omega}}$ is the characteristic function of the set~$\widetilde{\Omega}$. That is 
$
(\Pi \mathbf{u})(\mathbf{x})=(2\pi )^{-d/2}\int _{\widetilde{\Omega}}e^{i\langle\mathbf{x},\boldsymbol{\xi}\rangle}\widehat{\mathbf{u}}(\boldsymbol{\xi})\,d\boldsymbol{\xi}$. 
Here $\widehat{\mathbf{u}}(\boldsymbol{\xi})$ is the Fourier image of the function $\mathbf{u}\in L_2(\mathbb{R}^d;\mathbb{C}^n)$.

\begin{theorem}
Under the assumptions of Subsection~\textnormal{\ref{Subsection 8/1 NEW}}, for $\varepsilon >0$ and $\tau\in\mathbb{R}$ we have
\begin{align}
\label{10.I}
\begin{split}
\Bigl\Vert &\Bigl(
f \mathcal{A}^{-1/2}\sin (\varepsilon ^{-1}\tau\mathcal{A}^{1/2})f^{-1}
-f_0(\mathcal{A}^0)^{-1/2}\sin (\varepsilon ^{-1}\tau(\mathcal{A}^0)^{1/2})f_0^{-1}
\Bigr)
\mathcal{R}(\varepsilon )^{1/2}\Bigr\Vert _{L_2(\mathbb{R}^d)\rightarrow L_2(\mathbb{R}^d)}
\\
&\leqslant {C}_{12}(1+\vert \tau\vert),
\end{split}
\\
\label{10.II}
\begin{split}
\Bigl\Vert &
\widehat{\mathcal{A}}^{1/2}
\Bigl(
f \mathcal{A}^{-1/2}\sin (\varepsilon ^{-1}\tau\mathcal{A}^{1/2})f^{-1}\\
&-(I+\Lambda b(\mathbf{D})\Pi )f_0(\mathcal{A}^0)^{-1/2}\sin (\varepsilon ^{-1}\tau (\mathcal{A}^0)^{1/2})f_0^{-1}
\Bigr)
\mathcal{R}(\varepsilon)
\Bigr\Vert _{L_2(\mathbb{R}^d)\rightarrow L_2(\mathbb{R}^d)}
\leqslant {C}_{13}\varepsilon (1+\vert \tau\vert ).
\end{split}
\end{align}
The constants $C_{12}$ and $C_{13}$ depend only on $m$, $\alpha_0$, $\alpha_1$, $\Vert g\Vert _{L_\infty}$, $\Vert g^{-1}\Vert _{L_\infty}$, $\Vert f\Vert _{L_\infty}$, $\Vert f^{-1}\Vert _{L_\infty}$, and the parameters of the lattice $\Gamma$.
\end{theorem}

\begin{proof}
By \eqref{5.7a}, the similar identity for $\mathcal{A}^0$, and \eqref{9.0}, 
from \eqref{8.4_0} we deduce estimate \eqref{10.I}.

From  \eqref{8.4 III} via the Gelfand transform we derive inequality \eqref{10.II}.
\end{proof}

\subsection{Removal of the operator $\Pi$ in the corrector for $d\leqslant 4$}

Now, we show that the operator $\Pi$ in estimate \eqref{10.II} can be removed for $d\leqslant 4$.

\begin{theorem}
\label{Theorem 8/1 NO PI}
Under the assumptions of Subsection~\textnormal{\ref{Subsection 8/1 NEW}}, let $d\leqslant 4$. Then for $0<\varepsilon\leqslant 1$ and $\tau\in\mathbb{R}$ we have
\begin{equation}
\label{Th 8/1 NO PI}
\begin{split}
\Vert &
\widehat{\mathcal{A}}^{1/2}
\Bigl(
f \mathcal{A}^{-1/2}\sin (\varepsilon ^{-1}\tau\mathcal{A}^{1/2})f^{-1}\\
&-(I+\Lambda b(\mathbf{D}) )f_0(\mathcal{A}^0)^{-1/2}\sin (\varepsilon ^{-1}\tau (\mathcal{A}^0)^{1/2})f_0^{-1}
\Bigr)
\mathcal{R}(\varepsilon)
\Bigr\Vert _{L_2(\mathbb{R}^d)\rightarrow L_2(\mathbb{R}^d)}
\leqslant {C}_{14}\varepsilon (1+\vert \tau\vert ).
\end{split}
\end{equation}
The constant $C_{14}$ depends only on $m$, $n$, $d$, $\alpha _0$, $\alpha _1$, $\Vert g\Vert _{L_\infty}$, $\Vert g^{-1}\Vert _{L_\infty}$, $\Vert f\Vert _{L_\infty}$, $\Vert f^{-1}\Vert _{L_\infty}$, and the parameters of the lattice $\Gamma$.
\end{theorem}

To prove Theorem~\ref{Theorem 8/1 NO PI}, we need the following result, see \cite[Proposition 9.3]{Su_MMNP}.

\begin{proposition}
\label{Proposition MMNP 9/3}
Let $l=1$ for $d=1$, $l>1$ for $d=2$, and $l=d/2$ for $d\geqslant 3$. Then the operator $\widehat{\mathcal{A}}^{1/2}[\Lambda]$ is a continuous mapping of $H^l(\mathbb{R}^d;\mathbb{C}^m)$ to $L_2(\mathbb{R}^d;\mathbb{C}^n)$, and
\begin{equation}
\label{A1/2Lambda <=}
\Vert \widehat{\mathcal{A}}^{1/2}[\Lambda]\Vert _{H^l(\mathbb{R}^d)\rightarrow L_2(\mathbb{R}^d)}
\leqslant \mathcal{C}_d.
\end{equation}
Here the constant $\mathcal{C}_d$ depends only on $m$, $n$, $d$, $\alpha _0$, $\alpha _1$, $\Vert g\Vert _{L_\infty}$, $\Vert g^{-1}\Vert _{L_\infty}$, and the parameters of the lattice $\Gamma$\textnormal{;} for $d=2$ it depends also on $l$.
\end{proposition}

\begin{proof}[Proof of Theorem \textnormal{\ref{Theorem 8/1 NO PI}}.]
Taking into account that the matrix-valued function~\eqref{t2 M0 hat S M0=} is the symbol of the operator $\mathcal{A}^0$ and the function $\chi _{\widetilde{\Omega}}(\boldsymbol{\xi})$ is the symbol of $\Pi$, using \eqref{<b^*b<}, \eqref{f0<=}, and \eqref{f_0 dots >=} we have
\begin{equation}
\label{proof Th 8/1 no PI 1}
\begin{split}
\Vert &b(\mathbf{D})(I-\Pi)f_0(\mathcal{A}^0)^{-1/2}\sin (\varepsilon ^{-1}\tau (\mathcal{A}^0)^{1/2})f_0^{-1}\mathcal{R}(\varepsilon)\Vert _{L_2(\mathbb{R}^d)\rightarrow H^2(\mathbb{R}^d)}
\\
&\leqslant
\sup _{\boldsymbol{\xi}\in\mathbb{R}^d}(1+\vert \boldsymbol{\xi}\vert ^2)\vert b(\boldsymbol{\xi})\vert (1-\chi _{\widetilde{\Omega}}(\boldsymbol{\xi}))
\vert f_0\vert \vert (f_0b(\boldsymbol{\xi})^*g^0b(\boldsymbol{\xi})f_0)^{-1/2}\vert \vert f_0^{-1}\vert\varepsilon ^2(\vert \boldsymbol{\xi}\vert ^2+\varepsilon ^2)^{-1}
\\
&\leqslant
\sup _{\vert \boldsymbol{\xi}\vert \geqslant r_0}(1+\vert \boldsymbol{\xi}\vert ^2)\alpha _1^{1/2}\vert \boldsymbol{\xi}\vert \Vert f\Vert _{L_\infty}c_*^{-1/2}\vert \boldsymbol{\xi}\vert ^{-1}\Vert f^{-1}\Vert _{L_\infty}\varepsilon ^2 (\vert \boldsymbol{\xi}\vert ^2 +\varepsilon ^2)^{-1}
\\
&\leqslant\alpha _1^{1/2}c_*^{-1/2}\Vert f\Vert _{L_\infty}\Vert f^{-1}\Vert _{L_\infty}\varepsilon ^2 \sup _{\vert \boldsymbol{\xi}\vert \geqslant r_0}(1+\vert \boldsymbol{\xi}\vert ^2)\vert \boldsymbol{\xi}\vert ^{-2}
\\
&\leqslant
\alpha _1^{1/2}c_*^{-1/2}\Vert f\Vert _{L_\infty}\Vert f^{-1}\Vert _{L_\infty}(r_0^{-2}+1)\varepsilon ^2.
\end{split}
\end{equation}
For $d\leqslant 4$, we can take $l\leqslant 2$ in Proposition~\ref{Proposition MMNP 9/3}. So, combining \eqref{A1/2Lambda <=} and \eqref{proof Th 8/1 no PI 1}, we have
\begin{equation*}
\Vert \widehat{\mathcal{A}}^{1/2}[\Lambda]b(\mathbf{D})(I-\Pi)f_0(\mathcal{A}^0)^{-1/2}\sin (\varepsilon ^{-1}\tau(\mathcal{A}^0)^{1/2})f_0^{-1}\mathcal{R}(\varepsilon)\Vert _{L_2(\mathbb{R}^d)\rightarrow L_2 (\mathbb{R}^d)}
\leqslant \varepsilon ^2C_{14}',
\end{equation*}
$C_{14}':=\alpha _1 ^{1/2}c_*^{-1/2}(r_0^{-2}+1)\mathcal{C}_d\Vert f\Vert _{L_\infty}\Vert f^{-1}\Vert _{L_\infty}$. 
Combining this with \eqref{10.II}, we arrive at estimate \eqref{Th 8/1 NO PI} with $C_{14}=C_{13}+C_{14}'$.
\end{proof}

\subsection{On the possibility of removal of the operator $\Pi$ from the corrector. Sufficient conditions on $\Lambda$}

It is possible to eliminate the operator $\Pi$ for $d\geqslant 5$  by imposing the following assumption on the matrix-valued function $\Lambda$.

\begin{condition}
\label{Condition Lambda multiplier}
The operator $[\Lambda ]$ is continuous from $H^2(\mathbb{R}^d;\mathbb{C}^m)$ to $H^1(\mathbb{R}^d;\mathbb{C}^n)$.
\end{condition}

Actually, it is sufficient to impose the following condition to remove $\Pi$ for $d\geqslant 5$.

\begin{condition}
\label{Condition Lambda in Ld}
 Assume that the periodic solution $\Lambda$ of problem \eqref{Lambda problem} belongs to $L_d(\Omega)$.
\end{condition}

\begin{proposition}
\label{Proposition Ld implies multiplier}
For $d\geqslant 3$, Condition~\textnormal{\ref{Condition Lambda in Ld}} implies Condition~\textnormal{\ref{Condition Lambda multiplier}}.
\end{proposition}

To prove Proposition~\ref{Proposition Ld implies multiplier} we need the following statement. 

\begin{lemma}
\label{Proposition d>=3 Lambda}
 Let $d\geqslant 3$. 
Assume that Condition~\textnormal{\ref{Condition Lambda in Ld}} is satisfied.  
 Then the operator $g^{1/2}b(\mathbf{D})[\Lambda ]$ is a continuous mapping of $H^2(\mathbb{R}^d;\mathbb{C}^m)$ to $L_2(\mathbb{R}^d;\mathbb{C}^m)$ and
 \begin{equation}
 \label{g1/2b(D)Lambda H2 to L2 Lm}
\Vert g^{1/2}b(\mathbf{D})[\Lambda ]\Vert _{H^2(\mathbb{R}^d)\rightarrow L_2(\mathbb{R}^d)}\leqslant \mathfrak{C}_\Lambda .
 \end{equation}
 The constant $\mathfrak{C}_\Lambda$ depends only on $d$, $\alpha _0$, $\alpha _1$, $\Vert g\Vert _{L_\infty}$, $\Vert g^{-1}\Vert _{L_\infty}$, $\Vert \Lambda\Vert _{L_d(\Omega)}$, and the parameters of the lattice $\Gamma$.
\end{lemma}

\begin{proof}
The proof is quite similar to the proof of Proposition 8.8 from \cite{SuAA10}.  

Let $\mathbf{v}_j(\mathbf{x})$, $j=1,\dots,m$, be the columns of the matrix $\Lambda (\mathbf{x})$. In other words, $\mathbf{v}_j$ is the $\Gamma$-periodic solution of the problem
\begin{equation}
\label{v_j columns Lambda problem}
b(\mathbf{D})^*g(\mathbf{x})\left( b(\mathbf{D})\mathbf{v}_j(\mathbf{x})+\mathbf{e}_j\right)=0,\quad
\int _\Omega \mathbf{v}_j(\mathbf{x})\,d\mathbf{x}=0.
\end{equation}
Here $\lbrace \mathbf{e}_j\rbrace _{j=1}^m$ is the standard orthonormal basis in $\mathbb{C}^m$. Let ${u}\in H^2 (\mathbb{R}^d)$. Then
\begin{equation}
\label{v_ju tozd}
g^{1/2}b(\mathbf{D})(\mathbf{v}_j u)=g^{1/2}\left(b(\mathbf{D})\mathbf{v}_j\right)u
+\sum _{l=1}^d g^{1/2}b_l (D_lu)\mathbf{v}_j.
\end{equation}
We estimate the second term on the right-hand side of \eqref{v_ju tozd}:
\begin{equation}
\label{v_ju tozd second summand}
\left\Vert \sum _{l=1}^d g^{1/2}b_l (D_lu)\mathbf{v}_j\right\Vert _{L_2(\mathbb{R}^d)}
\leqslant\Vert g\Vert ^{1/2}_{L_\infty}\alpha _1^{1/2}d^{1/2}\left(\int _{\mathbb{R}^d}\vert \mathbf{D}u\vert ^2\vert \mathbf{v}_j\vert ^2\,d\mathbf{x}\right)^{1/2}.
\end{equation}
Next,
\begin{equation}
\label{v_ju tozd second summand est start}
\int _{\mathbb{R}^d}\vert \mathbf{D}u\vert ^2\vert \mathbf{v}_j\vert ^2\,d\mathbf{x}
=\sum _{\mathbf{a}\in\Gamma}\int _{\Omega +\mathbf{a}}\vert \mathbf{D}u\vert ^2\vert \mathbf{v}_j\vert ^2\,d\mathbf{x}.
\end{equation}
By the H\"older inequality with indices $s=d/2$ and $s'=d/(d-2)$,
\begin{equation}
\int _{\Omega +\mathbf{a}}\vert \mathbf{D}u\vert ^2\vert \mathbf{v}_j\vert ^2\,d\mathbf{x}
\leqslant
\left(\int _\Omega \vert \mathbf{v}_j\vert ^d\,d\mathbf{x}\right)^{2/d}
\left(\int _{\Omega +\mathbf{a}}\vert \mathbf{D}u\vert ^{2d/(d-2)}\,d\mathbf{x}\right)^{(d-2)/d}.
\end{equation}
By the continuous embedding $H^1(\Omega)\hookrightarrow L_{2d/(d-2)}(\Omega)$,
\begin{equation}
\label{v_ju tozd second summand est fin}
\left(\int _{\Omega +\mathbf{a}}\vert \mathbf{D}u\vert ^{2d/(d-2)}\,d\mathbf{x}\right)^{(d-2)/2d}
\leqslant C_\Omega \Vert \mathbf{D}u\Vert _{H^1(\Omega +\mathbf{a})}.
\end{equation}
The embedding constant $C_\Omega$ depends only on $d$ and $\Omega$ (i. e., on the lattice $\Gamma$). From \eqref{v_ju tozd second summand est start}--\eqref{v_ju tozd second summand est fin} it follows that
\begin{equation}
\label{Du^2vj^2 <=}
\int _{\mathbb{R}^d}\vert \mathbf{D}u\vert ^2\vert \mathbf{v}_j\vert ^2\,d\mathbf{x}
\leqslant C_\Omega ^2 \Vert \mathbf{v}_j\Vert ^2 _{L_d(\Omega)}\Vert u\Vert ^2_{H^2(\mathbb{R}^d)}.
\end{equation}
Using \eqref{v_ju tozd second summand}, from \eqref{Du^2vj^2 <=} we derive the estimate
\begin{equation}
\label{8.13a new}
\left\Vert \sum _{l=1}^d g^{1/2}b_l (D_lu)\mathbf{v}_j\right\Vert _{L_2(\mathbb{R}^d)}
\leqslant
\Vert g\Vert ^{1/2}_{L_\infty}\alpha _1^{1/2}d^{1/2}C_\Omega\Vert \mathbf{v}_j\Vert _{L_d(\Omega)}\Vert u\Vert _{H^2(\mathbb{R}^d)}.
\end{equation}

Next, equation \eqref{v_j columns Lambda problem} implies that
\begin{equation}
\label{vj int tozd}
\int _{\mathbb{R}^d}\Bigl(\langle g(\mathbf{x})b(\mathbf{D})\mathbf{v}_j,b(\mathbf{D})\mathbf{w}\rangle
+\sum _{l=1}^d\langle b_l^* g(\mathbf{x})\mathbf{e}_j,D_l\mathbf{w}\rangle\Bigr)\,d\mathbf{x}=0
\end{equation}
for any $\mathbf{w}\in H^1(\mathbb{R}^d;\mathbb{C}^n)$ such that $\mathbf{w}(\mathbf{x})=0$ for $\vert \mathbf{x}\vert >R$ (with some $R>0$).

Let $u\in C_0^\infty (\mathbb{R}^d)$. We put $\mathbf{w}(\mathbf{x})=\vert u(\mathbf{x})\vert ^2\mathbf{v}_j(\mathbf{x})$. Then
\begin{equation*}
b(\mathbf{D})\mathbf{w}=\vert u\vert ^2 b(\mathbf{D})\mathbf{v}_j+\sum _{l=1}^d b_l (D_l \vert u\vert ^2)\mathbf{v}_j.
\end{equation*}
Substituting this expression into \eqref{vj int tozd}, we obtain
\begin{equation*}
\int _{\mathbb{R}^d}\Bigl(\langle g(\mathbf{x})b(\mathbf{D})\mathbf{v}_j,\vert u\vert ^2b(\mathbf{D})\mathbf{v}_j+\sum _{l=1}^d b_l (D_l\vert u\vert ^2)\mathbf{v}_j\rangle
+\sum _{l=1}^d\langle b_l^* g(\mathbf{x})\mathbf{e}_j,D_l(\vert u\vert ^2 \mathbf{v}_j)\rangle\Bigr)\,d\mathbf{x}=0.
\end{equation*}
Hence,
\begin{equation}
\label{J0=}
J_0:=\int _{\mathbb{R}^d}\vert g^{1/2}b(\mathbf{D})\mathbf{v}_j\vert ^2 \vert u\vert ^2\,d\mathbf{x}=J_1+J_2,
\end{equation}
where
\begin{align*}
J_1&=-\int _{\mathbb{R}^d}\langle g^{1/2}b(\mathbf{D})\mathbf{v}_j,\sum _{l=1}^d g^{1/2}b_l(D_l\vert u\vert ^2)\mathbf{v}_j\rangle\,d\mathbf{x},
\\
J_2&=-\int _{\mathbb{R}^d}\sum _{l=1}^d\langle b_l^*g(\mathbf{x})\mathbf{e}_j,D_l(\vert u\vert ^2\mathbf{v}_j)\rangle\,d\mathbf{x}
=
-\int _{\mathbb{R}^d}\sum _{l=1}^d\langle b_l^*g(\mathbf{x})\mathbf{e}_j,D_l(\mathbf{v}_ju)u^*+\mathbf{v}_ju(D_lu^*)\rangle\,d\mathbf{x}.
\end{align*}
By \eqref{b_j <=},
\begin{equation*}
\begin{split}
\vert J_1\vert &\leqslant \Vert g\Vert _{L_\infty}^{1/2}\alpha _1 ^{1/2}d^{1/2}\int _{\mathbb{R}^d} 2\vert g^{1/2}b(\mathbf{D})\mathbf{v}_j\vert \vert {u}\vert \vert \mathbf{D}u\vert \vert \mathbf{v}_j\vert \,d\mathbf{x}
\\
&\leqslant
\frac{1}{2}\int _{\mathbb{R}^d}\vert g^{1/2}b(\mathbf{D})\mathbf{v}_j\vert ^2\vert u\vert ^2\,d\mathbf{x}
+2\Vert g\Vert _{L_\infty}\alpha _1d\int _{\mathbb{R}^d}\vert \mathbf{D}u\vert ^2\vert \mathbf{v}_j\vert ^2\,d\mathbf{x}.
\end{split}
\end{equation*}
Combining this with \eqref{Du^2vj^2 <=}, we see that
\begin{equation}
\label{J1 estimate}
\vert J_1\vert \leqslant\frac{1}{2}J_0+2\Vert g\Vert _{L_\infty}\alpha _1 dC_\Omega^2\Vert \mathbf{v}_j\Vert ^2 _{L_d(\Omega)}\Vert u\Vert ^2_{H^2(\mathbb{R}^d)}.
\end{equation}

Now we proceed to estimating the term $J_2$. By \eqref{b_j <=},
\begin{equation*}
\int _{\mathbb{R}^d}\vert b_l^* g(\mathbf{x})\mathbf{e}_j\vert ^2\vert u\vert^2\,d\mathbf{x}
\leqslant \alpha _1\Vert g\Vert ^2_{L_\infty}\Vert u\Vert ^2_{L_2(\mathbb{R}^d)}.
\end{equation*}
Then
\begin{equation*}
\begin{split}
\vert J_2\vert &\leqslant \sum _{l=1}^d\Vert ub_l^*g\mathbf{e}_j\Vert _{L_2(\mathbb{R}^d)}\left(\Vert D_l(\mathbf{v}_ju)\Vert _{L_2(\mathbb{R}^d)}
+\Vert \mathbf{v}_j(D_lu^*)\Vert _{L_2(\mathbb{R}^d)}\right)
\\
&\leqslant \mu \Vert \mathbf{D}(\mathbf{v}_ju)\Vert ^2 _{L_2(\mathbb{R}^d)} +(4^{-1}+(4\mu)^{-1})d\alpha _1\Vert g\Vert ^2 _{L_\infty}\Vert u\Vert ^2 _{L_2(\mathbb{R}^d)}
+ \int _{\mathbb{R}^d}\vert \mathbf{v}_j\vert ^2\vert\mathbf{D}u^*\vert ^2\,d\mathbf{x}
\end{split}
\end{equation*}
for any $\mu >0$. By \eqref{Du^2vj^2 <=},
\begin{equation}
\label{J2 estimate}
\vert J_2\vert 
\leqslant \mu \Vert \mathbf{D}(\mathbf{v}_ju)\Vert ^2 _{L_2(\mathbb{R}^d)}+
\left( (4^{-1}+ (4\mu)^{-1})d\alpha _1\Vert g\Vert ^2_{L_\infty}
+ C_\Omega ^2 \Vert \mathbf{v}_j\Vert ^2_{L_d(\Omega)}\right)\Vert u\Vert ^2_{H^2(\mathbb{R}^d)}.
\end{equation}

Now, relations \eqref{J0=}, \eqref{J1 estimate}, and \eqref{J2 estimate} imply that
\begin{equation}
\label{1/2J0<=}
\frac{1}{2}J_0\leqslant \mu \Vert \mathbf{D}(\mathbf{v}_ju)\Vert ^2 _{L_2(\mathbb{R}^d)}
+\left((2\Vert g\Vert _{L_\infty}\alpha _1d+1)C_\Omega ^2 \Vert \mathbf{v}_j\Vert ^2_{L_d(\Omega)}
+\left(\frac{1}{4}+\frac{1}{4\mu}\right)d\alpha _1\Vert g\Vert ^2_{L_\infty}\right)\Vert u\Vert ^2_{H^2(\mathbb{R}^d)}.
\end{equation}

Comparing \eqref{v_ju tozd}, \eqref{8.13a new}, \eqref{J0=}, and \eqref{1/2J0<=}, we obtain
\begin{equation}
\label{g^1/2b(D)(v_ju)<=}
\begin{split}
\Vert & g^{1/2}b(\mathbf{D})(\mathbf{v}_ju)\Vert ^2 _{L_2(\mathbb{R}^d)}
\leqslant
2J_0+2\Vert g\Vert _{L_\infty}\alpha_1 d C_\Omega ^2 \Vert \mathbf{v}_j\Vert ^2_{L_d(\Omega)}\Vert u\Vert ^2 _{H^2(\mathbb{R}^d)}
\\
&\leqslant 4\mu \Vert \mathbf{D}(\mathbf{v}_ju)\Vert ^2 _{L_2(\mathbb{R}^d)}
\\
&+\left( (10\Vert g\Vert _{L_\infty}\alpha_1 d +4)C_\Omega ^2 \Vert \mathbf{v}_j\Vert ^2 _{L_d(\Omega)}+(1+\mu ^{-1})d\alpha _1\Vert g\Vert ^2_{L_\infty}\right)\Vert u\Vert ^2_{H^2(\mathbb{R}^d)}.
\end{split}
\end{equation}
By \eqref{<a<} (with $f=\mathbf{1}_n$),
\begin{equation*}
\begin{split}
4\mu \Vert \mathbf{D}(\mathbf{v}_ju)\Vert ^2 _{L_2(\mathbb{R}^d)}
&\leqslant 4\mu \alpha _0^{-1}\Vert g^{-1}\Vert _{L_\infty}\Vert g^{1/2}b(\mathbf{D})(\mathbf{v}_ju)\Vert ^2_{L_2(\mathbb{R}^d)}
\\
&=\frac{1}{2}\Vert g^{1/2}b(\mathbf{D})(\mathbf{v}_ju)\Vert ^2 _{L_2(\mathbb{R}^d)}\quad\mbox{for}\;\mu=\frac{1}{8}\alpha _0\Vert g^{-1}\Vert ^{-1}_{L_\infty}.
\end{split}
\end{equation*}
Together with \eqref{g^1/2b(D)(v_ju)<=} this implies
\begin{equation*}
\Vert g^{1/2}b(\mathbf{D})(\mathbf{v}_ju)\Vert ^2 _{L_2(\mathbb{R}^d)}\leqslant\mathcal{C}_j^2\Vert u\Vert ^2_{H^2(\mathbb{R}^d)},
\end{equation*}
where
\begin{equation*}
\mathcal{C}_j^2=(20\Vert g\Vert _{L_\infty}\alpha _1d+8)C_\Omega ^2 \Vert \mathbf{v}_j\Vert ^2_{L_d(\Omega)}
+(2+16\alpha _0^{-1}\Vert g^{-1}\Vert _{L_\infty})d\alpha _1\Vert g\Vert ^2_{L_\infty}.
\end{equation*}
Thus, 
\begin{equation*}
\Vert g^{1/2}b(\mathbf{D})[\mathbf{v}_j]\Vert _{H^2(\mathbb{R}^d)\rightarrow L_2(\mathbb{R}^d)}\leqslant \mathcal{C}_j,\quad j=1,\dots,m,
\end{equation*}
whence
\begin{equation*}
\Vert g^{1/2}b(\mathbf{D})[\Lambda ]\Vert _{H^2(\mathbb{R}^d)\rightarrow L_2(\mathbb{R}^d)}\leqslant\left(\sum _{j=1}^m \mathcal{C}_j^2\right)^{1/2}=:\mathfrak{C}_\Lambda;
\end{equation*}
i.~e., \eqref{g1/2b(D)Lambda H2 to L2 Lm} is true.
\end{proof}

\begin{proof}[Proof of Proposition~\textnormal{\ref{Proposition Ld implies multiplier}}.]
Let $u\in H^2(\mathbb{R}^d)$. Similarly to \eqref{v_ju tozd second summand est start}--\eqref{Du^2vj^2 <=},
\begin{equation*}
\Vert\mathbf{v}_ju\Vert ^2 _{L_2(\mathbb{R}^d)}\leqslant {C}_\Omega ^2 \Vert \mathbf{v}_j\Vert ^2_{L_d(\Omega)}\Vert  u\Vert ^2_{H^1(\mathbb{R}^d)}.
\end{equation*}
Here $\mathbf{v}_j(\mathbf{x})$, $j=1,\dots, m$, are the columns of the matrix $\Lambda (\mathbf{x})$. Thus,
\begin{equation}
\label{Proof of Proposition Ld implies multiplier 1}
\Vert [\Lambda ]u\Vert ^2_{L_2(\mathbb{R}^d)}\leqslant {C}_\Omega ^2\sum _{j=1}^m\Vert \mathbf{v}_j\Vert ^2_{L_d(\Omega)}\Vert u\Vert ^2_{H^1(\mathbb{R}^d)}.
\end{equation}

By \eqref{<a<} with $f=\mathbf{1}_n$, and Lemma~\ref{Proposition d>=3 Lambda},
\begin{equation}
\label{Proof of Proposition Ld implies multiplier 2}
\Vert \mathbf{D}[\Lambda ]u\Vert ^2_{L_2(\mathbb{R}^d)}\leqslant \alpha _0^{-1}\Vert g^{-1}\Vert _{L_\infty}\Vert g^{1/2}b(\mathbf{D})[\Lambda ] u\Vert ^2 _{L_2(\mathbb{R}^d)}
\leqslant \alpha _0^{-1}\Vert g^{-1}\Vert _{L_\infty}\mathfrak{C}_\Lambda ^2\Vert u\Vert ^2_{H^2(\mathbb{R}^d)}.
\end{equation}

Combining \eqref{Proof of Proposition Ld implies multiplier 1} and \eqref{Proof of Proposition Ld implies multiplier 2}, we obtain
\begin{equation*}
\Vert [\Lambda ]u\Vert ^2 _{H^1(\mathbb{R}^d)}\leqslant
\left( {C}_\Omega ^2\sum _{j=1}^m \Vert \mathbf{v}_j\Vert ^2_{L_d(\Omega)}
+\alpha _0^{-1}\Vert g^{-1}\Vert _{L_\infty}\mathfrak{C}_\Lambda ^2
\right)\Vert u\Vert ^2_{H^2(\mathbb{R}^d)},\quad u\in H^2(\mathbb{R}^d).
\end{equation*} 
\end{proof}

\begin{theorem}
\label{Theorem d>4 chapter 2}
Let $d\geqslant 5$. Under Condition~\textnormal{\ref{Condition Lambda multiplier}}, for $0<\varepsilon\leqslant 1$ and $\tau\in\mathbb{R}$ we have
\begin{equation}
\label{Th d>4 Lambda multiplier}
\begin{split}
\Bigl\Vert &
\widehat{\mathcal{A}}^{1/2}
\Bigl(
f \mathcal{A}^{-1/2}\sin (\varepsilon ^{-1}\tau\mathcal{A}^{1/2})f^{-1}\\
&-(I+\Lambda b(\mathbf{D}) )f_0(\mathcal{A}^0)^{-1/2}\sin (\varepsilon ^{-1}\tau (\mathcal{A}^0)^{1/2})f_0^{-1}
\Bigr)
\mathcal{R}(\varepsilon)
\Bigr\Vert _{L_2(\mathbb{R}^d)\rightarrow L_2(\mathbb{R}^d)}
\leqslant {C}_{15}\varepsilon (1+\vert \tau\vert ).
\end{split}
\end{equation}
The constant ${C}_{15}$ depends only on $m$, $\alpha _0$, $\alpha _1$, $\Vert g\Vert _{L_\infty}$, $\Vert g^{-1}\Vert _{L_\infty}$, $\Vert f\Vert _{L_\infty}$, $\Vert f^{-1}\Vert _{L_\infty}$, the parameters of the lattice $\Gamma$, and the norm $\Vert [\Lambda ]\Vert _{H^2(\mathbb{R}^d)\rightarrow H^1(\mathbb{R}^d)}$.
\end{theorem}

\begin{proof}
Under Condition~\ref{Condition Lambda multiplier}, by \eqref{<b^*b<}, \eqref{hat A}, and \eqref{proof Th 8/1 no PI 1}, we have
\begin{equation*}
\begin{split}
\Vert& \widehat{\mathcal{A}}^{1/2}[\Lambda]b(\mathbf{D})(I-\Pi)f_0(\mathcal{A}^0)^{-1/2}\sin (\varepsilon ^{-1}\tau(\mathcal{A}^0)^{1/2})f_0^{-1}\mathcal{R}(\varepsilon)\Vert _{L_2(\mathbb{R}^d)\rightarrow L_2 (\mathbb{R}^d)}
\\
&\leqslant
\Vert g\Vert ^{1/2}_{L_\infty}\alpha _1^{1/2}\Vert \mathbf{D}[\Lambda ]\Vert _{H^2(\mathbb{R}^d)\rightarrow L_2(\mathbb{R}^d)}
\alpha _1^{1/2}c_*^{-1/2}\Vert f\Vert _{L_\infty}\Vert f^{-1}\Vert _{L_\infty}(r_0^{-2}+1)\varepsilon ^2
\\
&\leqslant
C_{15}'\varepsilon ^2;\quad
C_{15}':=\alpha _1 c_*^{-1/2} \Vert g\Vert ^{1/2}_{L_\infty}\Vert f\Vert _{L_\infty}\Vert f^{-1}\Vert _{L_\infty}\Vert [\Lambda ]\Vert _{H^2(\mathbb{R}^d)\rightarrow H^1(\mathbb{R}^d)}(r_0^{-2}+1).
\end{split}
\end{equation*}
Combining this with \eqref{10.II}, we arrive at estimate \eqref{Th d>4 Lambda multiplier} with the constant $C_{15}:=C_{13}+C_{15}'$.

\end{proof}

For $d\geqslant 5$, removal of the operator $\Pi$ in the corrector also can be achieved by increasing the degree of the operator $\mathcal{R}(\varepsilon)$. In the application to homogenization of the hyperbolic Cauchy problem, this corresponds to more restrictive assumptions on the regularity of the initial data.

The proof of the following result is quite similar to that of Theorem~\ref{Theorem 8/1 NO PI}.

\begin{proposition}
\label{Proposition d>5 from H-kappa chapter 2}
Let $d\geqslant 5$. Then for $\tau \in\mathbb{R}$, $0<\varepsilon\leqslant 1$, we have
\begin{equation*}
\begin{split}
\Vert &
\widehat{\mathcal{A}}^{1/2}
\Bigl(
f \mathcal{A}^{-1/2}\sin (\varepsilon ^{-1}\tau\mathcal{A}^{1/2})f^{-1}\\
&-(I+\Lambda b(\mathbf{D}) )f_0(\mathcal{A}^0)^{-1/2}\sin (\varepsilon ^{-1}\tau (\mathcal{A}^0)^{1/2})f_0^{-1}
\Bigr)
\mathcal{R}(\varepsilon)^{d/4}
\Bigr\Vert _{L_2(\mathbb{R}^d)\rightarrow L_2(\mathbb{R}^d)}
\leqslant {C}_{16}\varepsilon (1+\vert \tau\vert ).
\end{split}
\end{equation*}
The constant $C_{16}$ depends only on  $m$, $n$, $d$, $\alpha _0$, $\alpha _1$, $\Vert g\Vert _{L_\infty}$, $\Vert g^{-1}\Vert _{L_\infty}$, $\Vert f\Vert _{L_\infty}$, $\Vert f^{-1}\Vert _{L_\infty}$, and the parameters of the lattice $\Gamma$.
\end{proposition}

\section*{Chapter III. Homogenization problem for hyperbolic systems} 
\label{Chapter 3}

\section{Approximation of the sandwiched operator ${\mathcal{A}}_\varepsilon ^{-1/2}\sin(\tau {\mathcal{A}}_\varepsilon^{1/2})$}

\label{Section main results in general case}

For a $\Gamma$-periodic measurable function $\psi(\mathbf{x})$ in $\mathbb{R}^d$ we denote $\psi^\varepsilon (\mathbf{x}):=\psi (\varepsilon ^{-1}\mathbf{x})$, $\varepsilon >0$. Let $[\psi^\varepsilon]$ be the operator of multiplication by the function  $\psi^\varepsilon (\mathbf{x})$. 
\textit{Our main object} is the operator  $\mathcal{A}_\varepsilon$, $\varepsilon >0$,  acting in $L_2(\mathbb{R}^d;\mathbb{C}^n)$ and formally given by the differential expression
\begin{equation}
\label{A_eps no hat}
\mathcal{A}_\varepsilon =f^\varepsilon (\mathbf{x})^*
b(\mathbf{D})^*g^\varepsilon (\mathbf{x})b(\mathbf{D})f^\varepsilon (\mathbf{x}).
\end{equation}
Denote
\begin{equation}
\label{A_eps}
\widehat{\mathcal{A}}_\varepsilon =b(\mathbf{D})^*g^\varepsilon (\mathbf{x})b(\mathbf{D}).
\end{equation}
The precise definitions of these operators are given in terms of the corresponding quadratic forms. The coefficients of the operators \eqref{A_eps no hat} and \eqref{A_eps}  oscillate rapidly as $\varepsilon \rightarrow 0$.

\textit{Our goal} is to approximate the sandwiched operator   $\mathcal{A}_\varepsilon ^{-1/2}\sin(\tau  \mathcal{A}_\varepsilon ^{1/2})
$.  The results are applied to homogenization of the solutions of the Cauchy problem for hyperbolic systems. 

\subsection{The principal term of approximation}

 Let $T_\varepsilon$ be the \textit{unitary scaling transformation} in $L_2(\mathbb{R}^d;\mathbb{C}^n)$: $(T_\varepsilon\mathbf{u})(\mathbf{x}):=\varepsilon ^{d/2}\mathbf{u}(\varepsilon \mathbf{x})$, $\varepsilon >0$. Then $\mathcal{A}_\varepsilon =\varepsilon ^{-2}T_\varepsilon ^* \mathcal{A}T_\varepsilon $. Thus,
\begin{equation*}
\mathcal{A}_\varepsilon ^{-1/2}\sin (\tau \mathcal{A}_\varepsilon ^{1/2})=\varepsilon T_\varepsilon ^* \mathcal{A}  ^{-1/2}\sin (\varepsilon ^{-1}\tau \mathcal{A}^{1/2})T_\varepsilon .
\end{equation*}
The operator $\mathcal{A}^0$ satisfies a similar identity.  Next,
\begin{equation*}
(\mathcal{H}_0+I)^{-1/2}=\varepsilon T_\varepsilon ^*(\mathcal{H}_0 +\varepsilon ^2I)^{-1/2}T_\varepsilon =T_\varepsilon ^*\mathcal{R}(\varepsilon)^{1/2}T_\varepsilon .
\end{equation*}
Note that for any $s$ 
the operator $(\mathcal{H}_0+I)^{s/2}$ is an isometric isomorphism of the Sobolev space $H^s(\mathbb{R}^d;\mathbb{C}^n)$ onto $L_2(\mathbb{R}^d;\mathbb{C}^n)$. Indeed, for $\mathbf{u}\in H^s(\mathbb{R}^d;\mathbb{C}^n)$ we have
\begin{equation}
\label{isometria}
\Vert (\mathcal{H}_0+I)^{s/2}\mathbf{u}\Vert ^2_{L_2(\mathbb{R}^d)}
=\int _{\mathbb{R}^d}(\vert \boldsymbol{\xi}\vert ^2 +1)^s\vert \widehat{\mathbf{u}}(\boldsymbol{\xi})\vert ^2\,d\boldsymbol{\xi}=\Vert \mathbf{u}\Vert ^2_{H^s (\mathbb{R}^d)}.
\end{equation}
Using these arguments, from  \eqref{10.I} we deduce the following result.

\begin{theorem}
\label{Theorem 12.1}
Let $\mathcal{A}_\varepsilon $ be the operator \eqref{A_eps no hat} and let $\mathcal{A}^0$ be the operator \eqref{A0 no hat}. Then for $\varepsilon >0$ and $\tau \in\mathbb{R}$ we have
\begin{equation}
\label{12.1}
\begin{split}
\Vert & 
f^\varepsilon\mathcal{A}_\varepsilon ^{-1/2}\sin (\tau \mathcal{A}_\varepsilon ^{1/2})(f^\varepsilon)^{-1}-f_0(\mathcal{A}^0)^{-1/2}\sin (\tau (\mathcal{A}^0)^{1/2})f_0^{-1}
\Vert _{H^1(\mathbb{R}^d)\rightarrow L_2(\mathbb{R}^d)}
\leqslant
{C}_{12}\varepsilon (1+\vert \tau\vert ).
\end{split}
\end{equation}
The constant ${C}_{12}$ is controlled in terms of $r_0$, $\alpha _0$, $\alpha _1$, $\Vert g\Vert _{L_\infty}$, $\Vert g^{-1}\Vert _{L_\infty}$, $\Vert f\Vert _{L_\infty}$, and $\Vert f^{-1}\Vert _{L_\infty}$.
\end{theorem}

By \eqref{f0<=} and the elementary inequality $\vert \sin x\vert /\vert x \vert \leqslant 1$, $x\in\mathbb{R}$,
\begin{equation}
\label{12.2}
\begin{split}
\Vert &
f^\varepsilon\mathcal{A}_\varepsilon ^{-1/2}\sin (\tau \mathcal{A}_\varepsilon ^{1/2})(f^\varepsilon)^{-1}-f_0(\mathcal{A}^0)^{-1/2}\sin (\tau (\mathcal{A}^0)^{1/2})f_0^{-1}
\Vert _{L_2(\mathbb{R}^d)\rightarrow L_2(\mathbb{R}^d)}
\\
&\leqslant 2\Vert f\Vert _{L_\infty}\Vert f^{-1}\Vert _{L_\infty}\vert \tau\vert
.
\end{split}
\end{equation}
Interpolating between  \eqref{12.2} and \eqref{12.1}, we obtain the following result.

\begin{theorem}
\label{Theorem 12.2}
Under the assumptions of Theorem~\textnormal{\ref{Theorem 12.1}}, for $0\leqslant s\leqslant 1$, $\tau\in\mathbb{R}$, and $\varepsilon >0$ we have
\begin{equation*}
\Vert f^\varepsilon \mathcal{A}_\varepsilon ^{-1/2}\sin (\tau \mathcal{A}_\varepsilon ^{1/2})(f^\varepsilon )^{-1}-f_0(\mathcal{A}^0)^{-1/2}\sin (\tau (\mathcal{A}^0)^{1/2})f_0^{-1}\Vert _{H^s(\mathbb{R}^d)\rightarrow L_2(\mathbb{R}^d)}
\leqslant \mathfrak{C}_1(s) (1+\vert \tau\vert )\varepsilon ^s,
\end{equation*}
where $\mathfrak{C}_1(s):=(2\Vert f\Vert _{L_\infty}\Vert f^{-1}\Vert _{L_\infty})^{1-s} {C}_{12}^s$. 
\end{theorem}

\subsection{Approximation with corrector}

Now, we obtain an approximation with the correction term taken into account. 
We put $\Pi _\varepsilon :=T_\varepsilon ^*\Pi T_\varepsilon$. Then $\Pi _\varepsilon $ is the pseudodifferential operator in  $L_2(\mathbb{R}^d;\mathbb{C}^n)$  with the symbol $\chi _{\widetilde{\Omega}/\varepsilon }(\boldsymbol{\xi})$, i.~e.,
\begin{equation}
\label{Pi eps}
(\Pi _\varepsilon \mathbf{u})(\mathbf{x})=(2\pi )^{-d/2}\int _{\widetilde{\Omega}/\varepsilon}e^{i\langle\mathbf{x},\boldsymbol{\xi}\rangle}\widehat{\mathbf{u}}(\boldsymbol{\xi})\,d\boldsymbol{\xi}.
\end{equation}
Obviously, $\Pi_\varepsilon \mathbf{D}^\sigma\mathbf{u}=\mathbf{D}^\sigma \Pi _\varepsilon \mathbf{u}$ for $\mathbf{u}\in H^\kappa(\mathbb{R}^d;\mathbb{C}^n)$ and any multiindex $\sigma$ of length $\vert\sigma\vert \leqslant \kappa$. Note that 
$
\Vert \Pi_\varepsilon\Vert _{H^\kappa(\mathbb{R}^d)\rightarrow H^\kappa(\mathbb{R}^d)}\leqslant 1, \quad\kappa\in\mathbb{Z}_+$.

The following results were obtained in \cite[Proposition 1.4]{PSu} and \cite[Subsec. 10.2]{BSu06}.

\begin{proposition}
\label{Proposition Pi eps -I}
For any function $\mathbf{u}\in H^1(\mathbb{R}^d;\mathbb{C}^n)$ we have
\begin{equation*}
\Vert \Pi _\varepsilon \mathbf{u}-\mathbf{u}\Vert _{L_2(\mathbb{R}^d)}\leqslant \varepsilon r_0^{-1}\Vert \mathbf{D}\mathbf{u}\Vert _{L_2(\mathbb{R}^d)},\quad\varepsilon >0.
\end{equation*}
\end{proposition}

\begin{proposition}
\label{Proposition Pi eps f eps}
Let $\Phi(\mathbf{x})$  be a $\Gamma$-periodic function in $\mathbb{R}^d$ such that $\Phi \in L_2(\Omega)$. Then the operator $[\Phi ^\varepsilon]\Pi_\varepsilon$ is bounded in $L_2(\mathbb{R}^d;\mathbb{C}^n)$, and
\begin{equation*}
\Vert [\Phi ^\varepsilon]\Pi_\varepsilon\Vert _{L_2(\mathbb{R}^d)\rightarrow L_2(\mathbb{R}^d)}\leqslant \vert \Omega\vert ^{-1/2}\Vert \Phi \Vert _{L_2(\Omega)},\quad\varepsilon >0.
\end{equation*}
\end{proposition}

\begin{theorem}
\label{Theorem 12.3}
Let $\Lambda (\mathbf{x})$ be the $\Gamma$-periodic solution of problem \eqref{Lambda problem}. Let $\Pi _\varepsilon$ be the operator \eqref{Pi eps}. Then, under the assumptions of Theorem~\textnormal{\ref{Theorem 12.1}}, for $\varepsilon >0$ and $\tau\in\mathbb{R}$ we have
\begin{equation}
\label{12.3}
\begin{split}
\bigl\Vert &
f^\varepsilon \mathcal{A}_\varepsilon ^{-1/2}\sin (\tau \mathcal{A}_\varepsilon ^{1/2})(f^\varepsilon )^{-1}
-(I+\varepsilon\Lambda ^\varepsilon b(\mathbf{D})\Pi _\varepsilon) f_0(\mathcal{A}^0)^{-1/2}\sin (\tau (\mathcal{A}^0)^{1/2})f_0^{-1}\bigr \Vert  _{H^2(\mathbb{R}^d)\rightarrow H^1(\mathbb{R}^d)}
\\
&\leqslant
{C}_{17}\varepsilon (1+\vert \tau\vert).
\end{split}
\end{equation}
The constant ${C}_{17}$ depends only on $m$, $\alpha _0$, $\alpha _1$, $\Vert g\Vert _{L_\infty}$, $\Vert g ^{-1}\Vert _{L_\infty}$, $\Vert f\Vert _{L_\infty}$, $\Vert f^{-1}\Vert _{L_\infty}$, and the parameters of the lattice $\Gamma$.
\end{theorem}

\begin{proof} 
By the scaling transformation, \eqref{10.II} implies that
\begin{equation}
\label{12.3a-2}
\begin{split}
\Bigl\Vert& \widehat{\mathcal{A}}_\varepsilon ^{1/2}
\Bigl(
f^\varepsilon \mathcal{A}_\varepsilon ^{-1/2}\sin (\tau\mathcal{A}_\varepsilon ^{1/2})(f^\varepsilon )^{-1}
-(I+\varepsilon \Lambda ^\varepsilon b(\mathbf{D})\Pi _\varepsilon )f_0 (\mathcal{A}^0)^{-1/2}\sin (\tau (\mathcal{A}^0)^{1/2})f_0^{-1} \Bigr)
\\
&\times
(\mathcal{H}_0 +I)^{-1}\Bigr\Vert _{L_2(\mathbb{R}^d)\rightarrow L_2(\mathbb{R}^d)}
\leqslant
{C}_{13}\varepsilon (1+\vert \tau\vert ).
\end{split}
\end{equation}
Note that, by \eqref{<b^*b<}, \eqref{g in}, and \eqref{A_eps},
\begin{equation}
\label{A_eps^1/2>=}
\widehat{c}_*\Vert \mathbf{D}\mathbf{u}\Vert ^2 _{L_2(\mathbb{R}^d)}\leqslant\Vert \widehat{\mathcal{A}}_\varepsilon ^{1/2}\mathbf{u}\Vert ^2 _{L_2(\mathbb{R}^d)},\quad\mathbf{u}\in H^1(\mathbb{R}^d;\mathbb{C}^n),
\end{equation}
where the constant $\widehat{c}_*$ is defined by \eqref{hat A(k)>=}. 
From \eqref{12.3a-2} and \eqref{A_eps^1/2>=} it follows that
\begin{equation}
\label{12.4}
\begin{split}
\Bigl\Vert &\mathbf{D}
\Bigl(
f^\varepsilon \mathcal{A}_\varepsilon ^{-1/2}\sin (\tau\mathcal{A}_\varepsilon ^{1/2})(f^\varepsilon )^{-1}
-(I+\varepsilon \Lambda ^\varepsilon b(\mathbf{D})\Pi _\varepsilon )f_0 (\mathcal{A}^0)^{-1/2}\sin (\tau (\mathcal{A}^0)^{1/2})f_0 ^{-1} \Bigr)
\\
&\times
(\mathcal{H}_0 +I)^{-1}\Bigr\Vert _{L_2(\mathbb{R}^d)\rightarrow L_2(\mathbb{R}^d)}
\leqslant
\widehat{c}_*^{-1/2}{C}_{13}\varepsilon (1+\vert \tau\vert ).
\end{split}
\end{equation}

Now, we estimate the $(L_2\rightarrow L_2)$-norm of the correction  term. Let $\Pi _\varepsilon ^{(m)}$ be the pseudodifferential operator in $L_2(\mathbb{R}^d;\mathbb{C}^m)$ with the symbol $\chi _{\widetilde{\Omega}/\varepsilon}(\boldsymbol{\xi})$. By Proposition \ref{Proposition Pi eps f eps} and \eqref{Lambda<=},
\begin{equation}
\label{proof fluxes 6}
\Vert \Lambda ^\varepsilon \Pi _\varepsilon ^{(m)}\Vert _{L_2(\mathbb{R}^d)\rightarrow L_2(\mathbb{R}^d)}\leqslant M_1.
\end{equation}
Using \eqref{g^0<=}, \eqref{f0<=}, \eqref{A0 no hat}, and \eqref{proof fluxes 6}, we have
\begin{equation}
\label{12.5}
\begin{split}
\Vert &\varepsilon \Lambda ^\varepsilon b(\mathbf{D})\Pi _\varepsilon  f_0(\mathcal{A}^0)^{-1/2}\sin (\tau (\mathcal{A}^0)^{1/2})f_0^{-1} (\mathcal{H}_0+I)^{-1}\Vert _{L_2(\mathbb{R}^d)\rightarrow L_2(\mathbb{R}^d)}
\\
&\leqslant
\varepsilon \Vert \Lambda ^\varepsilon \Pi _\varepsilon ^{(m)}\Vert  _{L_2(\mathbb{R}^d)\rightarrow L_2(\mathbb{R}^d)}
\Vert b(\mathbf{D})f_0(\mathcal{A}^0)^{-1/2}\sin (\tau (\mathcal{A}^0)^{1/2})f_0 ^{-1}(\mathcal{H}_0+I)^{-1}\Vert _{L_2(\mathbb{R}^d)\rightarrow L_2(\mathbb{R}^d)}
\\
&\leqslant\varepsilon M_1\Vert g^{-1}\Vert ^{1/2}_{L_\infty}
\Vert \sin (\tau (\mathcal{A}^0)^{1/2})f_0^{-1} (\mathcal{H}_0+I)^{-1}\Vert _{L_2(\mathbb{R}^d)\rightarrow L_2(\mathbb{R}^d)}
\leqslant
\varepsilon  M_1\Vert g^{-1}\Vert ^{1/2}_{L_\infty}\Vert f ^{-1}\Vert _{L_\infty}.
\end{split}
\end{equation}
Combining \eqref{isometria}, \eqref{12.1}, \eqref{12.4}, and \eqref{12.5}, we arrive at estimate \eqref{12.3} with the constant  
$
{C}_{17}:=\widehat{c}_*^{-1/2}{C}_{13}+
{C}_{12}+M_1\Vert g^{-1}\Vert ^{1/2}_{L_\infty}\Vert f ^{-1}\Vert _{L_\infty}$. 
\end{proof}

By interpolation, from Theorem~\ref{Theorem 12.3} we derive the following result.

\begin{theorem}
\label{Theorem 12.4}
Under the assumptions of Theorem~\textnormal{\ref{Theorem 12.3}}, for $0\leqslant s\leqslant 1$, $\tau\in\mathbb{R}$, and $0<\varepsilon\leqslant 1$ we have
\begin{equation}
\label{12.6}
\begin{split}
\Vert &f^\varepsilon \mathcal{A}_\varepsilon ^{-1/2}\sin (\tau \mathcal{A}_\varepsilon ^{1/2})(f^\varepsilon )^{-1}
-(I+\varepsilon\Lambda ^\varepsilon b(\mathbf{D})\Pi _\varepsilon )
f_0 (\mathcal{A}^0)^{-1/2}\sin (\tau (\mathcal{A}^0)^{1/2})f_0^{-1}\Vert _{H^{s+1}(\mathbb{R}^d)\rightarrow H^1(\mathbb{R}^d)}
\\
&\leqslant \mathfrak{C}_2(s) (1+\vert \tau\vert )\varepsilon ^s.
\end{split}
\end{equation}
Here the constant $\mathfrak{C}_2(s)$ depends only on $s$, $m$, $\alpha _0$, $\Vert g\Vert _{L_\infty}$, $\Vert g^{-1}\Vert _{L_\infty}$, $\Vert f\Vert _{L_\infty}$, $\Vert f^{-1}\Vert _{L_\infty}$, and the parameters of the lattice $\Gamma$. 
\end{theorem}

\begin{proof}
Let us  estimate the left-hand side of \eqref{12.6} for $s=0$. By \eqref{<b^*b<}, \eqref{A_eps no hat}, and the elementary inequality $\vert \sin x\vert /\vert x\vert \leqslant 1$, $x\in\mathbb{R}$,
\begin{equation}
\label{12.7}
\begin{split}
\Vert & f^\varepsilon \mathcal{A}_\varepsilon ^{-1/2}\sin (\tau\mathcal{A}_\varepsilon ^{1/2})(f^\varepsilon )^{-1}\Vert _{H^1(\mathbb{R}^d)\rightarrow H^1(\mathbb{R}^d)}
\\
&\leqslant
\Vert f\Vert _{L_\infty}\Vert f^{-1}\Vert _{L_\infty}\vert \tau\vert 
+\Vert \mathbf{D}f^\varepsilon \mathcal{A}_\varepsilon ^{-1/2}\sin (\tau \mathcal{A}_\varepsilon ^{1/2})(f^\varepsilon )^{-1}\Vert _{H^1(\mathbb{R}^d)\rightarrow L_2(\mathbb{R}^d)}
\\
&\leqslant
\Vert f\Vert _{L_\infty}\Vert f^{-1}\Vert _{L_\infty}\vert \tau\vert 
+\alpha _0^{-1/2}\Vert g^{-1}\Vert ^{1/2}_{L_\infty}\Vert f^{-1}\Vert _{L_\infty}.
\end{split}
\end{equation}
Similarly, by \eqref{<b^*b<}, \eqref{f0<=}, and \eqref{A0 no hat},
\begin{equation}
\label{12.8}
\Vert f_0(\mathcal{A}^0)^{-1/2}\sin (\tau (\mathcal{A}^0)^{1/2})f_0 ^{-1}\Vert _{H^1(\mathbb{R}^d)\rightarrow H^1(\mathbb{R}^d)}
\leqslant
\Vert f\Vert _{L_\infty}\Vert f^{-1}\Vert _{L_\infty}\vert\tau\vert 
+\alpha _0^{-1/2}\Vert g^{-1}\Vert ^{1/2}_{L_\infty}\Vert f^{-1} \Vert _{L_\infty}.
\end{equation}

From \eqref{g^0<=}, \eqref{f0<=}, and \eqref{proof fluxes 6} it follows that
\begin{equation}
\label{proof interpol corr 3}
\begin{split}
\Vert &\varepsilon \Lambda ^\varepsilon b(\mathbf{D})\Pi _\varepsilon  f_0({\mathcal{A}}^0) ^{-1/2}\sin (\tau ( {\mathcal{A}}^0) ^{1/2})f_0^{-1}\Vert _{H^1(\mathbb{R}^d)\rightarrow H^1(\mathbb{R}^d)}\\
&\leqslant \varepsilon M_1\Vert g^{-1}\Vert ^{1/2}_{L_\infty}\Vert f^{-1}\Vert _{L_\infty}
+\varepsilon\Vert \mathbf{D}\Lambda ^\varepsilon b(\mathbf{D})\Pi _\varepsilon f_0 ( {\mathcal{A}}^0) ^{-1/2}\sin (\tau ( {\mathcal{A}}^0) ^{1/2})f_0^{-1}\Vert _{H^1(\mathbb{R}^d)\rightarrow L_2(\mathbb{R}^d)}
\\
&\leqslant \varepsilon M_1\Vert g^{-1}\Vert ^{1/2}_{L_\infty}\Vert f^{-1}\Vert _{L_\infty}
+\Vert (\mathbf{D}\Lambda)^\varepsilon b(\mathbf{D})\Pi _\varepsilon f_0 ( {\mathcal{A}}^0) ^{-1/2}\sin (\tau ( {\mathcal{A}}^0) ^{1/2})f_0^{-1}\Vert _{H^1(\mathbb{R}^d)\rightarrow L_2(\mathbb{R}^d)}
\\
&+\varepsilon\Vert \Lambda ^\varepsilon \mathbf{D}b(\mathbf{D})\Pi _\varepsilon f_0 ( {\mathcal{A}}^0) ^{-1/2}\sin (\tau ( {\mathcal{A}}^0) ^{1/2})f_0^{-1}\Vert _{H^1(\mathbb{R}^d)\rightarrow L_2(\mathbb{R}^d)}.
\end{split}
\end{equation}
By Proposition \ref{Proposition Pi eps f eps}, \eqref{D Lambda <=}, \eqref{g^0<=}, \eqref{f0<=}, and \eqref{A0 no hat}, 
\begin{equation}
\label{proof interpol corr 8.18}
\Vert (\mathbf{D}\Lambda)^\varepsilon b(\mathbf{D})\Pi _\varepsilon f_0( {\mathcal{A}}^0) ^{-1/2}\sin (\tau ( {\mathcal{A}}^0) ^{1/2})f_0^{-1}\Vert _{H^1(\mathbb{R}^d)\rightarrow L_2(\mathbb{R}^d)}
\leqslant M_2\Vert g^{-1}\Vert ^{1/2}_{L_\infty}\Vert f^{-1}\Vert _{L_\infty}.
\end{equation}
Next, according to \eqref{g^0<=},  \eqref{A0 no hat}, and \eqref{proof fluxes 6},
\begin{equation}
\label{proof interpol corr e1}
\begin{split}
\varepsilon\Vert& \Lambda ^\varepsilon \mathbf{D}b(\mathbf{D})\Pi _\varepsilon f_0( {\mathcal{A}}^0) ^{-1/2}\sin (\tau ( {\mathcal{A}}^0) ^{1/2})f_0^{-1}\Vert _{H^1(\mathbb{R}^d)\rightarrow L_2(\mathbb{R}^d)}
\\
&\leqslant \varepsilon M_1\Vert g^{-1}\Vert _{L_\infty}^{1/2}\Vert \mathbf{D}\sin (\tau ( {\mathcal{A}}^0) ^{1/2})f_0^{-1}\Vert _{H^1(\mathbb{R}^d)\rightarrow L_2(\mathbb{R}^d)}.
\end{split}
\end{equation}
Since the operator ${\mathcal{A}}^0$ with constant coefficients commutes with the differentiation $\mathbf{D}$, we have
\begin{equation*}
\Vert \mathbf{D}\sin (\tau ( {\mathcal{A}}^0) ^{1/2})\Vert _{H^1(\mathbb{R}^d)\rightarrow L_2(\mathbb{R}^d)}\leqslant 1.
\end{equation*}
Together with \eqref{f0<=} and \eqref{proof interpol corr 3}--\eqref{proof interpol corr e1} this yields
\begin{equation}
\label{12.9}
\Vert\varepsilon \Lambda ^\varepsilon b(\mathbf{D})\Pi _\varepsilon f_0 (\mathcal{A}^0)^{-1/2}\sin (\tau (\mathcal{A}^0)^{1/2})f_0 ^{-1}\Vert _{H^1(\mathbb{R}^d)\rightarrow H^1(\mathbb{R}^d)}
\leqslant
(2\varepsilon M_1+M_2)\Vert g^{-1}\Vert ^{1/2}_{L_\infty}\Vert f ^{-1}\Vert _{L_\infty}.
\end{equation}
Relations \eqref{12.7}, \eqref{12.8}, and \eqref{12.9} imply that
\begin{equation}
\label{12.10}
\begin{split}
\Vert &f^\varepsilon \mathcal{A}_\varepsilon ^{-1/2}\sin (\tau \mathcal{A}_\varepsilon ^{1/2})(f^\varepsilon)^{-1}
-(I+\varepsilon \Lambda ^\varepsilon b(\mathbf{D})\Pi _\varepsilon)
f_0(\mathcal{A}^0)^{-1/2}\sin (\tau (\mathcal{A}^0)^{1/2})f_0 ^{-1}\Vert _{H^1(\mathbb{R}^d)\rightarrow H^1(\mathbb{R}^d)}
\\
&\leqslant { {C}}_{18}(1+\vert \tau\vert ),\quad\tau\in\mathbb{R},\quad 0<\varepsilon\leqslant 1,
\end{split}
\end{equation}
where ${ {C}}_{18}:=\max\lbrace2\Vert f\Vert _{L_\infty}\Vert f^{-1}\Vert _{L_\infty} ;(2\alpha _0 ^{-1/2}
+2M_1+M_2)\Vert g^{-1}\Vert ^{1/2}_{L_\infty}\Vert f ^{-1}\Vert _{L_\infty}\rbrace$. 
Interpolating between \eqref{12.10}  and \eqref{12.3}, we deduce estimate \eqref{12.6} with $\mathfrak{C}_2(s):={ {C}}_{18}^{\,1-s}{ {C}}_{17}^s$.
\end{proof}

\subsection{The case where $d\leqslant 4$}

Now we apply Theorem~\ref{Theorem 8/1 NO PI}. By the scaling transformation, \eqref{Th 8/1 NO PI} implies that
\begin{equation}
\label{9.20a NEW}
\begin{split}
\Bigl\Vert &\widehat{\mathcal{A}}_\varepsilon ^{1/2}\left(
f^\varepsilon\mathcal{A}_\varepsilon ^{-1/2}\sin (\tau \mathcal{A}_\varepsilon ^{1/2})(f^\varepsilon)^{-1}
-(I+\varepsilon \Lambda ^\varepsilon b(\mathbf{D}))f_0(\mathcal{A}^0)^{-1/2}\sin (\tau (\mathcal{A}^0)^{1/2})f_0^{-1}
\right)
\\
&\times
(\mathcal{H}_0+I)^{-1}
\Bigr\Vert _{L_2(\mathbb{R}^d)\rightarrow L_2(\mathbb{R}^d)}
\leqslant C_{14}\varepsilon (1+\vert \tau\vert),\quad 0<\varepsilon\leqslant 1,\quad \tau\in\mathbb{R}.
\end{split}
\end{equation}
Combining this with \eqref{A_eps^1/2>=}, we obtain
\begin{equation}
\label{1 p.9.3}
\begin{split}
\Bigl\Vert&\mathbf{D}\left(
f^\varepsilon\mathcal{A}_\varepsilon ^{-1/2}\sin (\tau \mathcal{A}_\varepsilon ^{1/2})(f^\varepsilon)^{-1}
-(I+\varepsilon \Lambda ^\varepsilon b(\mathbf{D}))f_0(\mathcal{A}^0)^{-1/2}\sin (\tau (\mathcal{A}^0)^{1/2})f_0^{-1}
\right)
\\
&\times
(\mathcal{H}_0+I)^{-1}
\Bigr\Vert _{L_2(\mathbb{R}^d)\rightarrow L_2(\mathbb{R}^d)}
\leqslant \widehat{c}_*^{-1/2} C_{14}\varepsilon (1+\vert \tau\vert),\quad 0<\varepsilon\leqslant 1,\quad\tau\in\mathbb{R}.
\end{split}
\end{equation}

Let us estimate the $(L_2\rightarrow L_2)$-norm of the corrector. By the scaling transformation,
\begin{equation}
\label{L2L2 corr norm}
\begin{split}
\varepsilon \Vert \Lambda ^\varepsilon b(\mathbf{D})f_0 (\mathcal{A}^0)^{-1/2}\sin (\tau (\mathcal{A}^0)^{1/2})f_0^{-1}(\mathcal{H}_0 +I)^{-1}\Vert _{L_2(\mathbb{R}^d)\rightarrow L_2(\mathbb{R}^d)}
\\
=
\varepsilon \Vert \Lambda b(\mathbf{D})f_0(\mathcal{A}^0)^{-1/2}\sin (\varepsilon ^{-1}\tau (\mathcal{A}^0)^{1/2})f_0^{-1}\mathcal{R}(\varepsilon)\Vert _{L_2(\mathbb{R}^d)\rightarrow L_2(\mathbb{R}^d)}.
\end{split}
\end{equation}

The $(H^s\rightarrow L_2)$-norm of the operator $[\Lambda ]$ was estimated in \cite[Proposition 11.3]{Su_MMNP}.

\begin{proposition}
\label{Proposition Lambda Hs to L2}
Let $s=0$ for $d=1$, $s>0$ for $d=2$, $s=d/2-1$ for $d\geqslant 3$. Then the operator $[\Lambda]$ is a continuous mapping of $H^s(\mathbb{R}^d;\mathbb{C}^m)$ to $L_2(\mathbb{R}^d;\mathbb{C}^m)$, and 
\begin{equation*}
\Vert [\Lambda ]\Vert _{H^s(\mathbb{R}^d)\rightarrow L_2(\mathbb{R}^d)}\leqslant \mathfrak{C}_d,
\end{equation*}
there the constant $\mathfrak{C}_d$ depends only on $d$, $m$, $n$, $\alpha _0$, $\Vert g\Vert _{L_\infty}$, $\Vert g^{-1}\Vert _{L_\infty}$, and the parameters of the lattice $\Gamma$\textnormal{;} in the case $d=2$ it depends also on $s$.
\end{proposition}

Now we consider only the case $d\leqslant 4$. So, by Proposition~\ref{Proposition Lambda Hs to L2},
\begin{equation}
\label{Lambda H1->L2}
\Vert [\Lambda]\Vert _{H^1(\mathbb{R}^d)\rightarrow L_2(\mathbb{R}^d)}\leqslant\mathfrak{C}_d,\quad d\leqslant 4.
\end{equation}
Thus, we need to estimate the operator $b(\mathbf{D})f_0(\mathcal{A}^0)^{-1/2}\sin (\varepsilon ^{-1}\tau (\mathcal{A}^0)^{1/2})f_0^{-1}\mathcal{R}(\varepsilon)$ in the\break $(L_2\rightarrow H^1)$-norm. By \eqref{g^0<=}, \eqref{f0<=}, and \eqref{A0 no hat}, for any $d$ we have
\begin{equation}
\label{4 in p.9.3}
\begin{split}
\Vert &
b(\mathbf{D})f_0(\mathcal{A}^0)^{-1/2}\sin (\varepsilon ^{-1}\tau (\mathcal{A}^0)^{1/2})f_0^{-1}\mathcal{R}(\varepsilon)\Vert _{L_2(\mathbb{R}^d)\rightarrow H^1(\mathbb{R}^d)}
\\
&\leqslant
\Vert b(\mathbf{D})f_0(\mathcal{A}^0)^{-1/2}\sin (\varepsilon ^{-1}\tau (\mathcal{A}^0)^{1/2})f_0^{-1}\mathcal{R}(\varepsilon)\Vert _{L_2(\mathbb{R}^d)\rightarrow L_2(\mathbb{R}^d)}
\\
&+\Vert \mathbf{D}b(\mathbf{D})f_0(\mathcal{A}^0)^{-1/2}\sin (\varepsilon ^{-1}\tau (\mathcal{A}^0)^{1/2})f_0^{-1}\mathcal{R}(\varepsilon)\Vert _{L_2(\mathbb{R}^d)\rightarrow L_2(\mathbb{R}^d)}
\\
&\leqslant 2\Vert g^{-1}\Vert ^{1/2}_{L_\infty}\Vert f^{-1}\Vert _{L_\infty},\quad \tau\in\mathbb{R},\quad 0<\varepsilon\leqslant 1.
\end{split}
\end{equation}

The following result is a direct consequence of \eqref{12.1} and \eqref{1 p.9.3}--\eqref{4 in p.9.3}.

\begin{theorem}
\label{Theorem d<=4 chapter 3}
Let $d\leqslant 4$. Under the assumptions of Theorem~\textnormal{\ref{Theorem 12.3}}, we have
\begin{equation}
\label{Th d<=4}
\begin{split}
\Vert &f^\varepsilon \mathcal{A}_\varepsilon ^{-1/2}\sin (\tau \mathcal{A}_\varepsilon ^{1/2})(f^\varepsilon)^{-1}
-(I+\varepsilon \Lambda ^\varepsilon b(\mathbf{D}))f_0(\mathcal{A}^0)^{-1/2}\sin (\tau (\mathcal{A}^0)^{1/2})f_0^{-1}\Vert _{H^2(\mathbb{R}^d)\rightarrow H^1(\mathbb{R}^d)}
\\
&\leqslant C_{19}\varepsilon (1+\vert \tau\vert),\quad 0<\varepsilon\leqslant 1,\quad \tau\in\mathbb{R}.
\end{split}
\end{equation}
The constant $C_{19}:=\widehat{c}_*^{-1/2}C_{14}+C_{12}+2\mathfrak{C}_d\Vert g^{-1}\Vert ^{1/2}_{L_\infty}\Vert f^{-1}\Vert _{L_\infty}$ depends only on $d$, $m$, $n$, $\alpha _0$, $\alpha _1$, $\Vert g\Vert _{L_\infty}$, $\Vert g^{-1}\Vert _{L_\infty}$, $\Vert f\Vert _{L_\infty}$, $\Vert f^{-1}\Vert _{L_\infty}$, and the parameters of the lattice $\Gamma$.
\end{theorem}

\subsection{Removal of $\Pi _\varepsilon$ from the corrector for $d\geqslant 5$}

The following result can be deduced from Theorem~\ref{Theorem d>4 chapter 2}.

\begin{theorem}
\label{Theorem d>=5 eps}
Let $d\geqslant 5$. Let Condition~\textnormal{\ref{Condition Lambda multiplier}} be satisfied. Then, under the assumptions of Theorem~\textnormal{\ref{Theorem 12.3}}, for $0<\varepsilon\leqslant 1$ and $\tau\in\mathbb{R}$ we have
\begin{equation}
\label{Th d>=5, chapter 3}
\begin{split}
\Vert &f^\varepsilon \mathcal{A}_\varepsilon ^{-1/2}\sin (\tau \mathcal{A}_\varepsilon ^{1/2})(f^\varepsilon)^{-1}
-(I+\varepsilon \Lambda ^\varepsilon b(\mathbf{D}))f_0(\mathcal{A}^0)^{-1/2}\sin (\tau (\mathcal{A}^0)^{1/2})f_0^{-1}\Vert _{H^2(\mathbb{R}^d)\rightarrow H^1(\mathbb{R}^d)}
\\
&\leqslant C_{20}\varepsilon (1+\vert \tau\vert).
\end{split}
\end{equation}
The constant $C_{20}$ depends only on  $m$, $\alpha _0$, $\alpha _1$, $\Vert g\Vert _{L_\infty}$, $\Vert g^{-1}\Vert _{L_\infty}$, $\Vert f\Vert _{L_\infty}$, $\Vert f^{-1}\Vert _{L_\infty}$,  the parameters of the lattice $\Gamma$, and the norm $\Vert [\Lambda ]\Vert _{H^2(\mathbb{R}^d)\rightarrow H^1(\mathbb{R}^d)}$.
\end{theorem}

\begin{proof}
The proof is similar to that of Theorem~\ref{Theorem 12.3}. Combining \eqref{g^0<=}, \eqref{f0<=}, \eqref{A0 no hat}, \eqref{Th d>4 Lambda multiplier}, \eqref{12.1}, and \eqref{A_eps^1/2>=}, we arrive at the estimate \eqref{Th d>=5, chapter 3} with $C_{20}:=\widehat{c}_*^{-1/2}C_{15}+C_{12}+\Vert [\Lambda ]\Vert _{H^2(\mathbb{R}^d)\rightarrow H^1(\mathbb{R}^d)}\Vert g^{-1}\Vert ^{1/2}_{L_\infty}\Vert f^{-1}\Vert _{L_\infty}$.
\end{proof}

\begin{theorem}
\label{Theorem D>=5 smooth date sec 9}
Let $d\geqslant 5$. Under the assumptions of Theorem~\textnormal{\ref{Theorem 12.3}}, for $0<\varepsilon\leqslant 1$ and $\tau\in\mathbb{R}$ we have
\begin{equation}
\label{Th d>5 smooth data ch 3}
\begin{split}
\Vert &f^\varepsilon \mathcal{A}_\varepsilon ^{-1/2}\sin (\tau \mathcal{A}_\varepsilon ^{1/2})(f^\varepsilon)^{-1}
-(I+\varepsilon \Lambda ^\varepsilon b(\mathbf{D}))f_0(\mathcal{A}^0)^{-1/2}\sin (\tau (\mathcal{A}^0)^{1/2})f_0^{-1}\Vert _{H^{d/2}(\mathbb{R}^d)\rightarrow H^1(\mathbb{R}^d)}
\\
&\leqslant C_{21}\varepsilon (1+\vert \tau\vert).
\end{split}
\end{equation}
The constant $C_{21}$ depends only on $d$, $m$, $n$, $\alpha _0$, $\alpha _1$, $\Vert g\Vert _{L_\infty}$, $\Vert g^{-1}\Vert _{L_\infty}$, $\Vert f\Vert _{L_\infty}$, $\Vert f^{-1}\Vert _{L_\infty}$, and the parameters of the lattice $\Gamma$.
\end{theorem}

\begin{proof}
By the scaling transformation, from Proposition~\ref{Proposition d>5 from H-kappa chapter 2} it follows that
\begin{equation*}
\begin{split}
\Bigl\Vert &\widehat{\mathcal{A}}_\varepsilon ^{1/2}\Bigl(f^\varepsilon \mathcal{A}_\varepsilon ^{-1/2}\sin(\tau \mathcal{A}_\varepsilon ^{1/2})(f^\varepsilon)^{-1}
-(I+\varepsilon \Lambda ^\varepsilon b(\mathbf{D}))f_0 (\mathcal{A}^0)^{-1/2}\sin(\tau (\mathcal{A}^0)^{1/2})f_0^{-1}\Bigr)
\\
&\times
(\mathcal{H}_0+I)^{-d/4}\Bigr\Vert _{L_2(\mathbb{R}^d)\rightarrow L_2(\mathbb{R}^d)}
\leqslant C_{16}\varepsilon (1+\vert \tau\vert),
\quad 0<\varepsilon \leqslant 1,\quad\tau\in\mathbb{R}.
\end{split}
\end{equation*}
By \eqref{A_eps^1/2>=},
\begin{equation}
\label{D... d>=5 proof chapter 3}
\begin{split}
\Bigl\Vert&\mathbf{D}\Bigl(f^\varepsilon \mathcal{A}_\varepsilon ^{-1/2}\sin(\tau \mathcal{A}_\varepsilon ^{1/2})(f^\varepsilon)^{-1}
-(I+\varepsilon \Lambda ^\varepsilon b(\mathbf{D}))f_0 (\mathcal{A}^0)^{-1/2}\sin(\tau (\mathcal{A}^0)^{1/2})f_0^{-1}\Bigr)
\\
&\times(\mathcal{H}_0+I)^{-d/4}\Bigr\Vert _{L_2(\mathbb{R}^d)\rightarrow L_2(\mathbb{R}^d)}
\leqslant \widehat{c}_*^{-1/2} C_{16}\varepsilon (1+\vert \tau\vert),
\quad 0<\varepsilon \leqslant 1,\quad\tau\in\mathbb{R}.
\end{split}
\end{equation}
By Proposition~\ref{Proposition Lambda Hs to L2}, and \eqref{g^0<=}, \eqref{f0<=}, \eqref{A0 no hat},
\begin{equation}
\label{d>5 final proof ch 3}
\begin{split}
\Vert &
\Lambda ^\varepsilon b(\mathbf{D})f_0 (\mathcal{A}^0)^{-1/2}\sin(\tau (\mathcal{A}^0)^{1/2})f_0^{-1}(\mathcal{H}_0+I)^{-d/4}\Vert _{L_2(\mathbb{R}^d)\rightarrow L_2(\mathbb{R}^d)}
\\
&\leqslant\mathfrak{C}_d\Vert b(\mathbf{D})f_0 (\mathcal{A}^0)^{-1/2}\sin(\tau (\mathcal{A}^0)^{1/2})f_0^{-1}(\mathcal{H}_0+I)^{-d/4}\Vert _{L_2(\mathbb{R}^d)\rightarrow H^{d/2-1}(\mathbb{R}^d)}
\\
&\leqslant\mathfrak{C}_d\Vert g^{-1}\Vert ^{1/2}_{L_\infty}\Vert f^{-1}\Vert _{L_\infty}\Vert (\mathcal{H}_0+I)^{-d/4}\Vert _{L_2(\mathbb{R}^d)\rightarrow H^{d/2-1}(\mathbb{R}^d)}
\leqslant
\mathfrak{C}_d\Vert g^{-1}\Vert ^{1/2}_{L_\infty}\Vert f^{-1}\Vert _{L_\infty}.
\end{split}
\end{equation}
Combining \eqref{12.1}, \eqref{D... d>=5 proof chapter 3}, and \eqref{d>5 final proof ch 3}, we arrive at estimate~\eqref{Th d>5 smooth data ch 3} with the constant $C_{21}:=\widehat{c}_*^{-1/2}C_{16}+C_{12}+\mathfrak{C}_d\Vert g^{-1}\Vert ^{1/2}_{L_\infty}\Vert f^{-1}\Vert _{L_\infty}$. 
\end{proof}

\subsection{Removal of $\Pi _\varepsilon$. Interpolational results}

To obtain the analogue of Theorem~\ref{Theorem 12.4} with $\Pi_\varepsilon$ replaced by $I$ 
we need the continuity of the operator $\varepsilon[\Lambda^\varepsilon]b(\mathbf{D})f_0(\mathcal{A}^0)^{-1/2}\sin(\tau (\mathcal{A}^0)^{1/2})f_0^{-1}$ in $H^1(\mathbb{R}^d;\mathbb{C}^n)$, i.~e., we need the boundedness of the norms $\Vert [(\mathbf{D}\Lambda)^\varepsilon ]\Vert _{H^1(\mathbb{R}^d)\rightarrow L_2(\mathbb{R}^d)}$ and $\Vert[\Lambda ^\varepsilon ]\Vert _{L_2(\mathbb{R}^d)\rightarrow L_2(\mathbb{R}^d)}$. Due to Parseval's theorem, the assumption $\Vert[\Lambda ^\varepsilon ]\Vert _{L_2(\mathbb{R}^d)\rightarrow L_2(\mathbb{R}^d)}<\infty$ holds if and only if the matrix-valued function $\Lambda$ is subject to the following condition.

\begin{condition}
\label{Condition Lambda in L infty}
Assume that the $\Gamma$-periodic solution $\Lambda (\mathbf{x})$ of problem \eqref{Lambda problem} is bounded, i.~e., $\Lambda\in L_\infty (\mathbb{R}^d)$.
\end{condition}

Under Condition~\ref{Condition Lambda in L infty}, the operator $[(\mathbf{D}\Lambda)^\varepsilon ]$ is bounded from $H^1$ to $L_2$ due to the following result obtained in \cite[Corollary 2.4]{PSu}.

\begin{lemma}
\label{Lemma PSu Lambda in L-infty}
Under Condition \textnormal{\ref{Condition Lambda in L infty}}, for any function $u\in H^1(\mathbb{R}^d)$ and $\varepsilon >0$ we have
\begin{equation*}
\int _{\mathbb{R}^d}\vert (\mathbf{D}\Lambda) ^\varepsilon(\mathbf{x})\vert ^2\vert u(\mathbf{x})\vert ^2\,d\mathbf{x}
\leqslant\mathfrak{c}_1\Vert u\Vert ^2_{L_2(\mathbb{R}^d)}+\mathfrak{c}_2\varepsilon ^2\Vert \Lambda\Vert ^2_{L_\infty}\Vert \mathbf{D}u\Vert ^2_{L_2(\mathbb{R}^d)}.
\end{equation*}
The constants $\mathfrak{c}_1$ and $\mathfrak{c}_2$ depend on $m$, $d$, $\alpha _0$, $\alpha _1$, $\Vert g\Vert _{L_\infty}$, and $\Vert g^{-1}\Vert _{L_\infty}$.
\end{lemma}

Some cases where Condition~\ref{Condition Lambda in L infty} is fulfilled automatically were distinguished in \cite[Lemma~8.7]{BSu06}.

\begin{proposition}
\label{Proposition Lambda in L infty <=}
Suppose that at least one of the following assumptions is satisfied\textnormal{:}

\noindent
$1^\circ )$ $d\leqslant 2${\rm ;}

\noindent
$2^\circ )$ the dimension $d\geqslant 1$ is arbitrary and the operator $\mathcal{A}_\varepsilon$ has the form $\mathcal{A}_\varepsilon =\mathbf{D}^* g^\varepsilon (\mathbf{x})\mathbf{D}$, where $g(\mathbf{x})$ is symmetric matrix with real entries{\rm ;}

\noindent
$3^\circ )$ the dimension  $d$ is arbitrary and $g^0=\underline{g}$, i.~e., relations \eqref{underline-g} are true.

Then Condition~\textnormal{\ref{Condition Lambda in L infty}}  is fulfilled.
\end{proposition}

Surely, if $\Lambda \in L_\infty$, then Condition~\ref{Condition Lambda in Ld} holds automatically. Then, by Proposition~\ref{Proposition Ld implies multiplier}, for $d\geqslant 5$, the assumptions of Theorem~\ref{Theorem d>=5 eps} are satisfied.

We are going to check that under Condition \ref{Condition Lambda in L infty} the analog of Theorem \ref{Theorem 12.4} is valid without any smoothing operator in the corrector. To do this, we estimate the $(H^1\rightarrow H^1)$-norm of the operators under the norm sign in  \eqref{Th d<=4} (or \eqref{Th d>=5, chapter 3}). 
By \eqref{g^0<=}, \eqref{f0<=}, \eqref{A0 no hat}, and Lemma~\ref{Lemma PSu Lambda in L-infty}, we obtain
\begin{equation}
\label{12.15}
\begin{split}
\Vert &\varepsilon \Lambda ^\varepsilon b(\mathbf{D})f_0(\mathcal{A}^0)^{-1/2}\sin(\tau (\mathcal{A}^0)^{1/2})f_0 ^{-1} \Vert _{H^1(\mathbb{R}^d)\rightarrow H^1(\mathbb{R}^d)}
\leqslant
2\varepsilon\Vert \Lambda\Vert _{L_\infty}\Vert g^{-1}\Vert ^{1/2}_{L_\infty}\Vert f ^{-1}\Vert _{L_\infty}\\
&+\Vert g^{-1}\Vert ^{1/2}_{L_\infty}\Vert f^{-1} \Vert _{L_\infty}(\mathfrak{c}_1+\mathfrak{c}_2\Vert \Lambda \Vert ^2_{L_\infty})^{1/2},
\quad 
0<\varepsilon\leqslant 1,\quad \tau\in\mathbb{R}.
\end{split}
\end{equation}
Combining \eqref{12.7}, \eqref{12.8}, and \eqref{12.15}, we deduce that
\begin{equation}
\label{12.16}
\begin{split}
\Vert &f^\varepsilon \mathcal{A}_\varepsilon ^{-1/2}\sin (\tau \mathcal{A}_\varepsilon ^{1/2})(f^\varepsilon)^{-1}
-(I+\varepsilon \Lambda ^\varepsilon b(\mathbf{D}))
f_0 (\mathcal{A}^0)^{-1/2}\sin (\tau (\mathcal{A}^0)^{1/2})f_0 ^{-1}\Vert _{H^1(\mathbb{R}^d)\rightarrow H^1(\mathbb{R}^d)}
\\
&\leqslant
{ {C}}_{22}(1+\vert \tau\vert),
\quad 
0<\varepsilon\leqslant 1,\quad\tau\in\mathbb{R},
\end{split}
\end{equation}
where 
\begin{equation*}
{ {C}}_{22}:=\Vert f^{-1}\Vert _{L_\infty}
\max\left\lbrace 2\Vert f\Vert _{L_\infty};
\Vert g^{-1}\Vert ^{1/2}_{L_\infty}\left(2\alpha _0^{-1/2}+2\Vert \Lambda\Vert _{L_\infty}
+(\mathfrak{c}_1+\mathfrak{c}_2\Vert \Lambda\Vert ^2_{L_\infty})^{1/2}\right)\right\rbrace .
\end{equation*}
Interpolating between \eqref{12.16} and \eqref{Th d<=4} for $d\leqslant 4$  and between \eqref{12.16} and \eqref{Th d>=5, chapter 3} for $d\geqslant 5$, we arrive at the following result.

\begin{theorem}
\label{Theorem 12.6}
Suppose that the assumptions of Theorem~\textnormal{\ref{Theorem 12.1}} are satisfied and Condition~\textnormal{\ref{Condition Lambda in L infty}} holds. 
Then for $0\leqslant s\leqslant 1$ and $\tau\in\mathbb{R}$, $0<\varepsilon\leqslant 1$ we have
\begin{equation*}
\begin{split}
\Vert & f^\varepsilon \mathcal{A}_\varepsilon ^{-1/2}\sin (\tau \mathcal{A}_\varepsilon ^{1/2})(f^\varepsilon)^{-1}
-(I+\varepsilon \Lambda ^\varepsilon b(\mathbf{D}))
f_0 (\mathcal{A}^0)^{-1/2}\sin (\tau (\mathcal{A}^0)^{1/2})f_0^{-1}\Vert _{H^{s+1}(\mathbb{R}^d)\rightarrow H^1(\mathbb{R}^d)}
\\
&\leqslant
\mathfrak{C}_3(s) (1+\vert \tau\vert )\varepsilon ^s;\quad \mathfrak{C}_3(s):={ {C}}_{22}^{\,1-s}\ {C}_{19}^s\quad\mbox{for}\quad d\leqslant 4,\quad \mathfrak{C}_3(s):={ {C}}_{22}^{\,1-s}\ {C}_{20}^s\quad\mbox{for}\quad d\geqslant 5.
\end{split}
\end{equation*}
\end{theorem}

\subsection{The case where the corrector is equal to zero}

Assume that $g^0=\overline{g}$, i.~e., relations  \eqref{overline-g} are valid. Then the $\Gamma$-periodic solution of problem \eqref{Lambda problem} is equal to zero: $\Lambda (\mathbf{x})=0$, and Theorem \ref{Theorem 12.4} implies the following result.

\begin{proposition}
Suppose that relations \eqref{overline-g} hold. Then under the assumptions of Theorem~\textnormal{\ref{Theorem 12.1}}, for $0\leqslant s\leqslant 1$ and $\tau\in\mathbb{R}$, $0<\varepsilon\leqslant 1$ we have
\begin{equation*}
\begin{split}
\Vert  f^\varepsilon \mathcal{A}_\varepsilon ^{-1/2}\sin (\tau \mathcal{A}_\varepsilon ^{1/2})(f^\varepsilon)^{-1}
-
f_0 (\mathcal{A}^0)^{-1/2}\sin (\tau (\mathcal{A}^0)^{1/2})f_0^{-1}\Vert _{H^{s+1}(\mathbb{R}^d)\rightarrow H^1(\mathbb{R}^d)}
\leqslant
\mathfrak{C}_2(s) (1+\vert \tau\vert )\varepsilon ^s.
\end{split}
\end{equation*}
\end{proposition}

\subsection{The Steklov smoothing operator. Another approximation with the corrector}

Let us show that the result of Theorem \ref{Theorem 12.3} remains true with the operator  $\Pi _\varepsilon$  
replaced by another smoothing operator.

The operator $S_\varepsilon$, $\varepsilon >0$, acting in $L_2(\mathbb{R}^d;\mathbb{C}^n)$ and defined by the relation
\begin{equation}
\label{S_eps}
(S_\varepsilon \mathbf{u})(\mathbf{x})=\vert\Omega\vert ^{-1}\int _\Omega \mathbf{u}(\mathbf{x}-\varepsilon\mathbf{z})\,d\mathbf{z},\quad\mathbf{u}\in L_2(\mathbb{R}^d;\mathbb{C}^n),
\end{equation}
is called the Steklov smoothing operator.

The following properties of the operator $S_\varepsilon$ were checked in \cite[Lemmata 1.1 and 1.2]{ZhPas} (see also \cite[Propositions 3.1 and 3.2]{PSu}).

\begin{proposition}
\label{Proposition S_eps -I}
For any function $\mathbf{u}\in H^1(\mathbb{R}^d;\mathbb{C}^n)$ we have
\begin{equation*}
\Vert S_\varepsilon\mathbf{u}-\mathbf{u}\Vert _{L_2(\mathbb{R}^d)}\leqslant \varepsilon r_1\Vert \mathbf{D}\mathbf{u}\Vert _{L_2(\mathbb{R}^d)},\quad\varepsilon >0.
\end{equation*}
\end{proposition}

\begin{proposition}
\label{Proposition S_eps f^eps}
Let $\Phi(\mathbf{x})$ be a $\Gamma$-periodic function in $\mathbb{R}^d$ such that $\Phi \in L_2(\Omega)$. Then the operator $[\Phi ^\varepsilon]S_\varepsilon$ is bounded in $L_2(\mathbb{R}^d;\mathbb{C}^n)$, and
\begin{equation*}
\Vert [\Phi ^\varepsilon]S_\varepsilon\Vert _{L_2(\mathbb{R}^d)\rightarrow L_2(\mathbb{R}^d)}\leqslant\vert\Omega\vert ^{-1/2}\Vert \Phi \Vert _{L_2(\Omega)},\quad\varepsilon >0.
\end{equation*}
\end{proposition}

We also need the following statement obtained in \cite[Lemma 3.5]{PSu}.

\begin{proposition}
\label{Prop Pi -S}
Let $\Pi _\varepsilon $ be the operator \eqref{Pi eps} and let $S_\varepsilon$ be the Steklov smoothing operator \eqref{S_eps}. Let $\Lambda (\mathbf{x})$ be the $\Gamma$-periodic solution of problem \eqref{Lambda problem}. Then
\begin{equation*}
 \Vert [\Lambda ^\varepsilon ]b(\mathbf{D})(\Pi _\varepsilon -S_\varepsilon )\Vert _{H^2(\mathbb{R}^d)\rightarrow H^1(\mathbb{R}^d)}\leqslant C_{23} ,\quad\varepsilon >0,
\end{equation*}
where the constant $C_{23}$ depends only on $d$, $m$, $r_0$, $r_1$, $\alpha _0$, $\alpha _1$, $\Vert g\Vert _{L_\infty}$, and $\Vert g^{-1}\Vert _{L_\infty}$.
\end{proposition}

Using \eqref{f0<=}, \eqref{A0 no hat}, Proposition~\ref{Prop Pi -S}, and Theorem~\ref{Theorem 12.3}, we obtain the following result.

\begin{theorem}
\label{Theorem 12.7}
Suppose that the assumptions of Theorem \textnormal{\ref{Theorem 12.1}} are satisfied. Let $\Lambda (\mathbf{x})$ be the $\Gamma$-periodic $(n\times m)$-matrix-valued solution of problem \eqref{Lambda problem}. Let $S_\varepsilon$ be the Steklov smoothing operator \eqref{S_eps}. Then for $\tau\in\mathbb{R}$ and $\varepsilon>0$ we have
\begin{equation*}
\begin{split}
\bigl\Vert & 
f^\varepsilon \mathcal{A}_\varepsilon ^{-1/2}\sin (\tau\mathcal{A}_\varepsilon ^{1/2})(f^\varepsilon)^{-1}
-(I+\varepsilon \Lambda ^\varepsilon b(\mathbf{D})S_\varepsilon)
f_0(\mathcal{A}^0)^{-1/2}\sin (\tau (\mathcal{A}^0)^{1/2})f_0^{-1}
\bigr\Vert _{H^2(\mathbb{R}^d)\rightarrow H^1(\mathbb{R}^d)}
\\
&\leqslant
{C}_{24}\varepsilon (1+\vert \tau\vert ),
\end{split}
\end{equation*}
where the constant ${C}_{24}:={C}_{17}+C_{23}\Vert f\Vert _{L_\infty}\Vert f^{-1}\Vert _{L_\infty}$ depends only on $d$, $m$, $r_0$, $r_1$, $\alpha _0$, $\alpha _1$, $\Vert g\Vert _{L_\infty}$, $\Vert g^{-1}\Vert _{L_\infty}$, $\Vert f\Vert _{L_\infty}$, and $\Vert f^{-1}\Vert _{L_\infty}$.
\end{theorem}

\begin{remark}
\label{Remark 12.8}
Similarly to the proof of Theorem \textnormal{\ref{Theorem 12.4}}, using the properties of the Steklov smoothing, one can check that the estimate of the form \eqref{12.6} remains true with $\Pi _\varepsilon$ replaced by $S_\varepsilon$.
\end{remark}

\section{Homogenization of hyperbolic systems with periodic coefficients}

\label{Section hyperbolic general case}

\subsection{The statement of the problem. Homogenization for the solutions of hyperbolic systems} 
Our goal is to apply the results of Section~\ref{Section main results in general case} to homogenization for the solutions of the problem
\begin{equation}
\label{14.1}
\begin{cases}
Q^\varepsilon (\mathbf{x})\frac{\partial ^2\mathbf{u}_\varepsilon (\mathbf{x},\tau)}{ \partial \tau ^2}
=-b(\mathbf{D})^*g^\varepsilon (\mathbf{x})b(\mathbf{D})\mathbf{u}_\varepsilon (\mathbf{x},\tau)+ Q^\varepsilon (\mathbf{x})\mathbf{F}(\mathbf{x},\tau),
\\
\mathbf{u}_\varepsilon (\mathbf{x},0)=0,\quad\frac{\partial\mathbf{u}_\varepsilon (\mathbf{x},0)}{\partial\tau}=\boldsymbol{\psi}(\mathbf{x}),
\end{cases}
\end{equation}
where $\boldsymbol{\psi}\in L_2(\mathbb{R}^d;\mathbb{C}^n)$, $\mathbf{F}\in L_{1,\mathrm{loc}}(\mathbb{R};L_2(\mathbb{R}^d;\mathbb{C}^n))$, and $Q(\mathbf{x})$ is a $\Gamma$-periodic $(n\times n)$-matrix-valued function \eqref{8.1a}. Substituting $\mathbf{z}_\varepsilon (\cdot ,\tau):=(f^\varepsilon)^{-1}\mathbf{u}_\varepsilon (\cdot ,\tau)$ into \eqref{14.1}, we rewrite problem \eqref{14.1} as
\begin{equation*}
\begin{cases}
\frac{\partial ^2\mathbf{z}_\varepsilon (\mathbf{x},\tau)}{ \partial \tau ^2}
=-f^\varepsilon(\mathbf{x})^* b(\mathbf{D})^*g^\varepsilon (\mathbf{x})b(\mathbf{D})f^\varepsilon (\mathbf{x})\mathbf{z}_\varepsilon (\mathbf{x},\tau)
+f^\varepsilon (\mathbf{x})^{-1}\mathbf{F}(\mathbf{x},\tau),
\\
\mathbf{z}_\varepsilon (\mathbf{x},0)=0,\quad \frac{\partial\mathbf{z}_\varepsilon (\mathbf{x},0)}{\partial\tau}=f^\varepsilon (\mathbf{x})^{-1}\boldsymbol{\psi}(\mathbf{x}).
\end{cases}
\end{equation*}
Then
\begin{equation}
\label{14.2}
\begin{split}
\mathbf{z}_\varepsilon (\cdot ,\tau)
=\mathcal{A}_\varepsilon ^{-1/2}\sin (\tau\mathcal{A}_\varepsilon ^{1/2})(f^\varepsilon)^{-1}\boldsymbol{\psi}
+\int _0^\tau \mathcal{A}_\varepsilon ^{-1/2}\sin ( (\tau -\widetilde{\tau})\mathcal{A}_\varepsilon ^{1/2})(f^\varepsilon)^{-1}\mathbf{F}(\cdot ,\widetilde{\tau})\,d\widetilde{\tau}
\end{split}
\end{equation}
and
\begin{equation}
\label{14.3}
\begin{split}
\mathbf{u}_\varepsilon (\cdot,\tau)&=
f^\varepsilon\mathcal{A}_\varepsilon ^{-1/2}\sin (\tau\mathcal{A}_\varepsilon ^{1/2})(f^\varepsilon)^{-1}\boldsymbol{\psi}
+\int _0^\tau f^\varepsilon \mathcal{A}_\varepsilon ^{-1/2}\sin ( (\tau -\widetilde{\tau})\mathcal{A}_\varepsilon ^{1/2})(f^\varepsilon)^{-1}\mathbf{F}(\cdot ,\widetilde{\tau})\,d\widetilde{\tau}.
\end{split}
\end{equation}

Let $\mathbf{u}_0(\mathbf{x},\tau)$ be the solution of the effective problem
\begin{equation}
\label{14.4}
\begin{cases}
\overline{Q}\frac{\partial ^2\mathbf{u}_0(\mathbf{x},\tau)}{\partial \tau ^2}=
-b(\mathbf{D})^*g^0b(\mathbf{D})\mathbf{u}_0(\mathbf{x},\tau)+ \overline{Q}\mathbf{F}(\mathbf{x},\tau),
\\
\mathbf{u}_0(\mathbf{x},0)=0,\quad\frac{\partial\mathbf{u}_0(\mathbf{x},0)}{\partial\tau}=\boldsymbol{\psi}(\mathbf{x}),
\end{cases}
\end{equation}
where $\overline{Q}=\vert\Omega\vert ^{-1}\int _\Omega Q(\mathbf{x})\,d\mathbf{x}$. Similarly to \eqref{14.2} and \eqref{14.3}, we obtain
\begin{equation}
\label{14.5}
\mathbf{u}_0(\cdot ,\tau)
=f_0 (\mathcal{A}^0)^{-1/2}\sin (\tau (\mathcal{A}^0)^{1/2})f_0^{-1}\boldsymbol{\psi}
+\int _0^\tau f_0(\mathcal{A}^0)^{-1/2}\sin ((\tau -\widetilde{\tau})(\mathcal{A}^0)^{1/2})f_0^{-1} \mathbf{F}(\cdot,\widetilde{\tau})\,d\widetilde{\tau}.
\end{equation}

Using Theorems \ref{Theorem 12.1}, \ref{Theorem 12.3}, and \ref{Theorem 12.7}, and identities \eqref{14.3} and \eqref{14.5}, we arrive at the following result.

\begin{theorem}
\label{Theorem 14.1}
Let $\mathbf{u}_\varepsilon$ be the solution of problem \eqref{14.1} and let $\mathbf{u}_0$ be the solution of the effective problem \eqref{14.4}.

\noindent
$1^\circ$. Let $\boldsymbol{\psi}\in H^1(\mathbb{R}^d;\mathbb{C}^n)$ and let $\mathbf{F}\in L_{1,\mathrm{loc}}(\mathbb{R};H^1(\mathbb{R}^d;\mathbb{C}^n))$. Then for $\tau\in\mathbb{R}$ and $\varepsilon >0$ we have
\begin{equation*}
\Vert \mathbf{u}_\varepsilon (\cdot ,\tau)-\mathbf{u}_0(\cdot,\tau)\Vert _{L_2(\mathbb{R}^d)}
\leqslant
C_{12}\varepsilon(1+\vert\tau\vert )
\left(
\Vert \boldsymbol{\psi}\Vert _{H^1(\mathbb{R}^d)}
+\Vert \mathbf{F}\Vert _{L_1((0,\tau);H^1(\mathbb{R}^d))}\right).
\end{equation*}

\noindent
$2^\circ$. Let $\boldsymbol{\psi}\in H^2(\mathbb{R}^d;\mathbb{C}^n)$ and let $\mathbf{F}\in L_{1,\mathrm{loc}}(\mathbb{R};H^2(\mathbb{R}^d;\mathbb{C}^n))$. Let $\Lambda (\mathbf{x})$ be the $\Gamma$-periodic solution of problem \eqref{Lambda problem}. Let $\Pi _\varepsilon$ be the smoothing operator \eqref{Pi eps}. By $\mathbf{v}_\varepsilon$ we denote the first order approximation\textnormal{:}
\begin{equation}
\label{14.5a}
\mathbf{v}_\varepsilon (\mathbf{x},\tau):=\mathbf{u}_0(\mathbf{x},\tau)+\varepsilon \Lambda ^\varepsilon b(\mathbf{D})\Pi_\varepsilon\mathbf{u}_0(\mathbf{x},\tau).
\end{equation}
Then for $\tau\in\mathbb{R}$ and $\varepsilon >0$ we have
\begin{equation}
\label{Th H1 solutions no interpolation}
\Vert \mathbf{u}_\varepsilon (\cdot ,\tau)-\mathbf{v}_\varepsilon(\cdot,\tau)\Vert _{H^1(\mathbb{R}^d)}
\leqslant
C_{17}\varepsilon(1+\vert\tau\vert )
\left(
 \Vert \boldsymbol{\psi}\Vert _{H^2(\mathbb{R}^d)}
+\Vert \mathbf{F}\Vert _{L_1((0,\tau);H^2(\mathbb{R}^d))}\right).
\end{equation}
Let $S_\varepsilon$ be the Steklov smoothing operator \eqref{S_eps}. We put
\begin{equation*}
\check{\mathbf{v}}_\varepsilon (\mathbf{x},\tau):=\mathbf{u}_0(\mathbf{x},\tau)+\varepsilon\Lambda ^\varepsilon b(\mathbf{D})S_\varepsilon \mathbf{u}_0(\mathbf{x},\tau).
\end{equation*}
Then for $\tau\in\mathbb{R}$ and $\varepsilon >0$ we have
\begin{equation*}
\Vert \mathbf{u}_\varepsilon (\cdot ,\tau)-\check{\mathbf{v}}_\varepsilon(\cdot,\tau)\Vert _{H^1(\mathbb{R}^d)}
\leqslant
C_{24}\varepsilon(1+\vert\tau\vert )
\left(
\Vert \boldsymbol{\psi}\Vert _{H^2(\mathbb{R}^d)}
+\Vert \mathbf{F}\Vert _{L_1((0,\tau);H^2(\mathbb{R}^d))}\right).
\end{equation*}
\end{theorem}

\begin{remark}
If $d\leqslant 4$ \textnormal{(}or $d\geqslant 5$ and Condition~\textnormal{\ref{Condition Lambda multiplier}} is satisfied\textnormal{),} then we can use Theorem~\textnormal{\ref{Theorem d<=4 chapter 3}} \textnormal{(}respectively, Theorem~\textnormal{\ref{Theorem d>=5 eps})}, i.~e., the estimate of the form \eqref{Th H1 solutions no interpolation} is valid with $\mathbf{v}_\varepsilon$ replaced by
\begin{equation}
\label{corrector V_eps^0}
\mathbf{v}_\varepsilon ^{(0)}(\mathbf{x},\tau):=\mathbf{u}_0(\mathbf{x},\tau)+\varepsilon \Lambda ^\varepsilon b(\mathbf{D})\mathbf{u}_0(\mathbf{x},\tau).
\end{equation}
\end{remark}

Theorem~\ref{Theorem D>=5 smooth date sec 9} implies the following statement.

\begin{proposition}
Assume that $d\geqslant 5$. Let $\boldsymbol{\psi}\in H^{d/2}(\mathbb{R}^d;\mathbb{C}^n)$ and let $\mathbf{F}\in L_{1,\mathrm{loc}}(\mathbb{R};H^{d/2}(\mathbb{R}^d;\mathbb{C}^n))$. Let $\mathbf{u}_\varepsilon$ and $\mathbf{u}_0$ be the solutions of problems \eqref{14.1} and  \eqref{14.4} respectively. Let $\mathbf{v}_\varepsilon ^{(0)}$ be given by \eqref{corrector V_eps^0}. Then for $\tau\in\mathbb{R}$ and $0<\varepsilon\leqslant 1$ we have
\begin{equation*}
\Vert \mathbf{u}_\varepsilon (\cdot ,\tau)-{\mathbf{v}}_\varepsilon ^{(0)}(\cdot,\tau)\Vert _{H^1(\mathbb{R}^d)}
\leqslant
C_{21}\varepsilon(1+\vert\tau\vert )
\left(
\Vert \boldsymbol{\psi}\Vert _{H^{d/2}(\mathbb{R}^d)}
+\Vert \mathbf{F}\Vert _{L_1((0,\tau);H^{d/2}(\mathbb{R}^d))}\right).
\end{equation*}
\end{proposition}

Applying Theorems \ref{Theorem 12.2} and \ref{Theorem 12.4}, we arrive at the following result.

\begin{theorem}
\label{Theorem 14.2}
Let $\mathbf{u}_\varepsilon$ be the solution of problem \eqref{14.1} and let $\mathbf{u}_0$ be the solution of the effective problem \eqref{14.4}.

\noindent
$1^\circ$. 
Let $\boldsymbol{\psi}\in H^s(\mathbb{R}^d;\mathbb{C}^n)$ and $\mathbf{F}\in L_{1,\mathrm{loc}}(\mathbb{R};H^s(\mathbb{R}^d;\mathbb{C}^n))$, $0\leqslant s\leqslant 1$. Then for $\tau\in\mathbb{R}$ and $\varepsilon >0$ we have
\begin{equation*}
\Vert \mathbf{u}_\varepsilon (\cdot ,\tau)-\mathbf{u}_0(\cdot ,\tau)\Vert _{L_2(\mathbb{R}^d)}
\leqslant \mathfrak{C}_1(s)(1+\vert\tau\vert)\varepsilon ^s\left(\Vert \boldsymbol{\psi}\Vert _{H^s(\mathbb{R}^d)}+ \Vert \mathbf{F}\Vert _{L_{1}((0,\tau);H^s(\mathbb{R}^d))}\right).
\end{equation*}
Under the additional assumption that $\mathbf{F}\in L_1(\mathbb{R}_\pm; H^{s}(\mathbb{R}^d;\mathbb{C}^n))$, for $0<s\leqslant 1$, $\vert \tau\vert =\varepsilon ^{-\alpha}$, $0<\varepsilon \leqslant 1$, $0<\alpha<s$, we have
\begin{equation*}
\begin{split}
\Vert \mathbf{u}_\varepsilon (\cdot ,\pm \varepsilon ^{-\alpha})-\mathbf{u}_0(\cdot ,\pm \varepsilon ^{-\alpha})\Vert _{L_2(\mathbb{R}^d)}
\leqslant 2 \mathfrak{C}_1(s) \varepsilon ^{s-\alpha}\left(\Vert \boldsymbol{\psi}\Vert _{H^s(\mathbb{R}^d)}+ \Vert \mathbf{F}\Vert _{L_{1}(\mathbb{R}_\pm ;H^s(\mathbb{R}^d))}\right).
\end{split}
\end{equation*}

\noindent
$2^\circ$. 
Let $\boldsymbol{\psi}\in H^{1+s}(\mathbb{R}^d;\mathbb{C}^n)$ and $\mathbf{F}\in L_{1,\mathrm{loc}}(\mathbb{R};H^{1+s}(\mathbb{R}^d;\mathbb{C}^n))$, $0\leqslant s\leqslant 1$.
 Let $\mathbf{v}_\varepsilon$ be given by \eqref{14.5a}. Then for $\tau\in\mathbb{R}$ and $0<\varepsilon\leqslant 1$ we have
\begin{equation*}
\Vert \mathbf{u}_\varepsilon(\cdot ,\tau)-\mathbf{v}_\varepsilon (\cdot ,\tau)\Vert _{H^1(\mathbb{R}^d)}
\leqslant  \mathfrak{C}_2(s)(1+\vert\tau\vert)\varepsilon ^s\left( \Vert \boldsymbol{\psi}\Vert _{H^{1+s}(\mathbb{R}^d)}+\Vert \mathbf{F}\Vert _{L_{1}((0,\tau);H^{1+s}(\mathbb{R}^d))}\right).
\end{equation*}
Under the additional assumption that $\mathbf{F}\in L_1(\mathbb{R}_\pm; H^{1+s}(\mathbb{R}^d;\mathbb{C}^n))$, where $0<s\leqslant 1$, for $\tau=\pm\varepsilon^{-\alpha}$, $0<\varepsilon\leqslant 1$, $0<\alpha<s$, we have
\begin{equation*}
\begin{split}
\Vert \mathbf{u}_\varepsilon(\cdot ,\pm\varepsilon^{-\alpha})-\mathbf{v}_\varepsilon (\cdot ,\pm\varepsilon^{-\alpha})\Vert _{H^1(\mathbb{R}^d)}
\leqslant 2\mathfrak{C}_2(s)\varepsilon ^{s-\alpha }\left( \Vert \boldsymbol{\psi}\Vert _{H^{1+s}(\mathbb{R}^d)}+ \Vert \mathbf{F}\Vert _{L_{1}(\mathbb{R}_\pm ;H^{1+s}(\mathbb{R}^d))}\right) .
 \end{split}
\end{equation*}
\end{theorem}

By the Banach-Steinhaus theorem, this result implies the following theorem.

\begin{theorem}
\label{Theorem 14.3}
Let $\mathbf{u}_\varepsilon$ be the solution of problem \eqref{14.1}, and let $\mathbf{u}_0$ be the solution of the effective problem \eqref{14.4}.

\noindent
$1^\circ$. 
Let $\boldsymbol{\psi}\in L_2(\mathbb{R}^d;\mathbb{C}^n)$ and $\mathbf{F}\in L_{1,\mathrm{loc}}(\mathbb{R};L_2(\mathbb{R}^d;\mathbb{C}^n))$. Then
\begin{equation*}
\lim _{\varepsilon\rightarrow 0 }\Vert \mathbf{u}_\varepsilon (\cdot ,\tau)-\mathbf{u}_0(\cdot ,\tau)\Vert _{L_2(\mathbb{R}^d)}
=0,\quad \tau\in\mathbb{R}.
\end{equation*}

\noindent
$2^\circ$. 
Let $\boldsymbol{\psi}\in H^{1}(\mathbb{R}^d;\mathbb{C}^n)$ and $\mathbf{F}\in L_{1,\mathrm{loc}}(\mathbb{R};H^{1}(\mathbb{R}^d;\mathbb{C}^n))$.
 Let $\mathbf{v}_\varepsilon$ be given by \eqref{14.5a}. Then for $\tau\in\mathbb{R}$  we have
\begin{equation*}
\lim _{\varepsilon\rightarrow 0}\Vert \mathbf{u}_\varepsilon(\cdot ,\tau)-\mathbf{v}_\varepsilon (\cdot ,\tau)\Vert _{H^1(\mathbb{R}^d)}
=0.
\end{equation*} 
\end{theorem}

\begin{remark}
Taking Remark \textnormal{\ref{Remark 12.8}} into account, we see that the results of Theorems \textnormal{\ref{Theorem 14.2}($2^\circ$)} and \textnormal{\ref{Theorem 14.3}($2^\circ$)} remain true with the operator $\Pi _\varepsilon$ replaced by the Steklov smoothing $S_\varepsilon$, i. e., 
with $\mathbf{v}_\varepsilon$ replaced by $\check{\mathbf{v}}_\varepsilon$. 
This only changes the constants in estimates.
\end{remark}

Applying Theorem \ref{Theorem 12.6}, we make the following observation.

\begin{remark}
\label{Remark no S-eps}
For $0<\varepsilon\leqslant 1$, under Condition \textnormal{\ref{Condition Lambda in L infty}}, the analogs of Theorems \textnormal{\ref{Theorem 14.1}, \ref{Theorem 14.2}}, and \textnormal{\ref{Theorem 14.3}} are valid with the operators $\Pi _\varepsilon$ and $S_\varepsilon$ replaced by the identity operator.
\end{remark}

\subsection{Approximation of the flux}

Let $\mathbf{p}_\varepsilon (\mathbf{x},\tau )$ be the ,,flux''
\begin{equation}
\label{P_eps ho hat}
\mathbf{p}_\varepsilon (\mathbf{x},\tau):=g^\varepsilon (\mathbf{x})b(\mathbf{D})\mathbf{u}_\varepsilon (\mathbf{x},\tau).
\end{equation}

\begin{theorem}
\label{Theorem 14.6}
Suppose that the assumptions of Theorem~\textnormal{\ref{Theorem 14.1}($2^\circ$)} are satisfied. Let $\mathbf{p}_\varepsilon $ be the ,,flux'' \eqref{P_eps ho hat}, and let $\widetilde{g}(\mathbf{x})$ be the matrix-valued function \eqref{tilde g}. Then for $\tau\in\mathbb{R}$ and $\varepsilon >0$ we have
\begin{align}
\label{flux with Pi_eps no hat}
&\Vert \mathbf{p}_\varepsilon (\cdot,\tau)-\widetilde{g}^\varepsilon b(\mathbf{D})\Pi _\varepsilon \mathbf{u}_0(\cdot ,\tau)\Vert _{L_2(\mathbb{R}^d)}
\leqslant C_{25} \varepsilon (1+\vert\tau\vert)\left(\Vert \boldsymbol{\psi}\Vert _{H^2(\mathbb{R}^d)}+\Vert \mathbf{F}\Vert _{L_{1}((0,\tau);H^2(\mathbb{R}^d))}\right),
\\
\label{flux with S_eps no hat}
&\Vert \mathbf{p}_\varepsilon (\cdot,\tau)-\widetilde{g}^\varepsilon b(\mathbf{D})S _\varepsilon \mathbf{u}_0(\cdot ,\tau)\Vert _{L_2(\mathbb{R}^d)}
\leqslant C_{26}\varepsilon (1+\vert\tau\vert)\left(\Vert \boldsymbol{\psi}\Vert _{H^2(\mathbb{R}^d)}+\Vert \mathbf{F}\Vert _{L_{1}((0,\tau);H^2(\mathbb{R}^d))}\right).
\end{align}
The constants ${C}_{25}$ and ${C}_{26}$ depend only on $m$, $d$, $\alpha _0$, $\alpha _1$, $\Vert g\Vert _{L_\infty}$, $\Vert g^{-1}\Vert _{L_\infty}$, $\Vert f\Vert _{L_\infty}$, $\Vert f^{-1}\Vert _{L_\infty}$, and the~parameters of the lattice $\Gamma$.
\end{theorem}

\begin{proof}
From \eqref{isometria}, \eqref{12.3a-2}, \eqref{14.3}, and \eqref{14.5}, it follows that
\begin{equation}
\label{14.6}
\begin{split}
\Bigl\Vert & \widehat{\mathcal{A}}_\varepsilon ^{1/2}
\Bigl(
\mathbf{u}_\varepsilon (\cdot ,\tau)-(I+\varepsilon \Lambda ^\varepsilon b(\mathbf{D})\Pi _\varepsilon )\mathbf{u}_0(\cdot,\tau)\Bigr)\Bigr\Vert _{L_2(\mathbb{R}^d)}
\\
&\leqslant
 C_{13}\varepsilon(1+\vert\tau\vert)
\left(\Vert \boldsymbol{\psi}\Vert _{H^2(\mathbb{R}^d)}+\Vert \mathbf{F}\Vert _{L_{1}((0,\tau);H^2(\mathbb{R}^d))}\right).
\end{split}
\end{equation}
By  \eqref{A_eps} and Proposition~\ref{Proposition Pi eps -I}, 
\begin{equation}
\label{14.7}
\begin{split}
\Vert &\widehat{\mathcal{A}}_\varepsilon ^{1/2}(\Pi _\varepsilon -I)\mathbf{u}_0(\cdot ,\tau)\Vert _{L_2(\mathbb{R}^d)}
\leqslant\varepsilon \alpha _1^{1/2}r_0^{-1}\Vert g\Vert ^{1/2}_{L_\infty}
\Vert \mathbf{D}^2 \mathbf{u}_0(\cdot,\tau)\Vert _{L_2(\mathbb{R}^d)}
.
\end{split}
\end{equation}
Using \eqref{f0<=},  \eqref{14.5},  and the  inequality $\vert \sin x\vert /\vert x\vert \leqslant 1$, $x\in\mathbb{R}$, we obtain
\begin{equation}
\label{u0 in H2<=}
\Vert \mathbf{D}^2 \mathbf{u}_0(\cdot,\tau)\Vert _{L_2(\mathbb{R}^d)}\leqslant\Vert \mathbf{u}_0(\cdot,\tau)\Vert _{H^2(\mathbb{R}^d)}\leqslant \vert \tau\vert \Vert f\Vert _{L_\infty}\Vert f^{-1}\Vert _{L_\infty}\left(\Vert \boldsymbol{\psi}\Vert _{H^2(\mathbb{R}^d)}+\Vert \mathbf{F}\Vert _{L_{1}((0,\tau);H^2(\mathbb{R}^d))}\right).
\end{equation}
Combining \eqref{P_eps ho hat} and \eqref{14.6}--\eqref{u0 in H2<=}, we arrive at
\begin{equation}
\label{14.8}
\begin{split}
\Vert &\mathbf{p}_\varepsilon (\cdot ,\tau)-g^\varepsilon b(\mathbf{D})(I+\varepsilon \Lambda ^\varepsilon b(\mathbf{D}))\Pi_\varepsilon \mathbf{u}_0(\cdot ,\tau)\Vert_{L_2(\mathbb{R}^d)}
\\
&\leqslant
C_{27}\varepsilon (1+\vert \tau\vert)
\left(\Vert \boldsymbol{\psi}\Vert _{H^2(\mathbb{R}^d)}+\Vert \mathbf{F}\Vert _{L_{1}((0,\tau);H^2(\mathbb{R}^d))}\right),
\end{split}
\end{equation}
where ${{C}}_{27}:={C}_{13}\Vert g\Vert ^{1/2}_{L_\infty}
+\alpha_1^{1/2}r_0^{-1}\Vert g\Vert _{L_\infty}\Vert f\Vert _{L_\infty}\Vert f^{-1}\Vert _{L_\infty}$.

We have
\begin{equation}
\label{proof fluxes 5}
\begin{split}
\varepsilon g^\varepsilon b(\mathbf{D})\Lambda^\varepsilon b(\mathbf{D})\Pi _\varepsilon\mathbf{u}_0(\cdot ,\tau)
=g^\varepsilon (b(\mathbf{D})\Lambda)^\varepsilon b(\mathbf{D})\Pi_\varepsilon\mathbf{u}_0(\cdot ,\tau)
+\varepsilon g^\varepsilon \sum _{l=1}^d b_l \Lambda ^\varepsilon \Pi _\varepsilon ^{(m)}D_lb(\mathbf{D})\mathbf{u}_0(\cdot ,\tau).
\end{split}
\end{equation}
By \eqref{<b^*b<}, \eqref{b_j <=}, \eqref{proof fluxes 6}, and \eqref{u0 in H2<=},
\begin{equation}
\label{14.9}
\begin{split}
\Bigl\Vert &\varepsilon g^\varepsilon \sum _{l=1}^d b_l\Lambda ^\varepsilon \Pi _\varepsilon ^{(m)}D_l b(\mathbf{D})\mathbf{u}_0(\cdot,\tau)\Bigr\Vert _{L_2(\mathbb{R}^d)}
\leqslant
\varepsilon \Vert g\Vert _{L_\infty}\alpha _1 d^{1/2}M_1\Vert \mathbf{D}^2\mathbf{u}_0(\cdot ,\tau)\Vert _{L_2(\mathbb{R}^d)}
\\
&\leqslant
\varepsilon\vert\tau\vert \alpha _1 d^{1/2}M_1\Vert g\Vert _{L_\infty}
\Vert f\Vert _{L_\infty}\Vert f^{-1}\Vert _{L_\infty}\left(\Vert \boldsymbol{\psi}\Vert _{H^2(\mathbb{R}^d)}+
\Vert \mathbf{F}\Vert _{L_{1}((0,\tau);H^2(\mathbb{R}^d))}\right).
\end{split}
\end{equation}
Now, relations \eqref{tilde g} and \eqref{14.8}--\eqref{14.9} imply estimate \eqref{flux with Pi_eps no hat} with the constant  ${C}_{25}:={{C}}_{27}+\alpha _1 d^{1/2}M_1\Vert g\Vert _{L_\infty}\Vert f\Vert _{L_\infty}\Vert f^{-1}\Vert _{L_\infty}$.

We proceed to the proof of inequality \eqref{flux with S_eps no hat}. By \eqref{14.6},
\begin{equation}
\label{14.9-2}
\begin{split}
\Bigl\Vert &\widehat{\mathcal{A}}_\varepsilon ^{1/2}
\Bigl(
\mathbf{u}_\varepsilon (\cdot ,\tau)
-(I+\varepsilon \Lambda ^\varepsilon b(\mathbf{D}))S_\varepsilon \mathbf{u}_0(\cdot ,\tau)\Bigr)
\Bigr\Vert _{L_2(\mathbb{R}^d)}
\\
&\leqslant
C_{13}\varepsilon (1+\vert\tau\vert)\left(\Vert \boldsymbol{\psi}\Vert _{H^2(\mathbb{R}^d)}+\Vert \mathbf{F}\Vert _{L_{1}((0,\tau);H^2(\mathbb{R}^d))}\right)
\\
&+\Vert \widehat{\mathcal{A}}_\varepsilon ^{1/2}(S_\varepsilon -I)\mathbf{u}_0(\cdot ,\tau)\Vert _{L_2(\mathbb{R}^d)}
+\varepsilon\Vert \widehat{\mathcal{A}}_\varepsilon ^{1/2}\Lambda^\varepsilon b(\mathbf{D})(\Pi _\varepsilon -S_\varepsilon)\mathbf{u}_0(\cdot,\tau)\Vert _{L_2(\mathbb{R}^d)}.
\end{split}
\end{equation}
Similarly to \eqref{14.7}, using 
Proposition~\ref{Proposition S_eps -I} and \eqref{u0 in H2<=}, we have
\begin{equation}
\label{14.10}
\begin{split}
\Vert &\widehat{\mathcal{A}}_\varepsilon ^{1/2}(S_\varepsilon -I)\mathbf{u}_0(\cdot ,\tau)\Vert _{L_2(\mathbb{R}^d)}
\leqslant
\varepsilon  r_1
\alpha _1^{1/2}\Vert g\Vert ^{1/2}_{L_\infty}\Vert \mathbf{D}^2\mathbf{u}_0(\cdot ,\tau)\Vert _{L_2(\mathbb{R}^d)}
\\
&\leqslant
\varepsilon \vert\tau\vert r_1
\alpha _1^{1/2}\Vert g\Vert ^{1/2}_{L_\infty}\Vert f\Vert _{L_\infty}\Vert f^{-1}\Vert _{L_\infty} \left(\Vert \boldsymbol{\psi}\Vert _{H^2(\mathbb{R}^d)}+\Vert \mathbf{F}\Vert _{L_{1}((0,\tau);H^2(\mathbb{R}^d))}\right).
\end{split}
\end{equation}
To estimate the third summand in the right-hand side of \eqref{14.9-2}, we use \eqref{<b^*b<}, \eqref{A_eps}, and Proposition~\ref{Prop Pi -S}. Then
\begin{equation}
\label{14.11}
\begin{split}
\varepsilon \Vert &\widehat{\mathcal{A}}_\varepsilon ^{1/2}\Lambda ^\varepsilon b(\mathbf{D})(\Pi _\varepsilon -S_\varepsilon )\mathbf{u}_0(\cdot,\tau)\Vert _{L_2(\mathbb{R}^d)}
\leqslant \varepsilon \alpha_1^{1/2}C_{23}\Vert g\Vert ^{1/2}_{L_\infty}\Vert \mathbf{u}_0(\cdot ,\tau)\Vert _{H^2(\mathbb{R}^d)}
\\
&\leqslant
\varepsilon\vert\tau\vert \alpha _1^{1/2}C_{23}\Vert g\Vert ^{1/2}_{L_\infty}
\Vert f\Vert _{L_\infty}\Vert f^{-1}\Vert _{L_\infty}\left( \Vert \boldsymbol{\psi}\Vert _{H^2(\mathbb{R}^d)}+\Vert \mathbf{F}\Vert _{L_{1}((0,\tau);H^2(\mathbb{R}^d))}\right).
\end{split}
\end{equation}
Combining \eqref{A_eps}, \eqref{P_eps ho hat}, and \eqref{14.9-2}--\eqref{14.11}, we have
\begin{equation}
\label{14.12}
\begin{split}
\Vert& \mathbf{p}_\varepsilon (\cdot ,\tau)-g^\varepsilon b(\mathbf{D})(I+\varepsilon \Lambda ^\varepsilon b(\mathbf{D}))S_\varepsilon \mathbf{u}_0 (\cdot ,\tau)\Vert _{L_2(\mathbb{R}^d)}
\\
&\leqslant
C_{28}\varepsilon (1+\vert \tau\vert )
\left(\Vert \boldsymbol{\psi}\Vert _{H^2(\mathbb{R}^d)}+\Vert \mathbf{F}\Vert _{L_{1}((0,\tau);H^2(\mathbb{R}^d))}\right).
\end{split}
\end{equation}
Here 
$
{{C}}_{28}:={C}_{13}\Vert g\Vert ^{1/2}_{L_\infty}
+\alpha _1^{1/2}(r_1+C_{23} )\Vert g\Vert _{L_\infty}\Vert f\Vert _{L_\infty}\Vert f^{-1}\Vert _{L_\infty}$. 
From Proposition \ref{Proposition S_eps f^eps} and \eqref{Lambda<=} it follows that 
$
\Vert \Lambda ^\varepsilon S_\varepsilon ^{(m)}\Vert _{L_2(\mathbb{R}^d)\rightarrow L_2(\mathbb{R}^d)}
\leqslant M_1$. 
(Here $S_\varepsilon ^{(m)}$ is the Steklov smoothing operator acting in $L_2(\mathbb{R}^d;\mathbb{C}^m)$.) 
Thus, 
by analogy with \eqref{proof fluxes 5} and \eqref{14.9}, from \eqref{14.12} we deduce estimate \eqref{flux with S_eps no hat} with the constant ${C}_{26}:={{C}}_{28}+\alpha _1 d^{1/2}M_1\Vert g\Vert _{L_\infty}\Vert f\Vert _{L_\infty}\Vert f^{-1}\Vert _{L_\infty}$.
\end{proof}

\begin{lemma}
\label{Lemma 14.7}
For $\varepsilon >0$ and $\tau\in\mathbb{R}$ we have
\begin{equation}
\label{Lm flux grubo}
\Vert g^\varepsilon b(\mathbf{D})f^\varepsilon\mathcal{A}_\varepsilon ^{-1/2}\sin (\tau\mathcal{A}_\varepsilon ^{1/2})(f^\varepsilon)^{-1}
-\widetilde{g}^\varepsilon b(\mathbf{D})\Pi _\varepsilon f_0 (\mathcal{A}^0)^{-1/2}\sin (\tau (\mathcal{A}^0)^{1/2})f_0^{-1}\Vert _{L_2(\mathbb{R}^d)\rightarrow L_2(\mathbb{R}^d)}
\leqslant
{C}_{29}.
\end{equation}
Here ${C}_{29}:=\left(\Vert g\Vert ^{1/2}_{L_\infty}+\Vert g\Vert _{L_\infty}\Vert g^{-1}\Vert_{L_\infty}^{1/2}(m^{1/2}\Vert g\Vert _{L_\infty}^{1/2}\Vert g^{-1}\Vert ^{1/2}_{L_\infty}+1)\right)\Vert f^{-1}\Vert _{L_\infty}$.
\end{lemma}

\begin{proof}
By \eqref{A_eps no hat},
\begin{equation}
\label{lm pr 1}
\Vert g^\varepsilon b(\mathbf{D})f^\varepsilon \mathcal{A}_\varepsilon ^{-1/2}\sin (\tau {\mathcal{A}}_\varepsilon ^{1/2})(f^\varepsilon)^{-1}\Vert _{L_2(\mathbb{R}^d)\rightarrow L_2(\mathbb{R}^d)}
\leqslant \Vert g\Vert ^{1/2}_{L_\infty}\Vert f^{-1}\Vert _{L_\infty}.
\end{equation}

Next, by \eqref{g^0<=},  \eqref{f0<=},  and \eqref{A0 no hat},  
\begin{equation}
\begin{split}
\Vert &\widetilde{g}^\varepsilon \Pi _\varepsilon ^{(m)}b(\mathbf{D})f_0 (\mathcal{A}^0)^{-1/2}\sin (\tau(\mathcal{A}^0)^{1/2})f_0^{-1}\Vert _{L_2(\mathbb{R}^d)\rightarrow L_2(\mathbb{R}^d)}
\\
&\leqslant
\Vert \widetilde{g}^\varepsilon \Pi _\varepsilon ^{(m)}\Vert _{L_2(\mathbb{R}^d)\rightarrow L_2(\mathbb{R}^d)}\Vert g^{-1}\Vert ^{1/2}_{L_\infty}\Vert f^{-1}\Vert _{L_\infty}.
\end{split}
\end{equation}
Using Proposition \ref{Proposition Pi eps f eps} and \eqref{tilde g},  \eqref{b(D)Lambda <=}, we obtain
\begin{equation}
\label{lm pr 3}
\Vert \widetilde{g}^\varepsilon \Pi _\varepsilon ^{(m)}\Vert _{L_2(\mathbb{R}^d)\rightarrow L_2(\mathbb{R}^d)}
\leqslant\Vert g\Vert _{L_\infty}(\vert\Omega\vert ^{-1/2}\Vert b(\mathbf{D})\Lambda\Vert _{L_2(\Omega)}+1)
\leqslant\Vert g\Vert _{L_\infty}(m^{1/2}\Vert g\Vert _{L_\infty}^{1/2}\Vert g^{-1}\Vert ^{1/2}_{L_\infty}+1).
\end{equation}

Combining \eqref{lm pr 1}--\eqref{lm pr 3}, we arrive at estimate \eqref{Lm flux grubo}. 
\end{proof}

\begin{theorem}
\label{Theorem 14.8}
$1^\circ$. 
Let $\mathbf{u}_\varepsilon$ and $\mathbf{u}_0$ be the solutions of problems \eqref{14.1} and \eqref{14.4}, respectively, for  $\boldsymbol{\psi}\in H^s(\mathbb{R}^d;\mathbb{C}^n)$ and $ \mathbf{F}\in L_{1,\mathrm{loc}}(\mathbb{R};H^s(\mathbb{R}^d;\mathbb{C}^n))$, where $0\leqslant s\leqslant 2$. Let $\mathbf{p}_\varepsilon$ be given by \eqref{P_eps ho hat} and let $\widetilde{g}(\mathbf{x})$ be the matrix-valued function \eqref{tilde g}. Then for $\tau\in\mathbb{R}$ and $\varepsilon >0$ we have
\begin{equation}
\label{10.24a}
\begin{split}
\Vert \mathbf{p}_\varepsilon (\cdot ,\tau)-\widetilde{g}^\varepsilon b(\mathbf{D})\Pi _\varepsilon \mathbf{u}_0(\cdot,\tau)\Vert _{L_2(\mathbb{R}^d)}
\leqslant \mathfrak{C}_4(s)(1+\vert\tau\vert)^{s/2}\varepsilon ^{s/2}\left(\Vert\boldsymbol{\psi}\Vert _{H^s(\mathbb{R}^d)}
+\Vert \mathbf{F}\Vert _{L_1((0,\tau);H^s(\mathbb{R}^d))}\right).
\end{split}
\end{equation}
Here $\mathfrak{C}_4(s):={C}_{29}^{1-s/2}{C}_{25}^{s/2}$. Under the additional assumption that $\mathbf{F}\in L_1(\mathbb{R}_\pm; H^{s}(\mathbb{R}^d;\mathbb{C}^n))$, where $0\leqslant s\leqslant 2$, for $\vert\tau\vert =\varepsilon^{-\alpha}$, $0<\varepsilon\leqslant 1$, $0<\alpha<1$, we have
\begin{equation}
\label{10.24b}
\begin{split}
\Vert &\mathbf{p}_\varepsilon (\cdot ,\pm\varepsilon^{-\alpha})-\widetilde{g}^\varepsilon b(\mathbf{D})\Pi _\varepsilon \mathbf{u}_0(\cdot,\pm\varepsilon^{-\alpha})\Vert _{L_2(\mathbb{R}^d)}
\\
&\leqslant 2^{s/2}\mathfrak{C}_4(s)\varepsilon ^{s(1-\alpha)/2}\left(\Vert\boldsymbol{\psi}\Vert _{H^s(\mathbb{R}^d)}+\Vert \mathbf{F}\Vert _{L_1(\mathbb{R}_\pm ;H^s(\mathbb{R}^d))}\right).
\end{split}
\end{equation}

\noindent
$2^\circ$. If $\boldsymbol{\psi}\in L_2(\mathbb{R}^d;\mathbb{C}^n)$ and $\mathbf{F}\in L_{1,\mathrm{loc}}(\mathbb{R};L_2(\mathbb{R}^d;\mathbb{C}^n))$, then
\begin{equation*}
\lim \limits _{\varepsilon\rightarrow 0}\Vert \mathbf{p}_\varepsilon(\cdot,\tau)-\widetilde{g}^\varepsilon b(\mathbf{D})\Pi_\varepsilon \mathbf{u}_0(\cdot,\tau)\Vert _{L_2(\mathbb{R}^d)}=0,\quad\tau\in\mathbb{R}.
\end{equation*}

\noindent
$3^\circ$. If $\boldsymbol{\psi}\in L_2(\mathbb{R}^d;\mathbb{C}^n)$ and $\mathbf{F}\in L_{1}(\mathbb{R}_\pm;L_2(\mathbb{R}^d;\mathbb{C}^n))$, then
\begin{equation*}
\lim _{\varepsilon \rightarrow 0}\Vert \mathbf{p}_\varepsilon (\cdot ,\pm \varepsilon ^{-\alpha})-\widetilde{g}^\varepsilon b(\mathbf{D})\Pi _\varepsilon\mathbf{u}_0(\cdot,\pm\varepsilon ^{-\alpha} )\Vert _{L_2(\mathbb{R}^d)}=0,
\quad 0<\varepsilon\leqslant 1,
\quad0<\alpha<1.
\end{equation*}
\end{theorem}

\begin{proof}
Rewriting estimate 
\eqref{flux with Pi_eps no hat} with $\mathbf{F}=0$ in operator terms and interpolating with estimate \eqref{Lm flux grubo}, we conclude that
\begin{equation*}
\begin{split}
\Vert& g^\varepsilon b(\mathbf{D})f^\varepsilon {\mathcal{A}}_\varepsilon ^{-1/2}\sin(\tau  {\mathcal{A}}_\varepsilon ^{1/2} )(f^\varepsilon)^{-1} -\widetilde{g}^\varepsilon b(\mathbf{D})\Pi _\varepsilon f_0({\mathcal{A}}^0)^{-1/2}\sin(\tau({\mathcal{A}}^0)^{1/2})f_0^{-1}\Vert _{H^s(\mathbb{R}^d)\rightarrow L_2(\mathbb{R}^d)}\\
&\leqslant C_{29}^{1-s/2}C_{25}^{s/2}(1+\vert \tau\vert )^{s/2}\varepsilon ^{s/2}.
\end{split}
\end{equation*}
Thus, by \eqref{14.3} and \eqref{14.5}, we derive estimate \eqref{10.24a}.

The assertion $2^\circ$ follows from \eqref{10.24a} by the Banach-Steinhaus theorem.

The result $3^\circ$ is a consequence of \eqref{10.24b} and  the Banach-Steinhaus theorem.
\end{proof}

\begin{remark}
\label{Remark flux no S-eps}
Using Proposition \textnormal{\ref{Proposition S_eps f^eps}}, 
it is easily seen that the results of Lemma~\textnormal{\ref{Lemma 14.7}}  are valid with the operator $\Pi_\varepsilon$ replaced by the  operator $S_\varepsilon$. Hence, by using \eqref{flux with S_eps no hat} and interpolation, we deduce the analog of Theorem \textnormal{\ref{Theorem 14.8}} with $\Pi _\varepsilon$ replaced by $S_\varepsilon$. This only changes the constants in estimates.
\end{remark}

\subsection{On the possibility to remove $\Pi _\varepsilon$ from approximation of the flux}

\begin{theorem}
\label{Theorem d<=4 fluxes}
Under the assumptions of Theorem~\textnormal{\ref{Theorem 14.6}}, let $d\leqslant 4$. Then for $\tau\in\mathbb{R}$ and $0<\varepsilon\leqslant 1$ we have
\begin{equation}
\label{Th fluxes d<=4}
\Vert \mathbf{p}_\varepsilon (\cdot,\tau)-\widetilde{g}^\varepsilon b(\mathbf{D}) \mathbf{u}_0(\cdot ,\tau)\Vert _{L_2(\mathbb{R}^d)}
\leqslant C_{30} \varepsilon (1+\vert\tau\vert)\left(\Vert \boldsymbol{\psi}\Vert _{H^2(\mathbb{R}^d)}+\Vert \mathbf{F}\Vert _{L_{1}((0,\tau);H^2(\mathbb{R}^d))}\right).
\end{equation}
The constant $C_{30}$ depends only on $m$, $n$, $d$, $\alpha _0$, $\alpha _1$, $\Vert g\Vert _{L_\infty}$, $\Vert g^{-1}\Vert _{L_\infty}$, $\Vert f\Vert _{L_\infty}$, $\Vert f^{-1}\Vert _{L_\infty}$, and the parameters of the lattice $\Gamma$.
\end{theorem}

\begin{proof}
The proof repeats the proof of Theorem~\textnormal{\ref{Theorem 14.6}} with some simplifications. By \eqref{9.20a NEW}, \eqref{14.3}, and \eqref{14.5}, 
\begin{equation}
\label{proof flux no Pi d<=4 0}
\Vert \widehat{\mathcal{A}}_\varepsilon ^{1/2}\left(\mathbf{u}_\varepsilon(\cdot ,\tau )
-(I+\varepsilon\Lambda ^\varepsilon b(\mathbf{D}))\mathbf{u}_0(\cdot ,\tau)\right)\Vert _{L_2(\mathbb{R}^d)}
\leqslant C_{14}\varepsilon (1+\vert\tau\vert)\left(\Vert \boldsymbol{\psi}\Vert _{H^2(\mathbb{R}^d)}+\Vert \mathbf{F}\Vert _{L_{1}((0,\tau);H^2(\mathbb{R}^d))}\right).
\end{equation}
Then, according to \eqref{A_eps} and \eqref{P_eps ho hat},
\begin{equation}
\label{proof flux no Pi d<=4 I}
\begin{split}
\Vert &\mathbf{p}_\varepsilon (\cdot ,\tau)-g^\varepsilon b(\mathbf{D})(I+\varepsilon \Lambda ^\varepsilon b(\mathbf{D}))\mathbf{u}_0(\cdot ,\tau)\Vert _{L_2(\mathbb{R}^d)}
\\
&\leqslant
\Vert g\Vert ^{1/2}_{L_\infty}C_{14}\varepsilon (1+\vert \tau\vert)\left(\Vert \boldsymbol{\psi}\Vert _{H^2(\mathbb{R}^d)}+\Vert \mathbf{F}\Vert _{L_{1}((0,\tau);H^2(\mathbb{R}^d))}\right).
\end{split}
\end{equation}
Similarly to \eqref{proof fluxes 5},
\begin{equation}
\label{proof flux no Pi d<=4 II}
\varepsilon g^\varepsilon b(\mathbf{D})\Lambda^\varepsilon b(\mathbf{D})\mathbf{u}_0(\cdot ,\tau)
=g^\varepsilon (b(\mathbf{D})\Lambda)^\varepsilon b(\mathbf{D})\mathbf{u}_0(\cdot ,\tau)
+\varepsilon g^\varepsilon \sum _{l=1}^d b_l \Lambda ^\varepsilon  D_lb(\mathbf{D})\mathbf{u}_0(\cdot ,\tau).
\end{equation}
Let us estimate the second summand in the right-hand side. By \eqref{b_j <=},
\begin{equation}
\label{proof flux no Pi d<=4 a}
\begin{split}
\Bigl\Vert &\varepsilon g^\varepsilon \sum _{l=1}^d b_l \Lambda ^\varepsilon  D_lb(\mathbf{D})\mathbf{u}_0(\cdot ,\tau)\Bigr\Vert _{L_2(\mathbb{R}^d)}
\leqslant \varepsilon\Vert g\Vert _{L_\infty}(d\alpha _1)^{1/2}\Vert \Lambda ^\varepsilon\mathbf{D}
b(\mathbf{D})\mathbf{u}_0(\cdot ,\tau)\Vert _{L_2(\mathbb{R}^d)}
\\
&\leqslant  \varepsilon\Vert g\Vert _{L_\infty}(d\alpha _1)^{1/2}
\Vert [\Lambda]\Vert _{H^1(\mathbb{R}^d)\rightarrow L_2(\mathbb{R}^d)}
\Vert \mathbf{D}
b(\mathbf{D})\mathbf{u}_0(\cdot ,\tau)\Vert _{H^1(\mathbb{R}^d)},\quad 0<\varepsilon\leqslant 1.
\end{split}
\end{equation} 
By \eqref{g^0<=}, \eqref{f0<=}, \eqref{A0 no hat}, and \eqref{14.5},
\begin{equation}
\label{proof flux no Pi d<=4 b}
\Vert \mathbf{D}
b(\mathbf{D})\mathbf{u}_0(\cdot ,\tau)\Vert _{H^1(\mathbb{R}^d)}\leqslant
\Vert g^{-1}\Vert ^{1/2}_{L_\infty}\Vert f^{-1}\Vert _{L_\infty}\left(\Vert \boldsymbol{\psi}\Vert _{H^2(\mathbb{R}^d)}+\Vert \mathbf{F}\Vert _{L_{1}((0,\tau);H^2(\mathbb{R}^d))}\right).
\end{equation}
Combining \eqref{Lambda H1->L2}, \eqref{proof flux no Pi d<=4 a}, and \eqref{proof flux no Pi d<=4 b}, we have
\begin{equation}
\label{proof flux no Pi d<=4 III}
\begin{split}
\Bigl\Vert \varepsilon g^\varepsilon \sum _{l=1}^d b_l \Lambda ^\varepsilon  D_lb(\mathbf{D})\mathbf{u}_0(\cdot ,\tau)\Bigr\Vert _{L_2(\mathbb{R}^d)}
&\leqslant 
\varepsilon\Vert g\Vert _{L_\infty}(d\alpha _1)^{1/2}\mathfrak{C}_d\Vert g^{-1}\Vert ^{1/2}_{L_\infty}\Vert f^{-1}\Vert _{L_\infty}
\\
&\times
\left(\Vert \boldsymbol{\psi}\Vert _{H^2(\mathbb{R}^d)}+\Vert \mathbf{F}\Vert _{L_{1}((0,\tau);H^2(\mathbb{R}^d))}\right),\quad d\leqslant 4.
\end{split}
\end{equation}
Now relations \eqref{tilde g}, \eqref{proof flux no Pi d<=4 I}, \eqref{proof flux no Pi d<=4 II}, and \eqref{proof flux no Pi d<=4 III} imply estimate \eqref{Th fluxes d<=4} with the constant
$$
C_{30}:=C_{14}\Vert g\Vert ^{1/2}_{L_\infty}+(d\alpha _1)^{1/2}\mathfrak{C}_d\Vert g\Vert _{L_\infty}\Vert g^{-1}\Vert ^{1/2}_{L_\infty}\Vert f^{-1}\Vert _{L_\infty} .$$
\end{proof}

Let $d\geqslant 5$ and let Condition~\ref{Condition Lambda multiplier} be satisfied. Then, by the scaling transformation, the analog of \eqref{9.20a NEW} (with the constant $C_{15}$ instead of $C_{14}$) follows from \eqref{Th d>4 Lambda multiplier}. We wish to remove $\Pi _\varepsilon$ from approximation for the flux similarly to \eqref{proof flux no Pi d<=4 0}--\eqref{proof flux no Pi d<=4 III}. According to \cite[Subsection~1.6, Proposition~1]{MaSh}, Condition~\ref{Condition Lambda multiplier} implies the boundedness of $[\Lambda]$ as an operator from $H^1(\mathbb{R}^d;\mathbb{C}^m)$ to $L_2(\mathbb{R}^d;\mathbb{C}^n)$ with the estimate 
$\Vert [\Lambda ]\Vert _{H^1(\mathbb{R}^d)\rightarrow L_2(\mathbb{R}^d)}\leqslant C \Vert [\Lambda ]\Vert _{H^2(\mathbb{R}^d)\rightarrow H^1(\mathbb{R}^d)}$.

The following statement can be checked by analogy with the proof of Theorem~\ref{Theorem d<=4 fluxes}.

\begin{theorem}
\label{Theorem d<=5 fluxes}
Let $d\geqslant 5$. Let Condition~\textnormal{\ref{Condition Lambda multiplier}} be satisfied. Then, under the assumptions of Theorem~\textnormal{\ref{Theorem 14.6}}, for $0<\varepsilon\leqslant 1$ and $\tau\in\mathbb{R}$ we have
\begin{equation*}
\Vert \mathbf{p}_\varepsilon (\cdot,\tau)-\widetilde{g}^\varepsilon b(\mathbf{D}) \mathbf{u}_0(\cdot ,\tau)\Vert _{L_2(\mathbb{R}^d)}
\leqslant C_{31} \varepsilon (1+\vert\tau\vert)\left(\Vert \boldsymbol{\psi}\Vert _{H^2(\mathbb{R}^d)}+\Vert \mathbf{F}\Vert _{L_{1}((0,\tau);H^2(\mathbb{R}^d))}\right).
\end{equation*}
The constant $C_{31}:=C_{15}\Vert g\Vert ^{1/2}_{L_\infty}+(d\alpha _1)^{1/2}\Vert g\Vert _{L_\infty}\Vert g^{-1}\Vert _{L_\infty}^{1/2}\Vert f^{-1}\Vert _{L_\infty}\Vert [\Lambda]\Vert _{H^1(\mathbb{R}^d)\rightarrow L_2(\mathbb{R}^d)}$ depends only on $d$, $m$, $n$, $\alpha _0$, $\alpha _1$, $\Vert g\Vert _{L_\infty}$, $\Vert g^{-1}\Vert _{L_\infty}$, $\Vert f\Vert _{L_\infty}$, $\Vert f^{-1}\Vert _{L_\infty}$, the parameters of the lattice $\Gamma$, and the norm $\Vert [\Lambda]\Vert _{H^2(\mathbb{R}^d)\rightarrow H^1(\mathbb{R}^d)}$.
\end{theorem}

By analogy with \eqref{proof flux no Pi d<=4 0}--\eqref{proof flux no Pi d<=4 III}, using Proposition~\ref{Proposition Lambda Hs to L2}, from Theorem~\ref{Theorem D>=5 smooth date sec 9} we derive the following result.

\begin{theorem}
Let $d\geqslant 5$. Let $\mathbf{u}_\varepsilon$ and $\mathbf{u}_0$ be the solutions of problems \eqref{14.1} and \eqref{14.4}, respectively, where $\boldsymbol{\psi}\in H^{d/2}(\mathbb{R}^d;\mathbb{C}^n)$ and $\mathbf{F}\in L_1((0,\tau);H^{d/2}(\mathbb{R}^d;\mathbb{C}^n))$. Let $\mathbf{p}_\varepsilon$ be defined by \eqref{P_eps ho hat} and let $\widetilde{g}$ be the matrix-valued function \eqref{tilde g}. Then for $0<\varepsilon\leqslant 1$ and $\tau\in\mathbb{R}$ we have
\begin{equation*}
\Vert \mathbf{p}_\varepsilon (\cdot,\tau)-\widetilde{g}^\varepsilon b(\mathbf{D}) \mathbf{u}_0(\cdot ,\tau)\Vert _{L_2(\mathbb{R}^d)}
\leqslant C_{32} \varepsilon (1+\vert\tau\vert)\left(\Vert \boldsymbol{\psi}\Vert _{H^{d/2}(\mathbb{R}^d)}+\Vert \mathbf{F}\Vert _{L_{1}((0,\tau);H^{d/2}(\mathbb{R}^d))}\right).
\end{equation*}
The constant $C_{32}$ depends only on $d$, $m$, $n$, $\alpha _0$, $\alpha _1$, $\Vert g\Vert _{L_\infty}$, $\Vert g^{-1}\Vert _{L_\infty}$, $\Vert f\Vert _{L_\infty}$, $\Vert f^{-1}\Vert _{L_\infty}$, and the parameters of the lattice $\Gamma$.
\end{theorem}

To obtain interpolational results without any smoothing operator, we need to prove the analog of Lemma~\ref{Lemma 14.7} without $\Pi _\varepsilon$. I.~e., we want to prove $(L_2\rightarrow L_2)$-boundedness of the operator
\begin{equation}
\label{g tilde dots i wan boundedness}
\widetilde{g}^\varepsilon b(\mathbf{D})f_0(\mathcal{A}^0)^{-1/2}\sin (\tau (\mathcal{A}^0)^{1/2})f_0^{-1}.
\end{equation}
The following property of $\widetilde{g}$ was obtained in \cite[Proposition 9.6]{Su_MMNP}. (The one dimensional case will be considered in Subsection~\ref{Subsection special case} below.)

\begin{proposition}
\label{Proposition tilde g multiplier}
Let $l>1$ for $d=2$, and $l=d/2$ for $d\geqslant 3$. The operator $[\widetilde{g}]$ is a continuous mapping of $H^l(\mathbb{R}^d;\mathbb{C}^m)$ to $L_2(\mathbb{R}^d;\mathbb{C}^m)$, and
\begin{equation*}
\Vert [\widetilde{g}]\Vert _{H^l(\mathbb{R}^d)\rightarrow L_2(\mathbb{R}^d)}\leqslant\mathfrak{C}_d'.
\end{equation*}
The constant $\mathfrak{C}_d'$ depends only $d$, $m$, $n$, $\alpha _0$, $\alpha _1$, $\Vert g\Vert _{L_\infty}$, $\Vert g^{-1}\Vert _{L_\infty}$, and the parameters of the lattice $\Gamma$\textnormal{;} for $d=2$ it depends also on $l$.
\end{proposition}

So, for $d\geqslant 2$, we can not expect the $(L_2\rightarrow L_2)$-boundedness of the operator~\eqref{g tilde dots i wan boundedness}. The $(H^2\rightarrow L_2)$-continuity of the operator \eqref{g tilde dots i wan boundedness} was used in Theorem~\ref{Theorem d<=4 fluxes} and, under Condition~\ref{Condition Lambda multiplier}, in Theorem~\ref{Theorem d<=5 fluxes}. (The $(H^2\rightarrow L_2)$-boundedness of $[\widetilde{g}]$ follows from  \cite[Subsection~1.3.2, Lemma~1]{MaSh}.) So, without any additional conditions on $\Lambda$, using Proposition~\ref{Proposition tilde g multiplier}, we can obtain some interpolational results only for $d\leqslant 3$.

By \eqref{g^0<=}, \eqref{f0<=}, \eqref{A0 no hat}, and Proposition~\ref{Proposition tilde g multiplier},
\begin{equation}
\label{tilde g dots Hl to L2}
\Vert \widetilde{g}^\varepsilon b(\mathbf{D})f_0(\mathcal{A}^0)^{-1/2}\sin (\tau (\mathcal{A}^0)^{1/2})f_0^{-1}\Vert _{H^l(\mathbb{R}^d)\rightarrow L_2(\mathbb{R}^d)}
\leqslant \mathfrak{C}_d'\Vert g^{-1}\Vert ^{1/2}_{L_\infty}\Vert f^{-1}\Vert _{L_\infty}.
\end{equation}
(Here $l$ is as in Proposition~\ref{Proposition tilde g multiplier}.)

Combining \eqref{lm pr 1} and \eqref{tilde g dots Hl to L2} and interpolating with \eqref{Th fluxes d<=4}, we obtain the following result.

\begin{theorem}
Let $2\leqslant d\leqslant 3$, and let $1<l<2$ for $d=2$ and $l=3/2$ for $d=3$. Let $0\leqslant s\leqslant 1$. 
Assume that  $\theta = l+(2-l)s$ for $d=2$ and $\theta =3/2+s/2$ for $d=3$. Let $\mathbf{u}_\varepsilon$ and $\mathbf{u}_0$ be the solutions of problems \eqref{14.1} and \eqref{14.4}, respectively, where 
$\boldsymbol{\psi}\in H^\theta (\mathbb{R}^d;\mathbb{C}^n)$ and $\mathbf{F}\in L_1((0,\tau);H^\theta (\mathbb{R}^d;\mathbb{C}^n))$. Let $\mathbf{p}_\varepsilon$ be the flux \eqref{P_eps ho hat} and let $\widetilde{g}$ be the matrix-valued function \eqref{tilde g}. Then for $0<\varepsilon\leqslant 1$ and $\tau\in\mathbb{R}$ we have
\begin{equation*}
\Vert \mathbf{p}_\varepsilon (\cdot ,\tau)-\widetilde{g}^\varepsilon b(\mathbf{D})\mathbf{u}_0(\cdot ,\tau)\Vert _{L_2(\mathbb{R}^d)}
\leqslant
\mathfrak{C}_5(s)\varepsilon ^s(1+\vert \tau\vert)^s\left(\Vert \boldsymbol{\psi}\Vert _{H^\theta (\mathbb{R}^d)}+\Vert \mathbf{F}\Vert _{L_1((0,\tau);H^\theta (\mathbb{R}^d))}\right).
\end{equation*}
Here $\mathfrak{C}_5(s):=C_{31}^s(\Vert g\Vert ^{1/2}_{L_\infty}+\mathfrak{C}_d'\Vert g^{-1}\Vert ^{1/2}_{L_\infty})^{1-s}\Vert f^{-1}\Vert ^{1-s}_{L_\infty}$.
\end{theorem}

\subsection{The special case} 

\label{Subsection special case}

Suppose that $g^0=\underline{g}$, i. e., relations \eqref{underline-g} hold. 
For $d=1$, identity $g^0=\underline{g}$ is always true, see, e. g., \cite[Chapter~I, \S 2]{ZhKO}. 
In accordance with \cite[Remark 3.5]{BSu05}, in this case the matrix-valued function \eqref{tilde g} is constant and coincides with $g^0$, i. e., $\widetilde{g}(\mathbf{x})=g^0=\underline{g}$. The following statement is a consequence of Theorem~\ref{Theorem 14.8}($1^\circ$.)

\begin{proposition}
Assume that relations \eqref{underline-g} hold. 
Let $\mathbf{u}_\varepsilon$ and $\mathbf{u}_0$ be the solutions of problems \eqref{14.1} and \eqref{14.4}, respectively, for  $\boldsymbol{\psi}\in H^s(\mathbb{R}^d;\mathbb{C}^n)$ and $ \mathbf{F}\in L_{1,\mathrm{loc}}(\mathbb{R};H^s(\mathbb{R}^d;\mathbb{C}^n))$, where $0\leqslant s\leqslant 2$. Let $\mathbf{p}_\varepsilon$ be given by \eqref{P_eps ho hat}. Then for $\tau\in\mathbb{R}$ and $\varepsilon >0$ we have
\begin{equation}
\label{fluxes in special case}
\Vert \mathbf{p}_\varepsilon (\cdot ,\tau)-g^0b(\mathbf{D})\mathbf{u}_0(\cdot ,\tau)\Vert _{L_2(\mathbb{R}^d)}
\leqslant \mathfrak{C}_6(s)(1+\vert \tau\vert )^{s/2}\varepsilon ^{s/2}
\left(\Vert\boldsymbol{\psi}\Vert _{H^s(\mathbb{R}^d)}
+\Vert \mathbf{F}\Vert _{L_1((0,\tau);H^s(\mathbb{R}^d))}\right).
\end{equation}
Here $\mathfrak{C}_6(s):=\mathfrak{C}_4(s)+2^{1-s/2}r_0^{-s/2}\Vert g\Vert ^{1/2}_{L_\infty}\Vert f^{-1}\Vert _{L_\infty}$.
\end{proposition}

\begin{proof}
We wish to remove the operator $\Pi _\varepsilon$ from the approximation \eqref{10.24a}. 
Obviously, \break$\Vert \Pi _\varepsilon -I\Vert _{L_2(\mathbb{R}^d)\rightarrow L_2(\mathbb{R}^d)}\leqslant 2$. According to Proposition~\ref{Proposition Pi eps -I}, 
$$\Vert \Pi _\varepsilon -I\Vert _{H^2(\mathbb{R}^d)\rightarrow L_2(\mathbb{R}^d)}\leqslant
\Vert \Pi _\varepsilon -I\Vert _{H^1(\mathbb{R}^d)\rightarrow L_2(\mathbb{R}^d)}\leqslant \varepsilon r_0^{-1}.$$ 
Then, by interpolation, 
$\Vert \Pi _\varepsilon -I\Vert _{H^s(\mathbb{R}^d)\rightarrow L_2(\mathbb{R}^d)}\leqslant 2^{1-s/2}r_0^{-s/2}\varepsilon ^{s/2}$, $0\leqslant s\leqslant 2$. Combining this with \eqref{g^0<=}, \eqref{f0<=}, \eqref{A0 no hat}, \eqref{14.5}, and taking into account that the operator $\mathcal{A}^0$ with constant coefficients commutes with the smoothing operator $\Pi _\varepsilon$, we obtain
\begin{equation}
\label{proof flux special case}
\begin{split}
\Vert &g^0 b(\mathbf{D})(\Pi _\varepsilon -I)\mathbf{u}_0(\cdot ,\tau )\Vert _{L_2(\mathbb{R}^d)}
\\
&\leqslant 2^{1-s/2}r_0^{-s/2}\Vert g\Vert ^{1/2}_{L_\infty}\Vert f^{-1}\Vert _{L_\infty}\varepsilon ^{s/2}\left(\Vert\boldsymbol{\psi}\Vert _{H^s(\mathbb{R}^d)}
+\Vert \mathbf{F}\Vert _{L_1((0,\tau);H^s(\mathbb{R}^d))}\right).
\end{split}
\end{equation}
Now, from identity $g^0=\widetilde{g}$, \eqref{10.24a}, and \eqref{proof flux special case} we derive estimate \eqref{fluxes in special case}.
\end{proof}

\section{Applications of the general results}

\label{Section Applications}

The following examples were previously considered in \cite{BSu,BSu08,DSu17,DSu17-2}.

\subsection{The acoustics equation}
\label{Subsection acoustics}

In $L_2(\mathbb{R}^d)$, we consider the operator 
\begin{equation}
\label{acoustics hat A}
\widehat{\mathcal{A}}=\mathbf{D}^*g(\mathbf{x})\mathbf{D}=-\mathrm{div}\,g(\mathbf{x})\nabla,
\end{equation}
where $g(\mathbf{x})$ is a periodic symmetric matrix with real entries. Assume that $g(\mathbf{x})>0$, $g,g^{-1}\in L_\infty$. The operator $\widehat{\mathcal{A}}$ describes a periodic acoustical medium. The operator \eqref{acoustics hat A} is a particular case of the operator \eqref{hat A}. Now we have $n=1$, $m=d$, $b(\mathbf{D})=\mathbf{D}$, $\alpha _0=\alpha _1=1$. Consider the operator $\widehat{\mathcal{A}}_\varepsilon =\mathbf{D}^* g^\varepsilon (\mathbf{x})\mathbf{D}$, whose coefficients oscillate rapidly for small $\varepsilon$.

Let us write down the effective operator. In the case under consideration, the $\Gamma$-periodic solution of problem \eqref{Lambda problem} is a row: $\Lambda (\mathbf{x})=i{\Phi}(\mathbf{x})$, ${\Phi}(\mathbf{x})=\left(\Phi _1(\mathbf{x}),\dots,\Phi _d(\mathbf{x})\right)$, where $\Phi _j\in \widetilde{H}^1(\Omega)$ is the solution of the problem
\begin{equation*}
\mathrm{div}\,g(\mathbf{x})\left(\nabla \Phi _j(\mathbf{x})+\mathbf{e}_j\right)=0,\quad \int _\Omega \Phi _j(\mathbf{x})\,d\mathbf{x}=0.
\end{equation*}
Here $\mathbf{e}_j$, $j=1,\dots,d$, is the standard orthonormal basis in $\mathbb{R}^d$. Clearly, the functions $\Phi _j(\mathbf{x})$ are real-valued, and the entries of $\Lambda (\mathbf{x})$ are purely imaginary. By \eqref{tilde g}, the columns of the $(d\times d)$-matrix-valued function $\widetilde{g}(\mathbf{x})$ are the vector-valued functions $g(\mathbf{x})\left(\nabla \Phi _j(\mathbf{x})+\mathbf{e}_j\right)$, $j=1,\dots,d$. The effective matrix is defined according to \eqref{g0}: $g^0=\vert \Omega \vert ^{-1}\int _\Omega \widetilde{g}(\mathbf{x})\,d\mathbf{x}$. Clearly, $\widetilde{g}(\mathbf{x})$ and $g^0$ have real entries. If $d=1$, then $m=n=1$, whence $g^0=\underline{g}$.

Let $Q(\mathbf{x})$ be a $\Gamma$-periodic function on $\mathbb{R}^d$ such that $Q(\mathbf{x})>0$, $Q,Q^{-1}\in L_\infty$. The function $Q(\mathbf{x})$ describes the density of the medium.

Consider the Cauchy problem for the acoustics equation in the medium with rapidly oscillating characteristics:
\begin{equation}
\label{acoustics problem}
\begin{cases}
Q^\varepsilon (\mathbf{x})\frac{\partial ^2 u_\varepsilon (\mathbf{x},\tau)}{\partial \tau ^2}=-\mathrm{div}\,g^\varepsilon (\mathbf{x})\nabla u_\varepsilon (\mathbf{x},\tau),\quad \mathbf{x}\in\mathbb{R}^d,\quad \tau\in\mathbb{R},
\\
u_\varepsilon (\mathbf{x},0)=0,\quad \frac{\partial u_\varepsilon (\mathbf{x},0)}{\partial \tau}=\psi (\mathbf{x}),
\end{cases}
\end{equation}
where $\psi\in L_2(\mathbb{R}^d)$ is a given function. (For simplicity, we consider the homogeneous equation.) Then the homogenized problem takes the form
\begin{equation}
\label{acoustics eff problem}
\begin{cases}
\overline{Q}\frac{\partial ^2 u_0 (\mathbf{x},\tau)}{\partial \tau ^2}=-\mathrm{div}\,g^0\nabla u_0(\mathbf{x},\tau),\quad \mathbf{x}\in\mathbb{R}^d,\quad \tau\in\mathbb{R},
\\
u_0 (\mathbf{x},0)=0,\quad \frac{\partial u_0 (\mathbf{x},0)}{\partial \tau}=\psi (\mathbf{x}).
\end{cases}
\end{equation}

According to \cite[Chapter~III, Theorem 13.1]{LaU}, $\Lambda\in L_\infty$ and the norm $\Vert\Lambda\Vert _{L_\infty}$ does not exceed a constant depending on $d$, $\Vert g\Vert _{L_\infty}$, $\Vert g^{-1}\Vert _{L_\infty}$, and $\Omega$. Applying Theorems~\ref{Theorem 14.2} and \ref{Theorem 14.8}($1^\circ$) and taking into account Remark~\ref{Remark no S-eps}, we arrive at the following result.

\begin{proposition}
Under the assumptions of Subsection~\textnormal{\ref{Subsection acoustics}}, let $u_\varepsilon$ be the solution of problem~\eqref{acoustics problem} and let $u_0$ be the solution of the effective problem \eqref{acoustics eff problem}.

\noindent
$1^\circ$. Let $\psi\in H^s(\mathbb{R}^d)$ for some $0\leqslant s\leqslant 1$. Then for $\tau\in\mathbb{R}$ and $\varepsilon >0$ we have
\begin{equation*}
\Vert u_\varepsilon (\cdot ,\tau)-u_0(\cdot ,\tau)\Vert _{L_2(\mathbb{R}^d)}\leqslant \mathfrak{C}_6(s)(1+\vert \tau\vert)\varepsilon ^s\Vert \psi\Vert _{H^s(\mathbb{R}^d)}.
\end{equation*}

\noindent
$2^\circ$. Let $\psi\in H^{s+1}(\mathbb{R}^d)$ for some $0\leqslant s\leqslant 1$. Then for $\tau \in\mathbb{R}$ and $0<\varepsilon\leqslant 1$ we have
\begin{equation*}
\Vert u_\varepsilon (\cdot ,\tau)-u_0(\cdot ,\tau)-\varepsilon {\Phi} ^\varepsilon \nabla u_0(\cdot ,\tau)\Vert _{H^1(\mathbb{R}^d)}
\leqslant \mathfrak{C}_7(s)(1+\vert \tau\vert )\varepsilon ^s\Vert \psi\Vert _{H^{1+s}(\mathbb{R}^d)}.
\end{equation*}

\noindent
$3^\circ$. Let $\psi \in H^s(\mathbb{R}^d)$ for some $0\leqslant s\leqslant 2$. Let $\Pi_\varepsilon$ be defined by \eqref{Pi eps}. Then for $\tau\in\mathbb{R}$ and $\varepsilon >0$ we have
\begin{equation*}
\Vert g^\varepsilon\nabla u_\varepsilon (\cdot ,\tau)-\widetilde{g}^\varepsilon\Pi _\varepsilon \nabla u_0(\cdot ,\tau)\Vert _{L_2(\mathbb{R}^d)}
\leqslant \mathfrak{C}_8(s)(1+\vert \tau\vert )^{s/2}\varepsilon ^{s/2}\Vert \psi\Vert _{H^s(\mathbb{R}^d)}.
\end{equation*}

The constants $\mathfrak{C}_6(s)$, $\mathfrak{C}_7(s)$, and $\mathfrak{C}_8(s)$ depend only on $s$, $d$, $\Vert g\Vert _{L_\infty}$, $\Vert g^{-1}\Vert _{L_\infty}$, and parameters of the lattice $\Gamma$.
\end{proposition}

\subsection{The operator of elasticity theory} 
Let $d\geqslant 2$. We represent the operator of elasticity theory in the form used in \cite[Chapter 5, \S 2]{BSu}. Let $\zeta$ be an orthogonal second rank tensor in $\mathbb{R}^d$; in the standard orthonormal basis in $\mathbb{R}^d$, it can be represented by a matrix $\zeta =\lbrace\zeta_{jl}\rbrace_{j,l=1}^d$. We shall consider symmetric tensors $\zeta$, which will be identified with vectors $\zeta _*\in\mathbb{C}^m$, $2m=d(d+1)$, by the following rule. The vector $\zeta _*$ is formed by all components $\zeta_{jl}$, $j\leqslant l$, and the pairs $(j,l)$ are put in order in some fixed way. Let $\chi$ be an $(m\times m)$-matrix, $\chi=\mathrm{diag}\,\lbrace\chi _{(j,l)}\rbrace$, where $\chi_{(j,l)}=1$ for $j=l$ and $\chi_{(j,l)}=2$ for $j<l$. Then ${\pmb |}\zeta{\pmb |}^2=\langle\chi\zeta _*,\zeta _*\rangle _{\mathbb{C}^m}$.

Let $\mathbf{u}\in H^1(\mathbb{R}^d;\mathbb{C}^d)$ be the \textit{displacement vector}. Then the \textit{deformation tensor} is given by $e(\mathbf{u})=\frac{1}{2}\left\lbrace\frac{\partial u_j}{\partial x_l}+\frac{\partial u_l}{\partial x_j}\right\rbrace$. The corresponding vector is denoted by $e_*(\mathbf{u})$. The relation $b(\mathbf{D})\mathbf{u}=-ie_*(\mathbf{u})$ determines an $(m\times d)$-matrix homogeneous DO $b(\mathbf{D})$ uniquely; the symbol of this DO is a matrix with real entries. For instance, with an appropriate ordering, we have
\begin{equation*}
b(\boldsymbol{\xi})=\begin{pmatrix}
\xi _1&0\\
\frac{\xi_2}{2}&\frac{\xi_1}{2}\\
0&\xi_2
\end{pmatrix},\quad d=2;\quad
b(\boldsymbol{\xi})=
\begin{pmatrix}
\xi_1&0&0\\
\frac{\xi_2}{2}&\frac{\xi_1}{2}&0\\
0&\xi_2&0\\
0&\frac{\xi_3}{2}&\frac{\xi_2}{2}\\
0&0&\xi_3\\
\frac{\xi_3}{2}&0&\frac{\xi_1}{2}
\end{pmatrix},\quad d=3.
\end{equation*}

Let $\sigma (\mathbf{u})$ be the \textit{stress tensor}, and let $\sigma_*(\mathbf{u})$ be the corresponding vector. The Hooke law can be represented by the relation $\sigma _*(\mathbf{u})=g(\mathbf{x})e_*(\mathbf{u})$, where $g(\mathbf{x})$ is an $(m\times m)$ matrix (which gives a ,,concise'' description of the Hooke tensor). This matrix characterizes the parameters of the elastic (in general, anisotropic) medium. We assume that $g(\mathbf{x})$ is $\Gamma$-periodic and such that $g(\mathbf{x})>0$, and $g,g^{-1}\in L_\infty$.

The energy of elastic deformations is given by the quadratic form
\begin{equation}
\label{form w}
\mathfrak{w}[\mathbf{u},\mathbf{u}]=\frac{1}{2}\int_{\mathbb{R}^d}\langle\sigma _*(\mathbf{u}),e_*(\mathbf{u})\rangle _{\mathbb{C}^m}\,d\mathbf{x}
=
\frac{1}{2}\int _{\mathbb{R}^d}\langle g(\mathbf{x})b(\mathbf{D})\mathbf{u},b(\mathbf{D})\mathbf{u}\rangle _{\mathbb{C}^m}\,d\mathbf{x},\quad
\mathbf{u}\in H^1(\mathbb{R}^d;\mathbb{C}^d).
\end{equation}
The operator $\mathcal{W}$ generated by this form is the operator of elasticity theory. Thus, the operator $2\mathcal{W}=b(\mathbf{D})^*g(\mathbf{x})b(\mathbf{D})=\widehat{\mathcal{A}}$ is of the form \eqref{hat A} with $n=d$ and $m=d(d+1)/2$.

In the case of an \textit{isotropic} medium, the expression for the form \eqref{form w} simplifies significantly and depends only on two functional \textit{Lam\'e parameters} $\lambda(\mathbf{x})$, $\mu(\mathbf{x})$:
\begin{equation*}
\mathfrak{w}[\mathbf{u},\mathbf{u}]=\int_{\mathbb{R}^d}\left(\mu(\mathbf{x}){\pmb |}e(\mathbf{u}){\pmb |}^2+\frac{\lambda(\mathbf{x})}{2}\vert \mathrm{div}\,\mathbf{u}\vert ^2\right)\,d\mathbf{x}.
\end{equation*}
The parameter $\mu$ is the \textit{shear modulus}. The modulus $\lambda (\mathbf{x})$ may be negative. Often, another parameter $\kappa (\mathbf{x})=\lambda (\mathbf{x})+2\mu(\mathbf{x})/d$ is introduced instead of $\lambda(\mathbf{x})$; $\kappa$ is called the \textit{modulus of volume compression}. In the isotropic case, the conditions that ensure the positive definiteness of the matrix  $g(\mathbf{x})$ are $\mu(\mathbf{x})\geqslant \mu _0>0$, $\kappa(\mathbf{x})\geqslant\kappa _0>0$. We write down the ,,isotropic'' matrices $g$ for $d=2$ and $d=3$:
\begin{align*}
g&=\begin{pmatrix}
\kappa+\mu&0&\kappa-\mu\\
0&4\mu&0\\
\kappa-\mu&0&\kappa+\mu
\end{pmatrix},\quad d=2;\\
g&=\frac{1}{3}\begin{pmatrix}
3\kappa+4\mu&0&3\kappa-2\mu&0&3\kappa-2\mu &0\\
0& 12\mu &0&0&0&0&\\
3\kappa -2\mu&0&3\kappa+4\mu&0&3\kappa-2\mu&0\\
0&0&0&12\mu&0&0\\
3\kappa -2\mu&0&3\kappa -2\mu&0&3\kappa+4\mu&0\\
0&0&0&0&0&12\mu
\end{pmatrix},\quad d=3.
\end{align*}

Consider the operator $\mathcal{W}_\varepsilon =\frac{1}{2}\widehat{\mathcal{A}}_\varepsilon$ with rapidly oscillating coefficients. The effective matrix $g^0$ and the effective operator $\mathcal{W}^0=\frac{1}{2}\widehat{\mathcal{A}}^0$ are defined by the general rules (see  \eqref{tilde g}, \eqref{g0}, and \eqref{A^0 hat}).

Let $Q(\mathbf{x})$ be a $\Gamma$-periodic $(d\times d)$-matrix-valued function such that $Q(\mathbf{x})>0$, $Q,Q^{-1}\in L_\infty$. Usually, $Q(\mathbf{x})$ is a scalar-valued function describing the density of the medium. We assume that $Q(\mathbf{x})$ is a matrix-valued function in order to take possible anisotropy into account.

Consider the following Cauchy problem for the system of elasticity theory:
\begin{equation}
\label{elasticity problem}
\begin{cases}
Q^\varepsilon (\mathbf{x})\frac{\partial ^2\mathbf{u}_\varepsilon (\mathbf{x},\tau)}{\partial \tau ^2}=-\mathcal{W}_\varepsilon\mathbf{u}_\varepsilon (\mathbf{x},\tau),\quad\mathbf{x}\in\mathbb{R}^d,\quad\tau\in\mathbb{R},
\\
\mathbf{u}_\varepsilon (\mathbf{x},0)=0,\quad\frac{\partial\mathbf{u}_\varepsilon (\mathbf{x},0)}{\partial\tau}=\boldsymbol{\psi}(\mathbf{x}),
\end{cases}
\end{equation}
where $\boldsymbol{\psi}\in L_2(\mathbb{R}^d;\mathbb{C}^d)$ is a given function. The homogenized problem takes the form
\begin{equation*}
\begin{cases}
\overline{Q}\frac{\partial ^2 \mathbf{u}_0(\mathbf{x},\tau)}{\partial \tau ^2}=-\mathcal{W}^0\mathbf{u}_0 (\mathbf{x},\tau),\quad\mathbf{x}\in\mathbb{R}^d,\quad\tau\in\mathbb{R},
\\
\mathbf{u}_0 (\mathbf{x},0)=0,\quad\frac{\partial\mathbf{u}_0 (\mathbf{x},0)}{\partial\tau}=\boldsymbol{\psi}(\mathbf{x}).
\end{cases}
\end{equation*}
Theorems \ref{Theorem 14.2} and \ref{Theorem 14.8} can be applied to problem \eqref{elasticity problem}. If $d=2$, then Condition~\ref{Condition Lambda in L infty} is satisfied according to Proposition~\ref{Proposition Lambda in L infty <=}. So, we can use Theorem~\ref{Theorem 12.6}. If $d=3$, then Theorem~\ref{Theorem d<=4 chapter 3} is applicable.

\subsection{The model equation of electrodynamics}
\label{Subsection The model equation of electrodynamics}

We cannot include the general Maxwell operator in the scheme developed above; we have to assume that the magnetic permeability is unit. In $L_2(\mathbb{R}^3;\mathbb{C}^3)$, we consider the model operator $\mathcal{L}$ formally given  by the expression $\mathcal{L}=\mathrm{curl}\,\eta(\mathbf{x})^{-1}\mathrm{curl}-\nabla \nu (\mathbf{x})\mathrm{div}$. Here the \textit{dielectric permittivity} $\eta(\mathbf{x})$ is $\Gamma$-periodic $(3\times 3)$-matrix-valued function in $\mathbb{R}^3$ with real entries such that $\eta (\mathbf{x})>0$; $\eta,\eta^{-1}\in L_\infty$; $\nu(\mathbf{x})$ is real-valued $\Gamma$-periodic function in $\mathbb{R}^3$ such that $\nu(\mathbf{x})>0$; $\nu,\nu ^{-1}\in L_\infty$. The precise definition of $\mathcal{L}$ is given via the closed positive form
\begin{equation*}
\int _{\mathbb{R}^3}\left(
\langle \eta (\mathbf{x})^{-1}\mathrm{curl}\,\mathbf{u},\mathrm{curl}\,\mathbf{u}\rangle
+\nu (\mathbf{x})\vert \mathrm{div}\,\mathbf{u}\vert ^2\right)\,d\mathbf{x},\quad\mathbf{u}\in H^1(\mathbb{R}^3;\mathbb{C}^3).
\end{equation*}
The operator $\mathcal{L}$ can be written as $\widehat{\mathcal{A}}=b(\mathbf{D})^*g(\mathbf{x})b(\mathbf{D})$ with $n=3$, $m=4$, and
\begin{equation}
\label{b(D)=,g= Maxwell}
b(\mathbf{D})=\begin{pmatrix}
-i\mathrm{curl}\\
-i\mathrm{div}
\end{pmatrix},
\quad
g(\mathbf{x})=\begin{pmatrix}
\eta (\mathbf{x})^{-1}&0\\
0&\nu (\mathbf{x})
\end{pmatrix}.
\end{equation}
The corresponding symbol of $b(\mathbf{D})$ is
\begin{equation*}
b(\boldsymbol{\xi})=
\begin{pmatrix}
0&-\xi _3&\xi _2\\
\xi _3&0&-\xi _1\\
-\xi _2&\xi _1&0\\
\xi _1&\xi _2&\xi _3
\end{pmatrix}.
\end{equation*}

According to \cite[\S 7.2]{BSu} the effective matrix has the form
\begin{equation*}
g^0=\begin{pmatrix}
(\eta ^0)^{-1}&0\\
0&\underline{\nu}
\end{pmatrix},
\end{equation*}
where $\eta ^0$ is the effective matrix for the scalar elliptic operator $-\mathrm{div}\,\eta\nabla =\mathbf{D}^*\eta\mathbf{D}$. The effective operator is given by
\begin{equation*}
\mathcal{L}^0=\mathrm{curl}\,(\eta ^0)^{-1}\mathrm{curl}-\nabla \underline{\nu}\mathrm{div}.
\end{equation*}

Let $\mathbf{v}_j\in \widetilde{H}^1(\Omega ;\mathbb{C}^3)$ be the $\Gamma$-periodic solution of the problem
\begin{equation*}
b(\mathbf{D})^*g(\mathbf{x})\left(b(\mathbf{D})\mathbf{v}_j(\mathbf{x})+\mathbf{e}_j\right)=0,\quad \int _\Omega \mathbf{v}_j(\mathbf{x})\,d\mathbf{x}=0,
\end{equation*}
$j=1,2,3,4$. Here $\mathbf{e}_j$, $j=1,2,3,4$, is the standard orthonormal basis in $\mathbb{C}^4$. As was shown in \cite[\S 14]{BSu05}, the solutions $\mathbf{v}_j$, $j=1,2,3$, can be determined as follows. Let $\widetilde{\Phi}_j(\mathbf{x})$ be the $\Gamma$-periodic solution of the problem
\begin{equation*}
\mathrm{div}\,\eta (\mathbf{x})\left(\nabla \widetilde{\Phi}_j(\mathbf{x})+\mathbf{c}_j\right)=0,\quad \int _\Omega \widetilde{\Phi}_j(\mathbf{x})\,d\mathbf{x}=0,
\end{equation*}
$j=1,2,3$, where $\mathbf{c}_j=(\eta ^0)^{-1}\widetilde{\mathbf{e}}_j$, and $\widetilde{\mathbf{e}}_j$, $j=1,2,3$, is the standard orthonormal basis in $\mathbb{C}^3$. Let $\mathbf{q}_j$ be the $\Gamma$-periodic solution of the problem
\begin{equation*}
\Delta\mathbf{q}_j=\eta\left(\nabla\widetilde{\Phi}_j+\mathbf{c}_j\right)-\widetilde{\mathbf{e}}_j,\quad\int _\Omega \mathbf{q}_j(\mathbf{x})\,d\mathbf{x}=0.
\end{equation*}
Then $\mathbf{v}_j=i\mathrm{curl}\,\mathbf{q}_j$, $j=1,2,3$.

Next, we have $\mathbf{v}_4=i\nabla \phi$, where $\phi$ is the $\Gamma$-periodic solution of the problem
\begin{equation*}
\Delta\phi=\underline{\nu}\left(\nu (\mathbf{x})\right)^{-1}-1,\quad \int _\Omega \phi (\mathbf{x})\,d\mathbf{x}=0.
\end{equation*}
The matrix $\Lambda (\mathbf{x})$ is the $(3\times 4)$-matrix with the columns $i\mathrm{curl}\,\mathbf{q}_1$, $i\mathrm{curl}\,\mathbf{q}_2$, $i\mathrm{curl}\,\mathbf{q}_3$, $i\nabla \phi$. By $\Psi (\mathbf{x})$ we denote the $(3\times 3)$-matrix-valued function with the columns $\mathrm{curl}\,\mathbf{q}_1$, $\mathrm{curl}\,\mathbf{q}_2$, $\mathrm{curl}\,\mathbf{q}_3$ (then $\Psi (\mathbf{x})$ has real entries). We put $\mathbf{w}=\nabla\phi$. Then
\begin{equation*}
\Lambda (\mathbf{x})b(\mathbf{D})=\Psi (\mathbf{x})\mathrm{curl} + \mathbf{w}(\mathbf{x})\mathrm{div}.
\end{equation*}

The application of Theorems~\ref{Theorem 12.1} and \ref{Theorem d<=4 chapter 3} gives the following result.

\begin{theorem}
\label{Theorem Maxwell}
Under the assumptions of Subsection~\textnormal{\ref{Subsection The model equation of electrodynamics}}, denote $$\mathcal{L}_\varepsilon :=\mathrm{curl}\,(\eta ^\varepsilon)^{-1}\mathrm{curl}-\nabla \nu ^\varepsilon\mathrm{div}.$$ Then for $\tau\in\mathbb{R}$ we have
\begin{align}
\label{Th Maxwell 1}
\Vert &\mathcal{L}_\varepsilon ^{-1/2}\sin (\tau\mathcal{L}_\varepsilon ^{1/2})-(\mathcal{L}^0)^{-1/2}\sin(\tau (\mathcal{L}^0)^{1/2})\Vert _{H^1(\mathbb{R}^3)\rightarrow L_2(\mathbb{R}^3)}\leqslant C_{12}\varepsilon (1+\vert \tau\vert),\quad\varepsilon >0,
\\
\label{Th Maxwell 2}
\begin{split}
\Vert  &\mathcal{L}_\varepsilon ^{-1/2}\sin (\tau\mathcal{L}_\varepsilon ^{1/2})
-(I+\varepsilon\Psi ^\varepsilon\mathrm{curl}+\varepsilon\mathbf{w}^\varepsilon\mathrm{div} )(\mathcal{L}^0)^{-1/2}\sin(\tau (\mathcal{L}^0)^{1/2})\Vert _{H^2(\mathbb{R}^3)\rightarrow H^1(\mathbb{R}^3)}
\\
&\leqslant C_{19}\varepsilon (1+\vert\tau\vert),\quad 0<\varepsilon\leqslant 1.
\end{split}
\end{align}
The constants $C_{12}$ and $C_{19}$ depend only on $\Vert \eta\Vert _{L_\infty}$, $\Vert \eta ^{-1}\Vert _{L_\infty}$, $\Vert \nu\Vert _{L_\infty}$, $\Vert \nu ^{-1}\Vert _{L_\infty}$, and the parameters of the lattice $\Gamma$.
\end{theorem}

Also, we can apply (interpolational) Theorems~\ref{Theorem 12.2} and \ref{Theorem 12.4}. But in this case the correction term contains the smoothing operator $\Pi _\varepsilon$ (see \eqref{Pi eps}). We omit the details.

It turns out that the operators $\mathcal{L}_\varepsilon$ and $\mathcal{L}^0$ split in the Weyl decomposition $L_2(\mathbb{R}^3;\mathbb{C}^3)=J\oplus G$ simultaneously. Here the ,,solenoidal'' subspace $J$ consists of vector functions $\mathbf{u}\in L_2(\mathbb{R}^3;\mathbb{C}^3)$ for which $\mathrm{div}\,\mathbf{u}=0$ (in the sense of distributions) and the ,,potential'' subspace is 
$$G:=\lbrace\mathbf{u}=\nabla\phi: \phi\in H^1_{\mathrm{loc}}(\mathbb{R}^3), \nabla \phi\in L_2(\mathbb{R}^3;\mathbb{C}^3)\rbrace.$$ The Weyl decomposition reduces the operators $\mathcal{L}_\varepsilon$ and $\mathcal{L}^0$, i.~e., $\mathcal{L}_\varepsilon =\mathcal{L}_{\varepsilon,J}\oplus \mathcal{L}_{\varepsilon,G}$ and $\mathcal{L}^0=\mathcal{L}^0_J\oplus \mathcal{L}^0_G$. The part $\mathcal{L}_{\varepsilon ,J}$ acting in the ,,solenoidal'' subspace $J$ is formally defined by the differential expression $\mathrm{curl}\,\eta ^\varepsilon(\mathbf{x})^{-1}\mathrm{curl}$, while the part $\mathcal{L}_{\varepsilon,G}$ acting in the ,,potential'' subspace $G$ corresponds to the expression $-\nabla \nu ^\varepsilon (\mathbf{x})\nabla$. The parts $\mathcal{L}_J^0$ and $\mathcal{L}^0_G$ can be written in the same way. The Weyl decomposition allows us to apply Theorem~\ref{Theorem Maxwell} to homogenization 
of the Cauchy problem for the model hyperbolic equation appearing in electrodynamics:
\begin{equation}
\label{maxwell problem}
\begin{cases}
\partial _\tau ^2\mathbf{u}_\varepsilon =-\mathrm{curl}\,\eta ^\varepsilon (\mathbf{x})^{-1}\mathrm{curl}\,\mathbf{u}_\varepsilon,\quad\mathrm{div}\,\mathbf{u}_\varepsilon =0,\\
\mathbf{u}_\varepsilon (\mathbf{x},0)=0,\quad \partial _\tau\mathbf{u}_\varepsilon (\mathbf{x},0)=\boldsymbol{\psi}(\mathbf{x}).
\end{cases}
\end{equation}
The effective problem takes the form
\begin{equation}
\label{maxwell eff problem}
\begin{cases}
\partial _\tau ^2\mathbf{u}_0 =-\mathrm{curl}\,(\eta ^0)^{-1}\mathrm{curl}\,\mathbf{u}_0,\quad\mathrm{div}\,\mathbf{u}_0 =0,\\
\mathbf{u}_0 (\mathbf{x},0)=0,\quad \partial _\tau\mathbf{u}_0 (\mathbf{x},0)=\boldsymbol{\psi}(\mathbf{x}).
\end{cases}
\end{equation}

Let $\mathcal{P}$ be the orthogonal projection of $L_2(\mathbb{R}^3;\mathbb{C}^3)$ onto $J$. Then (see \cite[Subsection 2.4 of Chapter 7]{BSu}) the operator $\mathcal{P}$ (restricted to $H^s(\mathbb{R}^3;\mathbb{C}^3)$) is also the orthogonal projection of the space $H^s(\mathbb{R}^3;\mathbb{C}^3)$ onto the subspace $J\cap H^s(\mathbb{R}^3;\mathbb{C}^3)$ for all $s>0$.

Restricting the operators under the norm sign in \eqref{Th Maxwell 1} and \eqref{Th Maxwell 2} 
to the subspaces \break$J\cap H^1(\mathbb{R}^3;\mathbb{C}^3)$ and   $J\cap H^2(\mathbb{R}^3;\mathbb{C}^3)$, respectively,  and multiplying by $\mathcal{P}$ from the left, we see that Theorem~\ref{Theorem Maxwell} implies the following result.

\begin{theorem}
\label{Theorem Maxwell solutions}
Under the assumptions of Subsection~\textnormal{\ref{Subsection The model equation of electrodynamics}}, let $\mathbf{u}_\varepsilon$ and $\mathbf{u}_0$ be the solutions of problems \eqref{maxwell problem} and \eqref{maxwell eff problem}, respectively.

\noindent
$1^\circ$. Let $\boldsymbol{\psi}\in J\cap H^1(\mathbb{R}^3;\mathbb{C}^3)$. Then for $\varepsilon >0$ and $\tau\in\mathbb{R}$ we have
\begin{equation*}
\Vert \mathbf{u}_\varepsilon (\cdot ,\tau)-\mathbf{u}_0(\cdot ,\tau)\Vert _{L_2(\mathbb{R}^3)}
\leqslant C_{12}\varepsilon (1+\vert \tau\vert)\Vert \boldsymbol{\psi}\Vert _{H^1(\mathbb{R}^3)}.
\end{equation*}

\noindent 
$2^\circ$. Let $\boldsymbol{\psi}\in J\cap H^2(\mathbb{R}^3;\mathbb{C}^3)$. Then for $0<\varepsilon\leqslant 1$ and $\tau\in\mathbb{R}$ we have
\begin{equation*}
\Vert \mathbf{u}_\varepsilon (\cdot ,\tau)-\mathbf{u}_0(\cdot ,\tau)-\varepsilon \Psi ^\varepsilon\mathrm{curl}\,\mathbf{u}_0(\cdot ,\tau)\Vert _{H^1(\mathbb{R}^3)}
\leqslant C_{19}\varepsilon (1+\vert \tau\vert)\Vert \boldsymbol{\psi}\Vert _{H^2(\mathbb{R}^3)}.
\end{equation*}
\end{theorem}

According to \eqref{b(D)=,g= Maxwell}, the role of the flux for problem \eqref{maxwell problem} is played by the vector-valued function
\begin{equation*}
\mathbf{p}_\varepsilon = g^\varepsilon b(\mathbf{D})\mathbf{u}_\varepsilon
=-i\begin{pmatrix}
(\eta ^\varepsilon )^{-1}\mathrm{curl}\,\mathbf{u}_\varepsilon \\
\nu ^\varepsilon \mathrm{div}\,\mathbf{u}_\varepsilon
\end{pmatrix}
=
-i\begin{pmatrix}
(\eta ^\varepsilon )^{-1}\mathrm{curl}\,\mathbf{u}_\varepsilon \\
0
\end{pmatrix}.
\end{equation*}
To approximate the flux, we apply Theorem~\ref{Theorem d<=4 fluxes}. The matrix $\widetilde{g}=g(\mathbf{1}+b(\mathbf{D})\Lambda)$ has a block-diagonal structure, see \cite[Subsection 14.3]{BSu05}): the upper left $(3\times 3)$ block is represented by the matrix with the columns $\nabla \widetilde{\Phi}_j(\mathbf{x})+\mathbf{c}_j$, $j=1,2,3$. We denote this block by $a(\mathbf{x})$. The element at the right lower corner is equal to $\underline{\nu}$. The other elements are zero. Then, by \eqref{b(D)=,g= Maxwell} and \eqref{maxwell eff problem},
\begin{equation*}
\widetilde{g}^\varepsilon b(\mathbf{D})\mathbf{u}_0=-i\begin{pmatrix}
a^\varepsilon\mathrm{curl}\,\mathbf{u}_0\\0
\end{pmatrix}.
\end{equation*}

We arrive at the following statement.

\begin{theorem}
Under the assumptions of Theorem~\textnormal{\ref{Theorem Maxwell solutions}}, let $\boldsymbol{\psi}\in J\cap H^2(\mathbb{R}^3;\mathbb{C}^3)$. Then for $0<\varepsilon\leqslant 1$ and $\tau \in\mathbb{R}$ we have
\begin{equation*}
\Vert (\eta ^\varepsilon)^{-1}\mathrm{curl}\,\mathbf{u}_\varepsilon (\cdot ,\tau)-a^\varepsilon\mathrm{curl}\,\mathbf{u}_0(\cdot ,\tau)\Vert _{L_2(\mathbb{R}^3)}
\leqslant C_{30}\varepsilon (1+\vert\tau\vert)\Vert \boldsymbol{\psi}\Vert _{H^2(\mathbb{R}^3)}.
\end{equation*}
The constant $C_{30}$ depends only on $\Vert \eta\Vert _{L_\infty}$,  $\Vert \eta^{-1}\Vert _{L_\infty}$, $\Vert \nu\Vert _{L_\infty}$, $\Vert \nu ^{-1}\Vert _{L_\infty}$, and the parameters of the lattice $\Gamma$.
\end{theorem}

\end{document}